\documentclass{daj}

\usepackage[T1]{fontenc}
\usepackage[utf8]{inputenc}
\usepackage[english]{babel}
\usepackage{amssymb,amsmath,amsthm}
\usepackage{hyperref,enumerate}
\usepackage[all]{xy}
\usepackage{todonotes}

\usepackage[all]{xy}

\theoremstyle{plain}
\newtheorem{theo}{Theorem}[section]
\newtheorem{lem}[theo]{Lemma}
\newtheorem{cor}[theo]{Corollary}
\newtheorem{prop}[theo]{Proposition}
\theoremstyle{definition}
\newtheorem{defi}[theo]{Definition}
\newtheorem{exem}[theo]{Example}
\newtheorem{Rem}[theo]{Remark}

\dajAUTHORdetails{%
  title = {Decidability of Isomorphism and Factorization between Minimal Substitution Subshifts}, 
  author = {Fabien Durand and Julien Leroy},
  plaintextauthor = {Fabien Durand and Julien Leroy},
  runningtitle = {Decidability of Isomorphism and Factorization between Minimal Substitution Subshifts}, 
  runningauthor = {Fabien Durand and Julien Leroy},
  copyrightauthor = {F. Durand and J. Leroy},
  keywords = {subshift, substitution, automorphism, factorization, factor map, recognizability, decidability},
}   

\dajEDITORdetails{%
   year={2022},
   number={7},
   received={22 February 2021},   
   published={15 August 2022},  
   doi={10.19086/da.36901},       
}   

\def\Z{{\mathbb Z}}

\def\N{{\mathbb N}}

\def\K{{\mathbb K}}

\def\cD{{\mathcal D}}

\def\cL{{\mathcal L}}

\def\cR{{\mathcal R}}

\def \id {{\rm id}}

\def\card{\mathrm{Card}}

\begin{document}

\begin{frontmatter}[classification=text]

\title{Decidability of Isomorphism and Factorization between Minimal Substitution Subshifts}

\author[fabd]{Fabien Durand\thanks{Supported by ANR project “Dyna3S” and by Hubert Curien Partnership “Tournesol” 34212UG}}
\author[jull]{Julien Leroy\thanks{Supported by Hubert Curien Partnership “Tournesol” 34212UG}}

\begin{abstract}
Classification is a central problem for dynamical systems, in particular for families that arise in a wide range of topics, like substitution subshifts. 
It is important to be able to distinguish whether two such subshifts are isomorphic, but the existing invariants are not sufficient for this purpose.  
We first show that given two minimal substitution subshifts, there exists a computable constant $R$ such that any factor map between these subshifts (if any) is the composition of a factor map with a radius smaller than $R$ and some power of the shift map.
Then we prove that it is decidable to check whether a given sliding block code is a factor map between two prescribed minimal substitution subshifts.
As a consequence of these two results, we provide an algorithm that, given two minimal substitution subshifts, decides whether one is a factor of the other and, as a straightforward corollary, whether they are isomorphic.
\end{abstract}
\end{frontmatter}


\section{Introduction}

Classification is a central problem in the study of dynamical systems, in particular for families of systems that arise in a wide range of topics.
Let us mention subshifts of finite type that appear, for example, in information theory, hyperbolic dynamics, $C^*$-algebra, statistical mechanics and thermodynamic formalism~\cite{Lind&Marcus:1995,Bowen:1975}.
The most important and longstanding open problem for this family originates in~\cite{Williams:1973} and is stated in~\cite{Boyle:2008} as follows: {\em Classify subshifts of finite type up to topological isomorphism. 
In particular, give a procedure which decides when two non-negative integer matrices define topologically conjugate subshifts of finite type.}

The existence of an isomorphism between two given subshifts of finite type is known to be equivalent to the {\em Strong Shift Equivalence} for matrices over $\Z^+$~\cite{Williams:1973}, which is not known to be decidable~\cite{Kim&Roush:1999}. 
In contrast, the {\em Shift Equivalence} for matrices over $\Z^+$ is decidable and provides an invariant of isomorphism for subshifts of finite type. 

Another well-known family of subshifts, also defined through matrices and interesting for a wide range of reasons, is the family of substitution subshifts. 
These subshifts are related, for example, to automata theory, first order logic, combinatorics on words, quasicrystallography, fractal geometry, group theory and number theory~\cite{Adamczewski&Bugeaud:2007,Allouche&Shallit:2003,Pytheas,Nekrashevych:2018,Rigo:2014b}. 
In this paper we show that not only is the existence of isomorphism between such subshifts decidable, but also the factorization. 
This answers a question asked in~\cite{Durand:2013b}.
Observe that substitution subshifts form a subclass of subshifts generated by morphic sequences (see Section~\ref{subsection:morphicsubstitutive} for the definitions).

\begin{theo}
\label{theo:cordecidfactor}
Let $(X , S)$ and $(Y , S)$ be subshifts generated by uniformly recurrent morphic sequences.
It is decidable whether or not there is a factor map $f:(X , S) \to (Y,S)$.
\end{theo}

Using the fact that minimal substitution subshifts are coalescent~\cite{Durand:2000} one deduces the decidability of the isomorphism problem.
\begin{cor}
\label{cor:cordecidfactor}
It is decidable whether or not two subshifts generated by uniformly recurrent morphic sequences are isomorphic.
\end{cor}

Our approach gives a precise and detailed description of the isomorphisms and factor maps. 
It is thus also concerned with the study of automorphism groups of topological dynamical systems and its recent developments.
This was a classical topic in the 70's and 80's~\cite{Hedlund:1969,Coven:1971,Boyle&Lind&Rudolph:1988} and it recently got a renewal of interest for subshifts with low complexity.
V.~Cyr and B.~Kra~\cite{Cyr&Kra:2016} proved that for minimal subshifts $(X,S)$ of sub-quadratic complexity, the quotient ${\rm Aut} (X,S)/\langle  S\rangle$ is periodic, where  ${\rm Aut} (X,S)$ stands for the automorphism group  of $(X,S)$ and $\langle  S\rangle$ is the group generated by the shift map.  
Whereas in~\cite{Coven&Quas&Yassawi:2016,Cyr&Kra:2015,Donoso&Durand&Maass&Petite:2016} it was shown that automorphism groups of subshifts with sub-affine complexity along a subsequence (which includes minimal substitution subshifs) are virtually $\mathbb{Z}$.
In other words, there exists a finite set of automorphisms $F$ such that any automorphism is the composition of an element of $F$ with a power of the shift.

When studying factor maps instead of automorphisms we lose the group structure. 
The main feature used in the articles cited above is that automorphisms can be iterated, which is not the case for factor maps, and that automorphisms preserve some key features of the dynamics. 
In~\cite{Donoso&Durand&Maass&Petite:2016}, the authors used the concept of asymptotic pairs.
Unfortunately images and preimages of asymptotic pairs via factor maps are neither well controlled nor understood. 
Nevertheless, for specific families and with different approaches similar results can be obtained: Theorem~\ref{theo:main} for minimal substitution subshifts and Theorem~\ref{theo:factorLR} for linearly recurrent subshifts.

The first author proved that for linearly recurrent subshifts there are finitely many subshift factors up to isomorphism~\cite{Durand:2000}.
V.~Salo and I.~T\"orm\"a~\cite{Salo&Torma:2015} gave a more precise result for factor maps between two minimal substitution subshifts of constant length  and for a very particular family of substitutions (mainly Pisot substitutions): there exists a bound $R$ such that any factor map is the composition of a sliding block code of radius less than $R$ with a power of the shift map.
A similar result has also been obtained in~\cite{Coven&Dykstra&Keane&Lemasurier:2014} (see also~\cite{Coven&Keane&Lemasurier:2008}).
Salo and T\"orm\"a use a renormalization process that, given a factor map with some radius, allows one to obtain another one with a smaller radius. 
It has also been used in~\cite{Coven&Quas&Yassawi:2016} to study the automorphism groups of constant length substitution subshifts. 
We extend the result of Salo and T\"orm\"a to the whole class of aperiodic minimal morphic subshifts. Furthermore, we show that the bound $R$ is computable.

\begin{theo}
\label{theo:main}
Let $(X , S)$ and $(Y , S)$ be two subshifts generated by uniformly recurrent morphic sequences.
Then there exists a computable constant $R$ such that every factor map from $(X, S)$ to $(Y , S)$ is the composition of a power of the shift map $S$ with a factor map of radius less than $R$.
\end{theo}

This greatly extends the results obtained in the papers mentioned above, which considered the much more restrictive case of constant length substitutions (except for~\cite{Salo&Torma:2015} where the non-uniform Pisot case is also considered).
The proof will be divided into two cases: when $(Y,S)$ is periodic and when it is not.
Theorem~\ref{theo:main} is one of the two main ingredients of the proof of Theorem~\ref{theo:cordecidfactor}.  
The other one is the following.

\begin{theo}
\label{theo:main2}
Let $(X , S)$ and $(Y , S)$ be two subshifts generated by uniformly recurrent morphic sequences and
$\phi $ be a sliding block code.
It is decidable whether $\phi$ defines a factor map from $(X,S)$ to $(Y,S)$.
\end{theo}

The proof of Theorem \ref{theo:main2} relies on a different method. 
We consider constructions based on return words to clopen sets and use the fact that minimal morphic subshifts have finitely many induced systems on cylinder sets up to conjugacy~\cite{Durand:1998,Holton&Zamboni:1999}. 
In the context of constant length substitutions (or more generally, for subshifts generated by automatic sequences), this theorem was obtained 20 years ago by I.~Fagnot~\cite{Fagnot:1997a} using the first order logic framework of Presburger arithmetic, without assuming minimality.

Let us mention an interesting problem tackled in~\cite{Coven&Dekking&Keane:2017} for minimal constant length substitution subshifts that naturally follows the discussion above.
From~\cite{Durand:2000,Durand:2013b}, minimal morphic subshifts have finitely many aperiodic subshift factors.
Thus, it is natural to try to produce (to compute) the finite list of all such factors of a given minimal morphic subshift $(X,S)$ (up to isomorphism). 
A strategy would be to consider one by one factor maps of increasing radius. 
For each new factor map, a factor $(Y ,S)$ is obtained. 
One can compare it with the previously obtained subshifts using Corollary~\ref{cor:cordecidfactor}.
It is tempting to say in view of Theorem~\ref{theo:main} that there are finitely many factor maps to test, those of radius less than $R$.
Unfortunately in our proof this constant strongly depends on the factor $(Y,S)$, in fact on the constant of recognizability of the substitution defining it.
We leave this as an open problem.
Nevertheless, if we ask for both $(X,S)$ and $(Y,S)$ to be constant length substitution subshifts then this strategy is successful and the list of factors can be given (Section~\ref{section:factorlist}) as shown in~\cite{Coven&Dekking&Keane:2017}.

\subsection*{Organization of the paper}

In Section~\ref{section:basic} the different classical notions and constructions are defined and some of their properties stated, which will be used throughout this paper.

In Section~\ref{section:theomorphic} we prove (Theorem~\ref{theo:durand13bbis}) that minimal morphic subshifts are isomorphic to primitive substitution subshifts. 
In fact this was already known as it is a straightforward consequence of a combination of results from~\cite{Durand&Host&Skau:1999} and 
~\cite{Durand:2013b}.
We provide a proof as we need to make precise the radius of the isomorphisms involved.
This is used in many proofs of this paper as it substantially simplifies the assumptions. 
 
 Dealing with decision problems we need throughout this paper computable bounds of key features of the substitution subshifts. 
 This is done in Section~\ref{section:boundsprim} for the recognizability, the linear recurrence, the return words, the growth of the substitutions, the complexity functions and the ergodic measures.
 We also recall the classical construction of substitutions on the blocks of length $n$~\cite{Queffelec:2010}.
 
Section~\ref{section:bounds} is devoted to the proof of Theorem~\ref{theo:main} for the aperiodic case.
We introduce a crucial  concept for this paper.
This is what Salo and T\"orm\"a~\cite{Salo&Torma:2015} called dill maps.  
These dill maps are defined by a renormalization process on factor maps.
Salo and T\"orm\"a showed that, starting with a factor map between two constant length substitution subshifts, this renormalization process yields another factor map between the same subshifts and with a smaller radius. 
This is no longer true for non-constant length substitutions.  
Nevertheless this produces dill maps for which the notion of radius can also be defined.
Furthermore, the renormalization of a dill map provides another dill map and the iteration of the process leads to dill maps with radius smaller than some computable constant $R$. 
Using the pigeonhole principle, we obtain twice the same dill map in this iteration process and it happens that these two dill maps are in fact factor maps, thus with a small computable radius.
It finally turns out that these factor maps are equal to the original factor map up to a power of the shift map.

When the factors are periodic the decision problems are easier to deal with. 
This is done in Section~\ref{section:periodic}.

Theorem~\ref{theo:cordecidfactor}, Corollary~\ref{cor:cordecidfactor}  and Theorem~\ref{theo:main2} are shown in Section~\ref{section:decidability}.
We first prove these results for minimal substitution subshifts and then for minimal morphic subshifts using Theorem~\ref{theo:durand13bbis}.
We use Theorem~\ref{theo:main} and constructions based on return words to sets of words. 

In Section~\ref{section:constantlength} we consider as a particular example the family of constant length substitution subshifts. 
For this context, we make precise Theorem~\ref{theo:main} showing we have a much better control on the radius of the factor maps.
This improves one of the main results in~\cite{Coven&Dekking&Keane:2017}.
We also give a different proof of Theorem~\ref{theo:main2} using first order logic of the Presburger arithmetic.
This provides a new proof of Theorem~\ref{theo:cordecidfactor} and more efficient algorithms for this framework.

In Section~\ref{section:LR} we consider the wider family of linearly recurrent subshifts. 
We obtain the following analogue of Theorem~\ref{theo:main} but without any control on the radius of the factor maps.
\begin{theo}
\label{theo:factorLR}
Let $(X,S)$ and $(Y,S)$ be aperiodic subshifts.
Suppose $(X,S)$ is linearly recurrent for the constant $K$.
Then, there exist $N \leq (2(K+1))^{36 (K+1)^5}$ and some factor maps $f_i : (X , S) \to (Y , S)$, $1 \leq i \leq N$ such that if $f:(X,S) \to (Y,S)$ is a factor map then there exist $k \in \mathbb{Z}$ and $i \in \{1,\dots,N\}$ such that $f = S^k \circ f_i$.
\end{theo}

In Section~\ref{section:openquestions} we discuss some open questions.

\section{Basic materials}\label{section:basic}

\subsection{Words and sequences}

An {\em alphabet} $A$ is a set of elements called {\em letters}. Unless explicitly stated, we consider that the alphabets are finite. 
A {\em word} over $A$ is an element of the free monoid generated by $A$, denoted by $A^*$. 
Let $u = u_0u_1 \cdots u_{n-1}$ (with $u_i\in A$, $0\leq i\leq n-1$) be a word, its {\em length} is $n$ and is denoted by $|u|$. 
The {\em empty word} is the neutral element of $A^*$ and is denoted by $\epsilon$; $|\epsilon| = 0$. 
The set of non-empty words over $A$ is denoted by $A^+$. 
A {\em subword} of $u$ is a finite word $y$ such that there exist two words $v$ and $w$ satisfying $u = vyw$.
When $v$ (resp. $w$) is the empty word, we say that $y$ is a {\em prefix} (resp. {\em suffix}) of $u$.
If $k,l$ are integers such that $0 \leq k \leq l < |u|$, we let $u_{[k,l]} = u_{[k,l+1)}$ denote the subword $u_k u_{k+1} \cdots u_l$ of $u$.
If $l < k$, then $u_{[k,l]}$ is the empty word.
If $y$ is a subword of $u$, the {\em occurrences} of $y$ in $u$ are the integers $i$ such that $u_{[i,i + |y|)}= y$. 
If $y$ has an occurrence in $u$, we also say that $y$ {\em occurs} in $u$.

In this article $\mathbb{K}$ will stand for $\mathbb{N}$ or $\mathbb{Z}$.
The elements of $A^{\mathbb{K}}$ are called {\em sequences}.
When we want to emphasize the type of sequences we are dealing with we say \emph{one-sided} sequences for elements of $A^\mathbb{N}$ and 
\emph{two-sided} sequences for elements of $A^\mathbb{Z}$.
For $x = (x_n)_{n \in \mathbb{Z}} \in A^\mathbb{Z}$, we let $x^+$ and $x^-$ respectively denote the sequences $(x_n)_{n \geq 0}$ and $(x_n)_{n < 0}$.
The notion of subword is naturally extended to sequences, as well as the notion of prefixes when $\K = \N$.
The set of subwords of length $n$ of $x$ is written $\mathcal{L}_n(x)$ and the set of subwords of $x$, or the {\em language} of $x$, is denoted by $\mathcal{L}(x)$.
We let $p_x: \N \to \N, n \mapsto \# \mathcal{L}_n(x)$, denote the {\em word complexity (function)} of $x$. 

The sequence $x\in A^\mathbb{N}$ is {\em ultimately periodic} if there exist a word $u$ and a non-empty word $v$ such that $x=uv^{\omega}$, where $v^{\omega}= vvv\cdots$.
It is {\em periodic} if $u$ is the empty word. 
A sequence that is not ultimately periodic is called {\em aperiodic}.
A word $u$ is {\em recurrent in} $x$ if it occurs in $x$ infinitely many times.
A sequence $x$ is {\em recurrent} if every subword $u$ of $x$ is recurrent in $x^+$ and in $x^-$ (whenever $x$ is two-sided).
It is {\em uniformly recurrent} if it is recurrent and for any subword $u$ of $x$, the greatest difference between two successive occurrences of $u$ in $x$ is finite.
It is {\em linearly recurrent} for the constant $K$ if it is recurrent and if this greatest difference is at most $K|u|$.

\subsection{Subshifts, minimality, factors, local and induced maps}

A {\em (topological) dynamical system} is a pair $(X,S)$ where $S$ is a continuous map from $X$ to $X$ and $X$ is a compact metric space.
It is {\em aperiodic} if there do not exist {\em $S$-periodic points}, that is points $x$ such that $S^n x = x$ for some $n > 0$.
A subset $Y \subset X$ is {\em $S$-invariant} if $S^{-1} Y = Y$.
A dynamical system $(X,S)$ is {\em minimal} whenever the only closed $S$-invariant subsets of $X$ are $X$ and $\emptyset$.

A dynamical system $(X,S)$ is a {\em subshift} when $X$ is a closed $S$-invariant subset of $A^\mathbb{Z}$ (for the usual topology on $A^\mathbb{Z}$) and $S$ is the {\em shift map} ($S:A^{\mathbb{Z}}\to A^{\mathbb{Z}},\ (x_n)_{n\in \mathbb{Z}}\mapsto (x_{n+1})_{n\in \mathbb{Z}})$.
In this paper, even if the alphabet changes, $S$ will always denote the shift map and we usually say that $X$ itself is a subshift.
The {\em language} $\mathcal{L}(X)$ of a subshift $(X,S)$ is the union of the languages of $x$ for $x \in X$.
The word complexity function naturally extends to subshifts.
We denote it by $p_X$.

Let $x$ be a sequence belonging to $A^\mathbb{Z}$ or $A^\mathbb{N}$.
Consider the set $\Omega (x) = \{ y\in A^\mathbb{Z} \mid \mathcal{L} (y) \subset \mathcal{L} (x) \}$.
Then, $(\Omega (x) ,S)$ is a subshift.
We call it the {\em subshift generated (or defined)} by $x$. 
If $x$ is two-sided or one-sided and recurrent, the subshift $(\Omega (x) ,S)$ is minimal if and only if $x$ is uniformly recurrent. 
In this case all sequences $y$ in $\Omega (x)$ are uniformly recurrent and satisfy $\Omega (y) = \Omega (x)$, $\mathcal{L} (x) = \mathcal{L}(y)$ and $\Omega (x) = \overline{\{  S^n y \mid n\in \mathbb{Z} \}}$.
A subshift $(X,S)$ is {\em linearly recurrent} if it is a subshift generated by a linearly recurrent sequence $x$.
If $K$ is a constant of linear recurrence of $x$, then all sequences $y \in X$ are linearly recurrent for the constant $K$.

Let $(X,S)$ be a subshift on the alphabet $A$, let $D$ be some set endowed with the discrete topology and let $g : X \to D$ be a continuous map.
Then $g$ is locally constant, which implies that there exist $t,s \in \mathbb{N}$ and a map $\hat{g} : A^{t+s+1} \to D$, called a {\em local map} defining $g$, such that 
$g(x) = \hat{g} (x_{[-t,s]})$ for all $x\in X$.
We say that $\hat{g}$ defines $g$, and $t$ is called the \emph{memory} and $s$ the \emph{anticipation} of $\hat{g}$.
When $s=t$ we call their common value $r$ the  {\em radius} of $\hat{g}$ and a radius of $g$ (there exist other local maps defining $g$).
The {\em radius} of $g$ is the smallest one among the set of radii of local maps defining $g$.

Let $(X,T)$ and $(Y,R)$ be two dynamical systems.
We say that $(Y,R)$ is a {\em factor} of $(X,T)$ if there is a continuous and onto map $f: X \rightarrow Y$ such that $f \circ T = R \circ f$.
We call $f$ a {\em factor map}.
If, moreover, $f$ is one-to-one we say that $f $ is an {\em isomorphism}.

Suppose $f \colon (X,S ) \to (Y , S )$ is a factor map between subshifts, where $X \subset A^\Z$ and $Y \subset B^\Z$.
Then, considering the map $g:X \to B$ given by $g(x) = f(x)_0$ there is a local map $\hat{g} : A^{t+s+1} \to B$ defining $g$ and satisfying 
 $f(x)_n=\hat{g} (x_{[n-t,n+ s]})$ for any $n \in \mathbb{Z}$ and $x \in X$.
 This is the Curtis-Hedlund-Lyndon Theorem (see~\cite{Lind&Marcus:1995}).
 The map $\hat{g}$ is called a {\em sliding block code} defining $f$ and we extend the notions of memory, anticipation and radius of the map $g$ to the map $f$.
The sliding block code $\hat{g}$ naturally extends to the set of words of length at least $t+s+1$, and we denote this map also by $\hat{g}$.

\subsection{Morphic and substitutive sequences}
\label{subsection:morphicsubstitutive}

Let $A$ and $B$ be finite or infinite alphabets. 
By a {\em morphism} from $A^*$ to $B^*$ we mean a homomorphism of free monoids.
Let $\sigma$ be a morphism from $A^*$ to $B^*$. 
When $\sigma (A) = B$, we say that $\sigma$ is a {\em coding}. 
Thus, codings are onto.
We set $|\sigma | = \max_{a\in A} |\sigma (a)|$ and $\langle \sigma \rangle = \min_{a\in A} |\sigma (a)|$.
The morphism $\sigma$ is of {\em constant length} if $\langle \sigma \rangle = |\sigma|$.
We say $\sigma$ is {\em erasing} if there exists $b\in A$ such that $\sigma (b)$ is the empty word. 
If $\sigma$ is non-erasing, it induces by concatenation a map from $A^{\mathbb{N}}$ to $B^{\mathbb{N}}$ and a map from $A^{\mathbb{Z}}$ to $B^{\mathbb{Z}}$. 
These maps are also denoted by $\sigma$.
The language of the endomorphism $\sigma : A^* \to A^*$ is the set $\mathcal{L} (\sigma )$ of words occurring in some $\sigma^n (a)$, $a\in A$, $n \in \mathbb{N}$.
When $\mathcal{L} (\sigma ) $ is infinite, we let $X_\sigma$ denote the set $\{y \in A^\Z \mid \mathcal{L}(y) \subset \mathcal{L}(\sigma)\}$.
It is closed and $S$-invariant.
We say that $(X_\sigma , S)$ is the subshift generated by $\sigma$.

In this paper we use the definition of substitution from~\cite{Queffelec:2010} and the notion of substitutive sequence defined in~\cite{Durand:1998}.
Both are restrictive in general but not in the uniformly recurrent case as shown in~\cite{Durand:2013}. 
For non-restrictive context we will speak of (prolongable) endomorphisms and (purely) morphic sequences.
We warn the reader that in what follows, we define several closely related notions such as morphic subshift, (primitive) substitutive subshift and (primitive) substitution subshift. 
The notions of minimal morphic subshifts and primitive substitution subshifts are equivalent in terms of topological dynamical systems, as shown in Theorem \ref{theo:durand13bbis}, but not in term of sets. 
In Section \ref{subsec:notpurely} we give an example of a minimal morphic subshift $(X,S)$ that is not, in terms of sets, a primitive substitution subshift. 

Let $\sigma:A^* \to A^*$ be an endomorphism.
If there exist a letter $a\in A$ and a non-empty word $u$ such that $\sigma(a)=au$ and if moreover $\lim_{n\to+\infty}|\sigma^n(a)|=+\infty$, then $\sigma$ is said to be {\em right-prolongable on $a$}.
It is a {\em substitution} whenever it is right-prolongable on some letter $a \in A$ and {\em growing} (that is, $\lim_n \langle \sigma^n \rangle = +\infty$). 
Then $(X_\sigma , S)$ is called the {\em substitution subshift} generated by $\sigma$.
Suppose that $\sigma$ is right-prolongable on $a \in A$. 
Since for all $n\in\mathbb{N}$, $\sigma^n(a)$ is a prefix of $\sigma^{n+1}(a)$ and because $|\sigma^n(a)|$ tends to infinity with $n$, the sequence $(\sigma^n(aaa\cdots ))_{n\ge 0}$ converges in $A^\mathbb{N}$ to a sequence denoted by $\sigma^\omega(a)$ which is a fixed point of $\sigma$: $\sigma (\sigma^\omega(a)) = \sigma^\omega(a)$.
A sequence obtained in this way is said to be {\em purely morphic} (w.r.t. $\sigma$) or {\em purely substitutive} when $\sigma$ is a substitution. 
If $x\in A^\mathbb{N}$ is purely morphic and $\phi:A^*\to B^*$ is a morphism such that $\phi(x)$ is a sequence, then $y=\phi (x)$ is said to be a {\em morphic sequence} (w.r.t. ($\sigma$,$\phi$)) and the subshift it generates is called a {\em morphic subshift}.
When $\phi $ is a coding and $\sigma $ a substitution, we say $y$ is {\em substitutive} (w.r.t. ($\sigma  ,\phi$)) and the subshift it generates is called {\em substitutive subshift}.

Two-sided fixed points of $\sigma$ can be similarly defined: 
When $\sigma$ is right-prolongable on $a \in A$ and {\em left-prolongable on $b \in A$} (that is $\sigma(b)=vb$ for some non-empty word $v$ and $\lim_{n \to +\infty} |\sigma^n(b)|=+\infty$), the sequence $\sigma^n(\cdots bbb \cdot aaa \cdots)$ converges to $\sigma^{ \omega}(b \cdot a)$ which is a fixed point of $\sigma$.
It can happen that $\sigma^{ \omega}(b \cdot a)$ does not belong to $X_\sigma$. 
In fact, $\sigma^{ \omega}(b \cdot a)$ belongs to $X_\sigma$ if and only if $ba$ belongs to $\cL(\sigma)$.
In this case, we say that $\sigma^{ \omega}(b \cdot a)$ is an {\em admissible} fixed point of $\sigma$.

With a morphism $\sigma:A^* \to B^*$ is naturally associated the {\em incidence matrix} $M_{\sigma} = (m_{b,a})_{b\in B , a \in A }$ where $m_{b,a}$ is the number of occurrences of $b$ in the word $\sigma(a)$.
Let $\tau : B^* \to C^*$ be a morphism. 
It is useful, and classical, to observe that $M_{\tau \circ \sigma } = M_\tau M_\sigma $, $|\sigma | = \max_{a\in A}({\bf 1}M_\sigma )_a$ and $\langle \sigma \rangle = \min_{a\in A}({\bf 1}M_\sigma )_a$ where ${\bf 1}$ is the row vector consisting of ones.

Whenever $A=B$ and the matrix associated with $\sigma $ is primitive ({\em i.e.}, when it has a power with strictly positive coefficients) we say that $\sigma$ is a {\em primitive endomorphism}.  
In this situation we easily check that $\mathcal{L} (\sigma ) = \mathcal{L} (\sigma^\omega (a))$ for any letter $a \in A$ on which $\sigma$ is right-prolongable. 
In particular, we also have that for all $x \in X_\sigma$, $\Omega(x) = \Omega(\sigma^\omega(a)) = X_\sigma$. 
If $\sigma$ is primitive we say $(X_\sigma , S)$ is a {\em primitive substitution subshift}.
Such a subshift is minimal~\cite{Queffelec:2010}.
A sequence $x$ is {\em primitive substitutive} if it is substitutive w.r.t. a primitive substitution.
The subshift $(\Omega (x) , S)$ it generates is called {\em primitive substitutive subshift}.
It is clearly minimal.
It can be easily checked that a subshift is a primitive substitution subshift if and only if it is the orbit closure of a primitive purely substitutive sequence.
We say $\sigma $ is {\em aperiodic} whenever $(X_\sigma , S)$ is aperiodic.

It is well known that if a matrix $M$ is primitive, its spectral radius $\rho(M) = \max\{|\lambda| \mid \lambda \in \mathrm{Spec}(M)\}$ is an eigenvalue of $M$ which is algebraically simple.
Furthermore, any eigenvalue of $M$ different from $\rho(M)$ has modulus less than $\rho(M)$.
We call $\rho(M)$ the {\em dominant eigenvalue of $M$}. 
By abuse of language, when $M$ is the incidence matrix of a primitive endomorphism $\sigma$, we call $\rho(M)$ the {\em dominant eigenvalue of $\sigma$}.

\subsection{Data of computation}

This paper is concerned with decidability and computability problems about morphic subshifts.
It is therefore important to make precise what are the data that we are working with.

When a statement concerns a morphic sequence $x$, the data that we assume to be given are two morphisms $\sigma:A^* \to A^*$, $\phi:A^* \to B^*$ and a letter $a \in A$ such that 
\begin{enumerate}
    \item $\sigma$ is right prolongable on $a$
    \item $\lim_{n \to +\infty} |\phi(\sigma^n(a))| = +\infty$    
\end{enumerate}
and it is understood that $x = \lim_{n \to +\infty} \phi(\sigma^n(a))$. 
Observe that both conditions are decidable from $(\sigma,\phi,a)$.
Furthermore, such statements usually assume that $x$ is uniformly recurrent, which is also decidable from $(\sigma,\phi,a)$~\cite{Durand:2013b}.
They also sometimes assume that $x$ is aperiodic, which is again decidable from $(\sigma,\phi,a)$~\cite{Durand:2013}.

\subsection{Examples}

Throughout the paper, we will try to illustrate our results with examples. To keep computations reasonable, we stick to well-known examples. In particular, we will mainly focus on:
\begin{enumerate}
\item
The {\em Fibonacci sequence} is the fixed point $\varphi^{\omega}(a)$, where $\varphi$ is the {\em Fibonacci substitution} defined by
\[
	\varphi: 
	\begin{cases}
		a \mapsto ab \\
		b \mapsto a
	\end{cases}.
\]
We have
\[
	\varphi^{\omega}(a) = abaababaabaababaababaabaababaabaababaababaabaababaa \cdots
\]
The Fibonacci substitution is primitive, and therefore the Fibonacci sequence is uniformly recurrent. 
The {\em Fibonacci subshift} is $\Omega(\varphi^{\omega}(a))$.

\item
The {\em Thue-Morse sequence} is the fixed point $\nu^{\omega}(a)$, where $\nu$ is the {\em Thue-Morse substitution} defined by
\[
	\nu: 
	\begin{cases}
		a \mapsto ab \\
		b \mapsto ba
	\end{cases}.
\]
We have
\[
	\nu^{\omega}(a) = abbabaabbaababbabaababbaabbabaabbaababbaabbabaababb \cdots
\]
The Thue-Morse substitution is primitive, hence the Thue-Morse sequence is uniformly recurrent. 
The {\em Thue-Morse subshift} is $\Omega(\nu^{\omega}(a))$.

\item
The {\em Chacon sequence} is the fixed point $\gamma^{\omega}(a)$, where $\gamma$ is the {\em Chacon morphism} defined by
\[
	\gamma: 
	\begin{cases}
		a \mapsto aaba \\
		b \mapsto b
	\end{cases}.
\]
We have
\[
	\gamma^{\omega}(a) = aabaaababaabaaabaaababaababaabaaababaabaaabaaababaa \cdots
\]
The Chacon sequence is uniformly recurrent, but $\gamma$ is clearly not primitive. 
The {\em Chacon subshift} is $\Omega(\gamma^{\omega}(a))$.
\end{enumerate}

\section{From uniformly recurrent morphic sequences to purely substitutive sequences w.r.t. proper primitive substitutions}
\label{section:theomorphic}

In this section we show that we can restrict our study to subshifts generated by purely substitutive sequences where the underlying substitution is primitive and proper. 
A substitution $\sigma: A^* \to A^*$ is {\em left proper} if there exists $a\in A$ such that $\sigma(A)$ is included in $ a A^*$.
It is {\em right proper} if there exists $b\in A$ such that $\sigma(A)$ is included in $  A^*b$.
It is {\em proper} whenever it is both left and right proper.

We show in Theorem~\ref{theo:durand13bbis} below that any aperiodic minimal morphic subshift is isomorphic to a subshift generated by a purely substitutive sequence whose underlying substitution is primitive and proper.
We also give computable bounds for the radii of the isomorphism and of its inverse.

Nevertheless it is not true that minimal morphic subshifts are substitution subshifts. 
In Section \ref{subsec:notpurely} we give an example of a minimal morphic subshift that is not a substitution subshift.

\begin{theo}
\label{theo:durand13bbis}
Let $y \in A^\mathbb{N}$ be an aperiodic uniformly recurrent morphic sequence and let $(Y,S)$ be the subshift it generates. 
There exist a computable constant $K$, a computable primitive proper substitution $\sigma $ and an isomorphism $\phi $ from $(X_\sigma ,S)$ onto $(Y,S)$ whose radius is 0 and such that the radius of $\phi^{-1}$ is less than $K$.
\end{theo}

\begin{Rem}
\label{rem:factor of radius K+r}
Theorem~\ref{theo:durand13bbis} is the first step to prove Theorem~\ref{theo:main} and Theorem~\ref{theo:main2} (hence Theorem~\ref{theo:cordecidfactor}) for aperiodic subshifts.
Indeed, suppose that we are given two aperiodic minimal morphic subshifts $(X , S)$ and $(Y , S)$.
Using Theorem~\ref{theo:durand13bbis}, we can compute primitive and proper substitutions $\sigma$ and $\tau$ such that $(X,S)$ is isomorphic to $(X_\sigma,S)$ and $(Y,S)$ is isomorphic to $(X_\tau,S)$, with isomorphisms of computable radii.
In the next sections, we show that for factor map $f : (X_\sigma,S) \to (X_\tau,S)$, with $\sigma $ and $\tau $ primitive and $(X_\tau,S)$ non-periodic,  the radius of some $f\circ S^n$ is bounded by a computable constant only depending on $\sigma$ and $\tau$, Theorem~\ref{theo:mainpropersub}.
We also show that it is decidable whether a given sliding block code defines a factor map from $(X_\sigma,S)$ onto $(X_\tau,S)$ (Theorem~\ref{theo:decidcoding}). 
This will complete the proofs of Theorem~\ref{theo:main} and Theorem~\ref{theo:main2} by composing the factor maps with the isomorphisms.
\end{Rem}

\begin{Rem}
As the sequence $y$ in Theorem~\ref{theo:durand13bbis} is uniformly recurrent, we have $Y = \Omega(y^+)$.
Therefore, except for the definitions, we essentially consider one-sided sequences in this section.
\end{Rem}

\subsection{Return words and derived sequences}
\label{subsection:return words}

Let $x \in A^\K$ be a sequence and $u$ be a word in $\mathcal{L} (x)$.
We call {\em return word to $u$} any word $w \in A^*$ such that $wu$ is a word in $\mathcal{L}(x)$ that admits $u$ as a prefix and $u$ occurs exactly twice in $wu$.
The set of return words to $u$ in $x$ is denoted by ${\mathcal{R}}_{x,u}$.
It is easily seen that a recurrent sequence $x$ is uniformly recurrent if and only if for all $u \in \cL(x)$, the set ${\mathcal{R}}_{x,u}$ is finite.

\begin{Rem}
\label{remark:power return word}
If $w$ is a return word to $u$ in $x \in A^\mathbb{K}$, then $wu = w^{1+|u|/|w|}$, where the fractional power $w^{p/|w|}$, $p \in \mathbb{N}$, of a word $w$ is the word $w^n w_{[0,r[}$ such that $p/|w| = n+r/|w|$ with $n,r \in \mathbb{N}$ and $r < |w|$.
In particular, the word $w^{1+|u|/|w|}$ belongs to $\mathcal{L}(x)$.
\end{Rem}

When dealing with return words, it is convenient to enumerate the elements of $\mathcal{R}_{x,u}$ in the order of their first occurrence in $x^+$.
Formally, we consider the set $R_{x,u} = \{1,2,\dots\}$ with $\#R_{x,u} = \#\mathcal{R}_{x,u}$ and we define 
\[
	\Theta_{x,u}: R_{x,u}^* \to \mathcal{R}_{x,u}^*
\]
as the unique morphism that maps $R_{x,u}$ bijectively onto $\mathcal{R}_{x,u}$ and such that for all $i \in R_{x,u}$, $\Theta_{x,u}(i)u$ is the $i$th word of $\mathcal{R}_{x,u}u$ occurring in $x^+$.
When $x$ is uniformly recurrent, for all $y \in X = \Omega(x)$ we have ${\mathcal{R}}_{y,u}={\mathcal{R}}_{x,u}$ so we let $\mathcal{R}_{X,u}$ denote this set.

\begin{exem}
\label{ex:returns}
For the Fibonacci sequence $x$, we have the return words $\cR_{x,a} = \{a, ab\}$ and $\cR_{x,aa} = \{aab, aabab\}$ and the morphisms $\Theta_{x,a}$, $\Theta_{x,aa}$ are given by
\[
	\Theta_{x,a}:
		\begin{cases}
			1 \mapsto a \\
			2 \mapsto ab
		\end{cases}
	\quad \text{and } \quad
	\Theta_{x,aa}:
		\begin{cases}
			1 \mapsto aabab \\
			2 \mapsto aab
		\end{cases}.
\]

For the Thue-Morse sequence $y$, we have the following sets of return words $\cR_{y,b} = \{b, ba, baa\}$, $\cR_{y,bb} = \{bbaa, bbaaba, bbabaa, bbabaaba\}$ and the morphisms $\Theta_{y,b}$, $\Theta_{y,bb}$ are given by
\[
	\Theta_{y,b}:
		\begin{cases}
			1 \mapsto b 	\\
			2 \mapsto ba 	\\
			3 \mapsto baa
		\end{cases}
	\quad \text{and } \quad
	\Theta_{y,bb}:
		\begin{cases}
			1 \mapsto bbabaa 	\\
			2 \mapsto bbaaba 	\\
			3 \mapsto bbaa		\\
			4 \mapsto bbabaaba
		\end{cases}.
\]

For the Chacon sequence $z$, we have the following sets of return words 
$\cR_{z,a} = \{a,ab\}$, $\cR_{z,aab} = \{aaba,aabab\}$ and the morphisms $\Theta_{z,a}$, $\Theta_{z,aab}$ are given by
\[
	\Theta_{z,a}:
		\begin{cases}
			1 \mapsto a 	\\
			2 \mapsto ab 	
		\end{cases}
	\quad \text{and } \quad
	\Theta_{z,aab}:
		\begin{cases}
			1 \mapsto aaba 	\\
			2 \mapsto aabab 
		\end{cases}.
\]
\end{exem}

The next results are classical when dealing with return words. 
The first three are proved in~\cite{Durand:1998} for one-sided sequences and when $u$ is a prefix of $x$, but the proofs are exactly the same in the two-sided case and when $u$ is a factor but not a prefix. 
The proof of Proposition~\ref{prop:def suite derivee singleton} that is in~\cite{Durand:1998} only concerns $x$ itself (and not $y \in \Omega(x)$) but again, the proof is the same.

\begin{prop}[\cite{Durand:1998}]
\label{prop: return words code}
Let $x \in A^\K$ be a uniformly recurrent sequence and $u$ be a word in $\cL(x)$.
The set $\mathcal{R}_{x,u}$ is a code, {\em i.e.}, the map $\Theta_{x,u} : R_{x,u}^* \to \mathcal{R}_{x,u}^*$ is one-to-one.
\end{prop}

\begin{prop}[\cite{Durand:1998}]
\label{prop:def suite derivee singleton}
Let $x$ be a uniformly recurrent sequence in $A^{\mathbb{K}}$ and let $u$ be a word in $\mathcal{L}(x)$.
Let $y \in \Omega(x)$ and $i \in \mathbb{N}$ be the first occurrence of $u$ in $y^+$.
Then, there exists a unique sequence $z \in R_{x,u}^\mathbb{K}$ satisfying 
$\Theta_{x,u} (z) = S^i(y)$; we call the sequence $z$ the {\em derived sequence} of $y$ (w.r.t. $u$) and we denote it by $\mathcal{D}_u(y)$.
Furthermore, for all $y \in \Omega(x)$, $\mathcal{D}_u(y)$ is uniformly recurrent and $\Omega(\cD_u(y)) = \Omega(\cD_u(x))$. 
\end{prop}

\begin{prop}[\cite{Durand:1998}]
\label{prop:derived sequence fixed point}
Let $x$ be a uniformly recurrent sequence in $A^\mathbb{K}$. 
Let $u$ and $u'$ be two prefixes of $x^+$ where $u$ is a prefix of $u'$.
\begin{enumerate}
\item
There is a unique non-erasing morphism $\lambda_{x,u,u'}:R_{x,u'}^* \to R_{x,u}^*$ such that 
\[
	\Theta_{x,u'} = \Theta_{x,u} \circ \lambda_{x,u,u'};
\]

\item
If $\cD_u(x^+) = \cD_{u'}(x^+)$ and if for all $(w,w') \in \mathcal{R}_{x,u} \times \mathcal{R}_{x,u'}$, $wu$ occurs in $w'u'$, then $\lambda_{x,u,u'}$ is primitive, left proper ($\lambda_{x,u,u'}(R_{x,u'}) \subset 1 R_{x,u}^+$) and such that $\cD_u(x^+) = \lambda_{x,u,u'}^{\omega}(1)$.
\end{enumerate}
\end{prop}

\begin{exem}
\label{ex:tau}
For the Fibonacci sequence $x$, the Thue-Morse sequence $y$ and the Chacon sequence $z$, the morphisms $\lambda_{x,a,aa}:R_{x,aa}^* \to R_{x,a}^*$, $\lambda_{y,b,bb}:R_{y,bb}^* \to R_{y,b}^*$ and $\lambda_{z,a,aab}:R_{z,aab}^* \to R_{z,a}^*$ given by the previous proposition are respectively  
\[
	\lambda_{x,a,aa}:
		\begin{cases}
			1 \mapsto 122  \\
			2 \mapsto 121
		\end{cases},
	\quad
	\lambda_{y,b,bb}:
		\begin{cases}
			1 \mapsto 123 	\\
			2 \mapsto 132 	\\
			3 \mapsto 13	\\
			4 \mapsto 1232
		\end{cases}
	\quad \text{and } \quad
	\lambda_{z,a,aab}: 
		\begin{cases}
			1 \mapsto 121 \\
			2 \mapsto 122
		\end{cases}.
\]
It could furthermore be checked that $\cD_a(z) = \cD_{aab}(z)$ (see Theorem~\ref{theo:morphic UR ssi finite derive} below), so that
\begin{eqnarray*}
	\cD_a(z) = \lambda_{z,a,aab}^\omega(1) 
	&=& 
	1211221211211221221211221211211221211211 \cdots \\
	z = \Theta_{z,a}(\cD_a(z)) 
	&=&
	aabaaababaabaaabaaababaababaabaaababaaba \cdots
\end{eqnarray*}
\end{exem}

The next result is a direct consequence of results in~\cite{Durand:2013b}. 
Since it is not exactly stated as follows, we give a short proof.

\begin{theo}[\cite{Durand:2013b}]
\label{theo:morphic UR ssi finite derive}
Let $x$ be a uniformly recurrent morphic sequence in $A^\mathbb{N}$.
\begin{enumerate}
\item
For all prefixes $u$ of $x$, the morphism $\Theta_{x,u}$ is computable and there exist some computable morphisms $\sigma_u: B^* \to B^*$ and $\psi_u: B^* \to R_{x,u}^*$ such that $\cD_u(x) = \psi_u(\sigma_u^\omega(b))$.
\item
There is computable constant $C$ such that the set $\{(\sigma_u,\psi_u) \mid u \text{ prefix of } x\}$ has cardinality at most $C$. 
In particular, the number of derived sequences of $x$ (on its prefixes) is at most $C$. 
\end{enumerate}
\end{theo}
\begin{proof}
The existence of $\sigma_u$ and $\psi_u$ is~\cite[Proposition 28]{Durand:2013b} and the fact that $\Theta_{x,u}$, $\sigma_u$ and $\psi_u$ are computable is explained in~\cite[Section 4]{Durand:2013b}. 
The bound on the number of possible pairs $(\sigma_u,\psi_u)$ is~\cite[Theorem 29]{Durand:2013b}.
\end{proof}

\begin{theo}[{\cite{Durand&Host&Skau:1999}}]
\label{theo:encad}
Let $x$ be an aperiodic linearly recurrent sequence for the constant $K$. 
For all words $u$ of $\mathcal{L} (x)$, 
\begin{enumerate}
\item
\label{item:theo:encad}
For all $n$, every word of $\mathcal{L}_n (x)$ occurs in every word of $\mathcal{L}_{(K+1)n} (x)$;
\item 
\label{theo:item:1}
For all words $v$ belonging to $\mathcal{R}_{x,u}$, $\frac{1}{K}| u| \leq | v| \leq K  | u|$;
\item 
\label{theo:item:2}
$\# (\mathcal{R}_{x,u})\leq (K+1)^3$;
\item
$p_x(n) \leq K n$ for all $n\geq 1$.
\end{enumerate}
\end{theo}

The bound $K$ of Theorem~\ref{theo:durand13bbis} is associated with some constant of linear recurrence of $(Y,S)$ and $(X_\sigma ,S)$. 
The next result states that it is computable for primitive substitutive sequences.

\begin{prop}[\cite{Durand:1998,Durand&Host&Skau:1999}]
\label{prop:sublinrec}
If $x$ is an aperiodic primitive substitutive sequence (w.r.t. $\sigma:A^* \to A^*$), then it is linearly recurrent for the computable constant 
\begin{equation}
\label{eq:Ksigma}
    K_\sigma = Q_\sigma R_\sigma |\sigma|,
\end{equation}
where $Q_\sigma$ is a constant such that $|\sigma^n| \leq Q_\sigma \langle \sigma^n \rangle$ for all $n$ and $R_\sigma$ is the maximal length of a return word to a word of length 2 in $x$.
\end{prop}


\begin{exem}
\label{ex:constant of LR}
Using the previous result, the Fibonacci sequence is linearly recurrent for the constant $10\frac{1+\sqrt{5}}{2}$. 
Using Walnut\footnote{Walnut is a software program allowing to answer questions or compute quantities related to automatic sequences.}~\cite{Mousavi2016}, it is possible to show that the optimal constant is 3, but that 3 is only necessary for words of length 1 ($aab$ is a return word to $b$).
For longer words, the optimal bound is close to $2.65$.
Similarly, the Thue-Morse sequence is linearly recurrent for the constant $16$. Again using Walnut, it can be shown that the optimal constant is $9$.
For the Chacon sequence, we would first need to consider it as a primitive substitutive sequence (the Chacon morphism not being primitive) to obtain an upper bound on its constant of linear recurrence.
\end{exem}

Observe that in the next sections, we do not always suppose that the substitutions are proper, several results being true without this hypothesis.

\subsection{Proof of Theorem~\ref{theo:durand13bbis}}

\begin{proof}
We make use of 
Proposition~\ref{prop:def suite derivee singleton}, 
Proposition~\ref{prop:derived sequence fixed point},
Theorem~\ref{theo:morphic UR ssi finite derive},
Theorem~\ref{theo:encad} and
Proposition~\ref{prop:sublinrec}.

Assume that $y$ is a one-sided sequence, otherwise, take $y = y^+$.
As $y$ is uniformly recurrent, we have 
$
\lim_{n \to +\infty} \min \{|w| \mid u \in \cL_n(y), w \in \cR_{y,u}\} = +\infty.
$
Using Theorem~\ref{theo:morphic UR ssi finite derive}, we can algorithmically find two prefixes $u$ and $u'$ of $y$ such that 
\begin{itemize}
\item
$\cD_{u'}(y) = \cD_u(y)$;
\item
$\Theta_{y,u}(1)u$ is a prefix of $u'$;
\item
for all $(w,w') \in \cR_{y,u} \times \cR_{y,u'}$, $wu$ occurs in $w'u'$.
\end{itemize}

Set $R = R_{y,u}=R_{y,u'}$.
By Proposition~\ref{prop:derived sequence fixed point}, there is a primitive and left proper substitution $\tau = \lambda_{y,u,u'}$ ($\tau(1) \subset 1 R^+$) such that $\cD_u(y) = \tau^\omega(1)$.
Since $\Theta_{y,u}$ and $\Theta_{y,u'}$ are computable by Theorem~\ref{theo:morphic UR ssi finite derive}, the substitution $\tau$ is computable as well.
In particular, we have $X_\tau = \Omega(\cD_u(y))$, $\Theta_{y,u}(X_\tau) \subset Y$ and, by Proposition~\ref{prop:def suite derivee singleton}, for all $y' \in Y$, there exists a unique pair $(x,n) \in X_\tau \times \N$ such that 
\begin{itemize}
\item
$u$ is prefix of $S^{-n}(y')$ and is not prefix of $S^{-m}(y')$ for all integer $m$ such that $0 \leq m < n$ (in particular, we have $0 \leq n < |\Theta_{y,u}|$);
\item
$y' = S^n \Theta_{y,u}(x)$. 
\end{itemize}

Set $\Theta = \Theta_{y,u}$ and consider the alphabet $D = \{(r,k) \mid r\in R , 0\leq k < | \Theta (r)| \}$
and the morphism $\psi : R^* \rightarrow D^* $ defined by:
\[
\psi (r) = (r,0) \ldots
(r,| \Theta (r)|-1 ).
\]

Since $\langle \tau \rangle \geq 2$, there is an integer $n \leq |\Theta |$ such that $\langle \tau^n \rangle| \geq |\Theta|$.
Let $\sigma $ be the endomorphism from $D^* $ to $ D^* $ defined by:

\begin{center}
\begin{tabular}{rcll}
	$\sigma ((r,k))$	& $=$ & $\psi (\tau^n (r)_{[k,k]})$, & if $0 \leq k<|\Theta (r)|-1$;\\
	$\sigma ((r,|\Theta (r)|-1))$ & $=$ & $\psi(\tau^n(r)_{[|\Theta (r)|-1,|\tau^n(r)|-1]})$, & otherwise.\\
\end{tabular}
\end{center}

Observe that $\tau $ being primitive, so is $\sigma$.
Furthermore, since $\tau$ is left proper, it is easily seen that $\sigma^2$ is left proper.
For all $r$ in $R$, we have
\begin{eqnarray*}
\sigma (\psi (r)) 	
	&=&\sigma ((r,0)\cdots (r,| \Theta (r) |-1 ))\\	
	&=&\psi (\tau^n (r)_{[0,0]} )\cdots 
		\psi (\tau^n(r)_{[| \Theta (r)|-1,| \tau^n (r) |-1]}) \\
 	&=& \psi (\tau^n (r)),
\end{eqnarray*}
hence $ \sigma \circ \psi=\psi \circ \tau^n$.
Since $\tau$ is left proper, it has a fixed point $z \in R^\N$ (which is unique) and it satisfies $\sigma (\psi (z))=\psi (\tau^n (z))=\psi(z)$.
Hence $\psi(z)$ is a fixed point of the substitution $\sigma$ and $X_\sigma = \Omega(\psi (z))$.
In particular, $\psi(X_\tau) \subset X_\sigma$ and for all $x \in X_\sigma$, there exists a unique pair $(z',n) \in X_\tau \times \N$ such that
\begin{itemize}
\item
$(S^{-n}(x))_0 = (r,0)$ for some $r \in R$ and $(S^{-m}(x))_0 \notin R \times \{0\}$ for all integer $m$ such that $0 \leq m < n$ (in particular we have $0 \leq n < |\psi| = |\Theta|$);
\item
$x = S^n \psi(z')$.
\end{itemize}

Let $\phi $ be the coding from $D^*$ to $ A^*$ defined
by $\phi ((r,k)) = \Theta (r)_{[k,k]}$ for all $(r,k)$ in $D$. For all
$r$ in $R$ we obtain
\[
\phi (\psi (r))
= \phi ((r,0) \cdots (r,| \Theta (r) | - 1))= \Theta (r),
\]
and then $\phi \circ \psi = \Theta $.
Let us show that $\phi$ defines a factor map of radius $0$ from $(X_\sigma , S)$ onto $(Y,S)$.
We have $\phi(X_\sigma) = Y$: if $x \in X_\sigma$, then $x = S^n \psi(z')$ for some integer $n$ and some $z' \in X_\tau$. 
Therefore, $\phi(x) = S^n(\phi \circ \psi(z')) = S^n(\Theta(z')) \in Y$.
Reciprocally, if $y' \in Y$, there exists a unique pair $(x,n) \in X_\tau \times \N$ such that $y' = S^n \Theta_{y,u}(x)$ where $u$ is not prefix of $S^{-m}(y')$ for all integer $m$ such that $0 \leq m < n$.
As $\phi \circ \psi = \Theta $, we then have $y' = \phi (S^n(\psi(x)))$, with $\psi(x) \in X_\sigma$.

To show that $\phi$ is an isomorphism, we build a factor map $f$ from $(Y,S)$ onto $(X_\sigma,S)$ and we show that $f \circ \phi$ is the identity.
Let $L$ be some constant of linear recurrence of $\cD_u(y)$. 
It is computable by Proposition~\ref{prop:sublinrec} (since $\tau$ is computable).
Since $y = \Theta(\cD_u(y))$, $y$ is also linearly recurrent and for the constant $L' = |\Theta|L$. 
Set $K = (L+1)|u|+1$.
Let $w = w_{-K} \cdots w_K$, $w_i \in A$ for all $i$, be a word of length $2K+1$ of $\mathcal{L} (Y)$.
By Theorem~\ref{theo:encad}, there is a smallest non-negative integer $i < (L+1)|u|$ such that $w_{[-i,-i+|u|[} = u$.
With the same argument, there is a smallest positive integer $j \leq L |u|+1$ such that $w_{[j,j+|u|[} = u$. 
This defines a letter $r_w \in R$ such that $\Theta(r_w) = w_{[-i,j[}$.
We define $\hat{f}:A^{2K+1} \to R$ by $\hat{f} (w) = \psi (r_w)_{i} = (r_w,i)$ and
we call $f : A^\mathbb{Z} \to R^\mathbb{Z}$ the map defined by $\hat{f}$, i.e., $(f(x))_n = \hat{f}(x_{[n-K,n+K]})$ for all $x,n$. 
Let us show $f \circ \phi$ is the identity.

Indeed, let $x\in X_\sigma$ and $y=\phi (x)$.
To show that $f(y) = x$, it is enough to show that $\hat{f}(y_{[-K,K]}) = x_0$.
Let $i\geq 0$ and $j>0$, be the smallest integers such that $x_{-i} = (r,0)$ and $x_j = (r',0)$ for some $r,r'\in R$.
Then, $x_{[-i, j-1]} = \psi  (r)$ and $y_{[i,j+|u|[} = \phi (x_{[i, j+|u|[}) = \phi (\psi  (r))u = \Theta (r)u$.
Setting $w = y_{[-K,K]}$, we have $r = r_w$ and $\hat{f} (y_{[-K,K]})= \psi (r_w)_i = x_0$.
This shows that $\phi$ is a isomorphism of radius 0 from $(X_\sigma,S)$ onto $(Y,S)$ and that its inverse $f = \phi^{-1}$ has radius $K$.

To conclude the proof, it suffices to show that $(X_\sigma,S)$ can be obtained as a primitive substitution subshift with a proper substitution. 
Indeed, up to now, we only know that $\sigma^2$ is left proper.
Let $a \in D$ be the letter such that $\sigma^2(D) \subset a D^+$.
We consider the substitutions $\sigma'$ and $\xi$ respectively defined by $a \sigma'(d) = \sigma^2(d) a$ for all $d \in D$ and $\xi = \sigma' \circ \sigma^2$.
The substitution $\xi$ is proper, primitive and such that $X_\xi = X_\sigma$, which ends the proof.
\end{proof}

\subsection{Examples}
\label{subsection:example isomorphic}

Let us apply Theorem~\ref{theo:durand13bbis} on two examples.
Our aim is to follow the construction given in the proof of the theorem.

The first one is the Chacon subshift generated by the Chacon sequence $z$.
Let us recall what we know from Example~\ref{ex:returns} and Example~\ref{ex:tau}:
\[
	\Theta_{z,a}:
		\begin{cases}
			1 \mapsto a 	\\
			2 \mapsto ab 	
		\end{cases};
	\quad
	\Theta_{z,aab}:
		\begin{cases}
			1 \mapsto aaba 	\\
			2 \mapsto aabab 
		\end{cases};
	\quad
	\lambda_{z,a,aab}: 
		\begin{cases}
			1 \mapsto 121 \\
			2 \mapsto 122
		\end{cases}.
\]
We consider the notation of the theorem and successively obtain:
\begin{itemize}
\item
$\tau = \lambda_{z,a,aab}$; 
\item
$D = \{(1,0),(2,0),(2,1)\}$ and $n = 1$;
\item
$
	\psi: 
		\begin{cases}
			1 \mapsto (1,0) \\
			2 \mapsto (2,0)(2,1) 			
		\end{cases}
$
and
$	\sigma: 
		\begin{cases}
			(1,0) \mapsto (1,0)(2,0)(2,1)(1,0) \\
			(2,0) \mapsto (1,0)		\\
			(2,1) \mapsto (2,0)(2,1)(2,0)(2,1)	
		\end{cases};
$
\item
$
	\phi:
		\begin{cases}
			(1,0) \mapsto a	 \\
			(2,0) \mapsto a	 \\
			(2,1) \mapsto b	
		\end{cases}.
$
\end{itemize}
The theorem states that $\phi$ is an isomorphism from the Chacon subshift onto the subshift $(X_\xi,S)$, where $\xi$ is the primitive and proper substitution defined by $\xi = \sigma' \circ \sigma^2$ and $\sigma'$ is defined by $(1,0)\sigma'(d) = \sigma^2(d)(1,0)$ for all $d \in D$.

\medskip

Let us consider a second example with a sequence $y$ which is not purely morphic like the Chacon sequence.
We consider the morphisms 
\[
	\sigma:
		\begin{cases}
			0 \mapsto 0123	 \\
			1 \mapsto 02	 \\
			2 \mapsto 13	 \\
			3 \mapsto 3	
		\end{cases}
	\quad \text{and } \quad
	\psi:	
		\begin{cases}
			0 \mapsto abb	 \\
			1 \mapsto ab	 \\
			2 \mapsto a	 	 \\
			3 \mapsto \epsilon	
		\end{cases}.
\]
Let $x= \sigma^\omega (0)$. 
It is clearly non-uniformly recurrent and non-periodic as the words $3^n$, $n\in \mathbb{N}$, appears infinitely many times separated by some words over the alphabet $\{0,1,2\}$.
Nonetheless, $y= \psi (x)$ is uniformly recurrent, as it could be checked by the algorithm given in~\cite{Durand:2013b}.
 
The algorithm of Theorem~\ref{theo:morphic UR ssi finite derive} gives $\mathcal{D}_u (y) = \mathcal{D}_{u'} (y)$, with $u=  abb  a$ and $u' =  abb  ab  a  abb$.
Then, we obtain $z = \cD_u(y) = \tau^\omega(1)$, with $\tau$ defined by $\Theta_{y,u'} = \Theta_{y,u} \circ \tau$ and
\[
	\Theta_{y,u}: 
		\begin{cases}
			1 \mapsto abbaba 	\\
			2 \mapsto abbaab	\\
			3 \mapsto abbabaab	\\
			4 \mapsto abba
		\end{cases},
	\,
	\Theta_{y,u'}: 
		\begin{cases}
			1 \mapsto abbabaabbaab 	\\
			2 \mapsto abbabaababba	\\
			3 \mapsto abbabaabbaababba	\\
			4 \mapsto abbabaab
		\end{cases},
	\,
	\tau: 
		\begin{cases}
			1 \mapsto 12 	\\
			2 \mapsto 34	\\
			3 \mapsto 124	\\
			4 \mapsto 3
		\end{cases}.
\]
We finally obtain a primitive and proper substitution $\xi : \mathcal{A}^* \to \mathcal{A}^*$ on the 24 letter alphabet $D = \{ (i, j) \mid i \in \{ 1,2,3,4\}, 0\leq j < |\Theta_{y,u} (i)|\} $ and a coding $\phi : D^* \to \{ a ,b\}^*$ that induces an isomorphism from $(X_\xi,S)$ onto $(\Omega(y),S)$. 
We could moreover show that $(\Omega(y),S)$ is the Thue-Morse subshift, as it could be checked by using the algorithm that we provide later in this paper.

\subsection{There are minimal morphic subshifts that are not pu\-re\-ly morphic}\label{subsec:notpurely}

When dealing with morphic sequences or subshifts it is rather natural to ask whether it is a purely morphic sequence. 
Indeed, it is much easier to work directly with an endomorphism without having to take care of some coding identifying letters.
A lot, most, of the papers dealing with morphic sequences only consider the case of purely morphic sequences.

Theorem~\ref{theo:durand13bbis} asserts that when working up to isomorphism one can always suppose that a subshift generated by a uniformly recurrent morphic sequences is in fact generated by a primitive proper substitution.

In \cite{Allouche&Cassaigne&Shallit&Zamboni:preprint2018} is discussed the case of sequences instead of subshifts.
The authors give an example (Example 23) of a sequence that is primitive morphic, but not purely morphic. 
But this sequence generates a subshift that is purely morphic. 

Thus, one may ask whether there are minimal morphic subshifts $(X,S)$ that are not purely morphic subshifts, that is $X \neq X_\sigma $ for any endomorphism $\sigma$. 
The following result shows that such examples exist.

\begin{prop}
Let $(X,S)$ be an aperiodic linearly recurrent subshift for the constant $K$.
There exists a factor map $\phi : (X,S) \to (Y,S)$ such that $(Y,S)$ is not purely morphic.
\end{prop}
\begin{proof}
Let $A$ be the alphabet of $X$.
Since $(X,S)$ is minimal, there exists an integer $r \geq 0$ such that for all $u \in \cL_{2r+1}(X)$, there exists a return word $w \in \cR_{X,u}$ of length bigger than $3(K+1)$.
Let us build a sliding block code $\phi: A^{2r+1} \to B$ such that $(\phi(X),S)$ is aperiodic and not generated by a primitive substitution.

For all $u \in \cL_{2r+1}(X)$, let $\phi_u: A^{2r+1} \to B$ be defined by
\[
	\phi_u(v) = 
		\begin{cases}
			1, & \text{if } v=u;	\\
			0, & \text{otherwise}.
		\end{cases}
\]
There exists $u \in \cL_{2r+1}(X)$ such that $(\phi_u(X),S)$ is aperiodic.
Indeed, if for all $u \in \cL_{2r+1}(X)$, $(\phi_u(X),S)$ is periodic of period $p_u$, then $(X,S)$ is periodic of period at most $\prod_{u \in \cL_{2r+1}(X)} p_u$.

Consider $Y = \phi_u(X)$ such that $(Y,S)$ is aperiodic.
By the assumption on $r$, the word $0^{3(K+1)}$ belongs to $\cL(Y)$.
If $Y$ is generated by a primitive substitution $\tau$, then $(\tau^n(0))^{3(K+1)}$ belongs to $\cL(Y)$ for all $n \in \mathbb{N}$.
As $\tau$ is primitive, we have $|\tau^N(0)| \geq r$ for some $N$.
From Theorem~\ref{theo:encad}, all words of $\cL_{|\tau^N(0)|+2r}(X)$ occur in all words of $\cL_{(K+1)(|\tau^N(0)|+2r)}(X)$.
Thus all words of $\cL_{|\tau^N(0)|}(Y)$ occur in all words of $\cL_{(K+1)(|\tau^N(0)|+2r)-2r}(Y)$.
In particular, all words of $\cL_{|\tau^N(0)|}(Y)$ occur in $(\tau^n(0))^{3(K+1)}$.
As there are at most $|\tau^N(0)|$ words of length $|\tau^N(0)|$ in $(\tau^N(0))^{3(K+1)}$, Morse and Hedlund's Theorem implies that $(Y,S)$ is periodic, which is a contradiction.
\end{proof}

\begin{exem}
Let $(X,S)$ be the Thue-Morse subshift that is linearly recurrent for the constant 16 (see Example~\ref{ex:constant of LR}).
Considering the word $u = b\nu^4(a)ba \in \cL(X)$, we can check that the word $$b \nu^4(abaa)\nu^3(b)\nu^2(a)\nu(b)a$$ is a return word to $u$ in $X$ and has length 80 and that the sliding block code $\phi_u:\{a,b\}^{19} \to \{0,1\}$ defined by 
\[
	\phi_u(v) = 1 \Leftrightarrow v=u
\]
is such that $(\phi_u(X),S)$ is aperiodic.
Indeed, it could be checked by hand or using Walnut.
Since $0^{79}$ belongs to $\cL(\phi_u(X))$, $(\phi_u(X),S)$ cannot be generated by a primitive substitution (see the proof above).
\end{exem}


\section{Bounds for primitive substitutive sequences}\label{section:bounds for substitutions}\label{section:boundsprim}

In the previous section, we have shown that to prove Theorem~\ref{theo:cordecidfactor}, Corollary~\ref{cor:cordecidfactor}, Theorem~\ref{theo:main} and Theorem~\ref{theo:main2}, we can restrict ourselves to the case of subshifts generated by primitive and proper substitutions (see Remark~\ref{rem:factor of radius K+r}).
In this section, we consider primitive substitution subshifts $(X_\sigma,S)$ and $(X_\tau,S)$ and develop some tools that will help us to bound the radius of factor maps from $(X_\sigma,S)$ onto $(X_\tau,S)$ (whenever such maps exist).
In particular, we give some computable constants about the substitutions $\sigma$ and $\tau$ (sometimes assuming that there is a factor map from $(X_\sigma,S)$ onto $(X_\tau,S)$) and we consider a variation of the notion of substitution on the words of length $n$~\cite{Queffelec:2010}.

Observe that in this section, we do not assume that $\sigma$ and $\tau$ are proper, the results being true without this hypothesis.

\subsection{Recognizability}
\label{subsec:mosse}

Let $\sigma : A^* \to A^*$ be a primitive substitution.
Taking a power of $\sigma $ if needed, we can suppose $\sigma$ has an admissible fixed point $x$ in $A^\mathbb{Z}$.
Let $E (x , \sigma ) = \{ e_n \mid n\in \mathbb{Z} \}$ be the subset of integers defined by:
$$
e_n = 
\left\{
\begin{array}{ll}
-|\sigma (x[n,-1])|, 	& \hbox{ if } n<0; \\
0, 						& \hbox{ if } n=0; \\
|\sigma (x[0,n-1])|, 	& \hbox{ if } n>0.
\end{array}
\right.
$$
The integers $e_n$ are usually called the {\em cutting bars} of the substitution. 
The following result is fundamental for substitutive subshifts.
Observe that while it was initially written for one-sided sequences, the same proof holds for two-sided ones (see for instance~\cite{Durand&Leroy:2017,Durand&Perrin:2022}).

\begin{theo}
\label{theo:mosse}
Let $\sigma :A^* \to A^*$ be an aperiodic primitive substitution and $x \in A^{\mathbb{Z}}$ be an admissible fixed point of $\sigma$.
\begin{enumerate}
\item
\cite{Mosse:1992}
There exists $M>0$ such that 
$$
(n\in E (x, \sigma ), m\in \mathbb{Z}, x[n-M,n+M] =  x[m-M,m+M]) \Longrightarrow m \in E (x, \sigma ) .  
$$
\item
\cite{Mosse:1996}
\label{th:def:L}
There exists $L\geq M$ such that 
$$
(i,j \in \mathbb{Z}, x[e_i-L,e_i+L] =  x[e_j-L,e_j+L]) \Longrightarrow x_i =x_j .  
$$
\end{enumerate}
In particular, the constants $M$ and $L$ do not depend on the admissible fixed point $x$.
\end{theo}

The smallest $L$ satisfying \eqref{th:def:L} in the previous theorem is called the {\em recognizability constant} for the substitution $\sigma$ and is denoted $L_\sigma $.

Observe that if $u \in \cL(\sigma)$ has length at least $2L+1$ and if $i,j \in \mathbb{Z}$, $i<j$, are two occurrences of $u$ in $x$, then the first part of Mossé's Theorem ensures that 
\[
	[i+L,i+|u|-L) \, \cap \, E(x,\sigma) = ([j+L,j+|u|-L) \, \cap \, E(x,\sigma)) - (j-i)
\]
and the second part states that the central part of $u$ has the same preimage under $\sigma$ at positions $i$ and $j$.
We call the set $([i+L,i+|u|-L) \, \cap \, E(x,\sigma)) -i$ the set of {\em relative cutting bars} of $u$.

\begin{exem}
Let us consider \[
	\tau : 
	\begin{cases}
		a \mapsto aab \\
		b \mapsto ab
	\end{cases}.
\]
It is classical to show that it also defines the Fibonacci subshift.
Let $x$ be the unique admissible fixed point of $\tau$. 
We have $E(x,\tau) = \{i \in \mathbb{Z} \mid x_{i-1}x_{i} = ba\}$ and so $M = 1$. 
The words of length 3 with prefix $ba$ are  $baa$ and $bab$. 
Observe that if $x_{[e_i-1,e_i+1]} = baa$, then $x_i= a$, and, if $x_{[e_i-1,e_i+1]} = bab$, then $x_i = b$.
Hence the constant of recognizability $L_\tau $ of $\tau$ is $1$.

For the Thue-Morse substitution $\nu$, the constant $M$ has to be greater than 1, since the word $aba$ both occurs in the language as $a \nu(b)$ and $\nu(a)b$.
On the contrary, any word $u$ in $\cL_5(\nu)$ contains an occurrence of the word $aa$ or $bb$ and these two words uniquely determine the relative cutting bars in $u$.
Then for all $i$, we have $x_{e_i} = x_i$, which shows that the constant of recognizability of $\nu$ is $L_\nu = 2$.
\end{exem}

The two following lemmas are consequences of Moss\'e's theorem (Theorem~\ref{theo:mosse}).

\begin{cor}[{\cite[Corollary 5.11]{Queffelec:2010}}]
\label{cor:mosse5.11}
Let $\sigma $ be an aperiodic primitive  substitution.
For all $n \geq 1$,
\begin{enumerate}
\item
$\sigma^n$ is a homeomorphism from $X_\sigma$ onto the clopen $\sigma^n(X_\sigma)$;
\item
for every $x \in X_\sigma$, there is a unique $y \in X_\sigma$ and a unique $k$, $0 \leq k < |\sigma^n(y_0)|$, such that $x = S^k(\sigma^n(y))$.
\end{enumerate}
\end{cor}

\begin{cor}
\label{cor:mosse}
Let $\sigma $ be an aperiodic primitive substitution and $x,y$ be two elements in $X_\sigma$.
Then,
\begin{enumerate}
\item
If there exists $n$ such that $\sigma (x) = S^n \sigma (y)$ then there exists $m$ such that $x = S^m (y)$. 
Moreover, if $n \geq 0$ then $0\leq m\leq n/\langle \sigma \rangle$;
\item
If  $\sigma^k (x)$ and $ \sigma^k (y)$ coincide on the indices $[-m , n]$, with $m,n \geq L_{\sigma^k}$, then $x$ and $y$ coincide on the indices $\left[\left\lceil \frac{-m+L_{\sigma^k}}{|\sigma^k |} \right\rceil,\left\lfloor \frac{n -L_{\sigma^k}}{|\sigma^k |} \right\rfloor\right]$.
\item
If $\sigma$ has constant length and is one-to-one on the letters, and if $\sigma^k (x)$ and $ \sigma^k (y)$ coincide on the indices $[-m , n]$, then $x$ and $y$ coincide on the indices $\left[\left\lceil \frac{-m}{|\sigma^k |} \right\rceil,\left\lfloor \frac{n}{|\sigma^k |} \right\rfloor\right]$.
\end{enumerate}
\end{cor}

\begin{cor}[{\cite[Proposition 5.20]{Queffelec:2010}}]
\label{lem:rec}
Let $\sigma : A^* \to A^*$ be an aperiodic primitive substitution with recognizability constant $L_\sigma$.
Then, $\{ S^i \sigma ([a]) \mid a \in A, 0\leq i < |\sigma(a) |\}$ is a clopen partition of $X_\sigma$.
Moreover, if $x,y \in X_\sigma $ coincide on the indices $[-n , n]$ with $n = L_\sigma + |\sigma |$, then they belong to the same atom of this partition.
\end{cor}

Later, we will not only need the existence of the recognizability constant but also a bound on its value depending on the data of the primitive substitution.

\begin{theo}[\cite{Durand&Leroy:2017}]
\label{theo:recognizability constant}
Let $\sigma :A^* \to A^*$ be an aperiodic primitive substitution that has an admissible fixed point $x \in A^\Z$. Then 
\[
L_\sigma \leq 	2 |\sigma|^{6(\card A)^2 + 6(\card A) |\sigma|^{28(\card A)^2}} + |\sigma|^{\card A}.
\]
Moreover, if $\sigma $ is one-to-one on letters then,
\[
L_\sigma \leq 2 |\sigma|^{6(\card A)^2 + 6 |\sigma|^{28(\card A)^2}} + |\sigma|. 
\]
\end{theo}

In~\cite{Klouda&Medkova:2016} better bounds are given for the case of constant length substitutions on a two letter alphabet but in terms of ``synchronization delay'' and ``circularity''. 
As the recognizability constant is bounded by the synchronization delay (see comments in~\cite{Durand&Leroy:2017}), this gives the following theorem that greatly improves our bounds.  

\begin{theo}[\cite{Klouda&Medkova:2016}]
If $\card A = 2$ and $\sigma : A^* \to A^*$ is an aperiodic primitive substitution of constant length $k$, then 
\begin{enumerate}
\item
$L_\sigma \leq 8$, if $k = 2$;
\item
$L_\sigma \leq k^2+3k-4$, if $k$ is an odd prime number;
\item
$L_\sigma \leq k^2 \left( \frac{k}{d}-1\right) +5k-4$, otherwise, where $d$ is the least divisor of $k$.
\end{enumerate}
\end{theo}

Later we will need to control the recognizability constant of $\sigma^k$ with respect to its growth.
\begin{prop}[{\cite[Proposition 13]{Durand&Leroy:2017}}]
\label{prop:constantrecsigmak}
If $\sigma:A^* \to A^*$ is an aperiodic primitive substitution that has an admissible fixed point $x \in A^\Z$. 
Then, for all $k>0$, so is $\sigma^k$ and we have
$$
L_{\sigma^k} \leq L_\sigma \sum_{i=0}^{k-1} |\sigma^i| .
$$ 
\end{prop}

\subsection{Block representations of subshifts and substitutions on the blocks of radius $n$}
\label{subsec:sublengthn}

Let $(X,S)$ be a subshift over the alphabet $A$.
For any positive integer $n$, we let $A^{(n)}$ denote the set $A^{(n)}=  \{ (w) \mid w\in  \mathcal{L}_{2n+1} (\sigma )\}$ considered as an alphabet.
The following result is classical~\cite{Queffelec:2010}.

\begin{lem}
\label{lemma:block representation isomorphic}
Let $\phi:A^{2n+1} \to \{ (w) \mid w\in  A^{2n+1}\}$ be the sliding block code defined by 
\[
\phi(a_1 \cdots a_{2n+1}) = (a_1 \cdots a_{2n+1}).
\]
Then $\phi$ defines an isomorphism $f$ from $(X,S)$ to $(X^{(n)},S)$, where $X^{(n)} = f(X)$ is a subshift over the alphabet $A^{(n)}$.
\end{lem}

We call $(X^{(n)},S)$ the {\em $n$-block representation} of $(X,S)$.
When the subshift $X$ of the previous result is a substitutive subshift associated with a substitution $\sigma$, then $X^{(n)}$ is also substitutive and associated with the {\em substitution on the words of length $2n+1$} of $\sigma$.
The definition we give below of the well-known notion of substitution on the words of length $n$ (see~\cite{Queffelec:2010} for more details) is slightly different from the usual one as we need to define it for a ``two-sided context''.
 
Let $\sigma : A^* \to A^*$ be a substitution and $n\geq 1$.
Below any time we take $(uv)$ in $A^{(n)}$, we suppose that $u$ belongs to $A^n$ and $v$ belongs to $A^{n+1}$.
We define the map $\pi^{(n)} : (A^{(n)})^* \to A^*$ by, for all $(uv) \in A^{(n)}$, $\pi^{(n)} ((u v))=v_0$, where $v = v_0 v_1 \dots v_n$, $v_i \in A$ for all $i$.
We also define the substitution $\sigma^{(n)} : (A^{(n)})^* \to (A^{(n)})^*$ by, 
for $(uv)\in A^{(n)}$ with $\sigma (u) = a_1 \cdots a_{k-1}$, $\sigma  (v) = a_{k} \dots a_{k+l}$ where the $a_i$'s belong to $A$ and $p= |\sigma (v_0)|$,
$$
\sigma^{(n)} ((uv)) = (a_{k-n} \cdots a_{k+n}) (a_{k-n+1} \cdots a_{k+n+1}) \cdots  (a_{k-n+p-1} \cdots a_{k+p+n-1}) .
$$

In other words, $\sigma^{(n)} ((u v))$ consists in the ordered list of the first $p$ factors of length $2n+1$ of $\sigma (uv)  = a_1 \cdots a_k a_{k+1} \cdots a_{k+l}$ starting from the index  $k-n$.
We say that $\sigma^{(n)}$ the {\em substitution on the blocks of radius $n$}.

\begin{exem}
Let us build the substitution $\gamma^{(5)}$ on the blocks of radius 2 associated with the Chacon morphism $\gamma$.
The words of length 5 in $\cL(\gamma)$ are
\[
	\cL_5(\gamma) = \{aabaa,abaaa,baaab,aaaba,aabab,ababa,babaa,abaab,baaba\}.
\]
We obtain
\[
\gamma^{(5)}:
\begin{cases}
(aabaa), (aabab),(babaa)  	& \mapsto  (babaa)\\
(abaaa),(abaab) 			& \mapsto  (abaab)(baaba)(aabab)(ababa) \\
(baaba),(aaaba) 			& \mapsto  (baaab)(aaaba)(aabab)(ababa) \\
(baaab) 					& \mapsto  (baaab)(aaaba)(aabaa)(abaaa) \\
(ababa) 					& \mapsto  (abaab)(baaba)(aabab)(ababa) 
\end{cases}.
\]
\end{exem}

\begin{prop}
\label{prop:conj-sublengthn}
Let $\sigma$ be a primitive substitution, then for all $n$, $\sigma^{(n)}$ is computable, primitive and such that $X_\sigma^{(n)} = X_{\sigma^{(n)}}$. 
In particular, for all $i$ and $n$, if $x$ is a two-sided admissible fixed point of $\sigma$ and $x^{(n)} = ((x_{[-n+m,n+m]}))_{m\in \mathbb{Z}}$, then $x^{(n)}$ is an admissible  fixed point of $\sigma^{(n)}$ and
\begin{equation}
\label{indice-reco2}
 \{ j \in \mathbb{Z} \mid S^j (x) \in \sigma^i (X_\sigma) \} =  \{ j \in \mathbb{Z} \mid S^j (x^{(n)}) \in (\sigma^{(n)})^i (X_{\sigma^{(n)}})  \} .
\end{equation}
Moreover, for all $n$ and $i$, $(\sigma^{(n)})^i = (\sigma^i)^{(n)}$ and $|\sigma^{(n)i} ((uv))| = |\sigma^i (v_0)|$ for all $uv\in A^{(n)}$ where $v=v_0 v_1 \cdots  v_n$.
In particular, the incidence matrices $M_\sigma$ and $M_{\sigma^{(n)}}$ have the same spectral radius.
\end{prop}

\begin{proof}
First, let us prove by induction that for every $i \geq 1$, 
\begin{equation} \label{eq:sigma^(n)^i=sigma^i^(n)}
(\sigma^{(n)})^i = (\sigma^i)^{(n)}.
\end{equation}
The base case is trivial so let us assume that the result holds for some $i \geq 1$. 
For any $(uv) \in A^{(n)}$, we have
\begin{eqnarray*}
	(\sigma^{(n)})^{i+1}((uv))
	&=&
	\sigma^{(n)}(\sigma^{(n)})^{i}((uv))	 \\
	&=& 
	\sigma^{(n)}(\sigma^i)^{(n)}((uv)),		
\end{eqnarray*}
where we have used the induction hypothesis.
Let us write
\begin{eqnarray*}
	\sigma^i (uv) &=& a_1 \cdots a_p, \qquad a_j \in A \text{ for all } j \\ 
	\sigma^{i+1} (uv) &=& b_1 \cdots b_q, \qquad b_j \in A \text{ for all } j.	
\end{eqnarray*}
By definition, we have, with the values $k_{i} = |\sigma^{i}(u)|+1$, $k_{i+1} = |\sigma^{i+1}(u)|+1$, $\ell_{i} = |\sigma^{i}(v_0)|-1$ and $\ell_{i+1} = |\sigma^{i+1}(v_0)|-1$,
\begin{eqnarray*}
	(\sigma^{i})^{(n)} ((uv))
	&=&
	(a_{k_i-n} \cdots a_{k_i+n})
	\cdots
	(a_{k_i-n+\ell_i} \cdots a_{k_i+n+\ell_i}); \\
	(\sigma^{i+1})^{(n)} ((uv))
	&=&
	(b_{k_{i+1}-n} \cdots b_{k_{i+1}+n})
	\cdots
	(b_{k_{i+1}-n+\ell_{i+1}} \cdots b_{k_{i+1}+n+\ell_{i+1}}).
\end{eqnarray*}
For all $m \in \{0,1,\dots,\ell_i\}$, if we set
\begin{eqnarray*}
	r_m &=& |\sigma(a_{1} \cdots a_{k_i-n+m-1})|		\\
	s_m &=& |\sigma(a_{k_i-n+m} \cdots a_{k_i+n+m})|	\\
	t_m &=& |\sigma(a_{k_i-n+m} \cdots a_{k_i+m-1})|+1	\\
	z_m &=& |\sigma(a_{k_i+m})|-1,	
\end{eqnarray*}
we get
\[
	\sigma(a_{k_i-n+m} \cdots a_{k_i+n+m})
	=
	b_{r_m+1} \cdots b_{r_m+s_m}
\]
and thus
\begin{multline*}
	\sigma^{(n)}((a_{k_i-n+m} \cdots a_{k_i+n+m})) 
	= 
	(b_{r_m+t_m-n} \cdots b_{r_m+t_m+n})
	\cdots \\
	(b_{r_m+t_m-n+z_m} \cdots b_{r_m+t_m+n+z_m}).
\end{multline*}
To finish the proof of~\eqref{eq:sigma^(n)^i=sigma^i^(n)}, it suffices to notice that we have
\begin{eqnarray*}
r_0 + t_0 &=& k_{i+1};	\\
r_m + t_m +z_m +1 &=& r_{m+1}+t_{m+1}, \quad \text{for all } m \in \{0,\dots,\ell_{i}-1\}; 	\\
r_{\ell_i}+t_{\ell_i}+z_{\ell_i} &=& k_{i+1}+ \ell_{i+1}.
\end{eqnarray*}

Using~\eqref{eq:sigma^(n)^i=sigma^i^(n)}, for the sake of simplicity we let $\sigma^{(n)i}$ denote the substitution $(\sigma^{(n)})^i$.
We deduce that for all $i$
\begin{eqnarray*}
\label{eq:sublongn}
\sigma^{(n)i} ((x_{[-n,n]})) 
&=& (x_{[-n,n]}) \cdots (x_{[|\sigma^i (x_0 )|-1 -n ,|\sigma^i (x_0 )| -1+n]}) \\
\sigma^{(n)i} ((x_{[-n-1,n-1]})) 
&=& (x_{[-|\sigma^i (x_{-1} )|-n ,-|\sigma^i (x_{-1})| +n]}) \cdots (x_{[-n-1,n-1]})
\end{eqnarray*}
and that $x^{(n)} = ((x_{[-n+i,n+i]}))_{i\in \mathbb{Z}}$ is an admissible  two-sided fixed point of $\sigma^{(n)}$. 
In particular, for all $(uv) \in A^{(n)}$, we have
\[
	|\sigma^{(n)i} ((uv))| = |\sigma^i (v_0)|,
\]
from which we deduce~\eqref{indice-reco2} and the fact that $M_\sigma$ and $M_{\sigma^{(n)}}$ have the same spectral radius.

Let us show that $\sigma^{(n)}$ is primitive.
If $i$ is large enough, every $w \in \mathcal{L}(\sigma)$ of length $2n+1$ is a factor of $\sigma^i (a)$ for every $a\in A$.
Thus, every $(uv) \in A^{(n)}$ occurs in $\sigma^{(n)i} ((u'v'))$ for every $(u'v')\in A^{(n)}$.

For the equality $X_\sigma^{(n)} = X_{\sigma^{(n)}}$, as $x^{(n)}$ is uniformly recurrent, we hence have $X_\sigma^{(n)} = \Omega(x^{(n)})$.
Furthermore,  as $x^{(n)}$ is also an admissible fixed point of the primitive substitution $\sigma^{(n)}$, we have $X_{\sigma^{(n)}} = \Omega(x^{(n)})$.
\end{proof}


\subsection{Computable bounds for linear recurrence of substitutive sequences}

We now develop some lemmas that allow to relate the quantities $|\sigma^n|$ and $\langle \sigma^n \rangle$ to the spectral radius of $M_\sigma^n$.
First recall that if $(X,T)$ is a dynamical system, an {\em invariant measure} for $(X,T)$ is a probability measure $\mu$ on the $\sigma$-algebra of Borel sets such that $\mu(T^{-1}B) = \mu(B)$ for all Borel sets $B$. 
Such a measure is {\em ergodic} if every $T$-invariant Borel set has measure $0$ or $1$.
The system $(X,T)$ is said to be {\em uniquely ergodic} if there exists a unique invariant measure (which is then ergodic).
Let us recall the ergodic theorem for such systems: Let $\mu$ be the unique $T$-invariant ergodic measure of $(X,T)$ and $f:X \to \mathbb{R}$ be a continuous function, then  $((1/n) \sum_{i=0}^{n-1} f \circ T^i )_n $ converges uniformly to $\int_X f {\rm d}\mu$. 
We refer the reader to~\cite{Hasselblatt&Katok:1995} for more details on ergodicity.

\begin{prop}[{\cite[Proposition 13]{Durand:2000}}]
\label{prop:lrmeasure}
If $(X,S)$ is an aperiodic linearly recurrent subshift, then it is uniquely ergodic.
Furthermore, if $\mu$ denotes the unique invariant measure of $(X,S)$ and if $K$ is a constant of linear recurrence for $X$, we have for all non-empty $u \in \mathcal{L}(X)$,
\[
	1/K \leq |u| \mu([u]) \leq K.
\]
\end{prop}

Any substitution subshift $(X_\sigma,S)$, where $\sigma $ is primitive, has a unique ergodic measure $\mu$~\cite{Queffelec:2010}.
We denote by $\mu^{(\ell)}$ the vector whose entries are $\mu ([ u] ) $, $u \in \mathcal{L}_{2 \ell+1} (\sigma )$.

As we already said above, our definition of $n$-block substitutions is slightly different from the usual one in~\cite{Queffelec:2010}, nevertheless
the  following result remains true for our definition.

\begin{lem}[{\cite[Corollary 5.4 and Proposition 5.10]{Queffelec:2010}}]
\label{lemma:muk}
Let $\sigma:A^* \to A^*$ be an aperiodic primitive substitution and let $\mu$ be the unique invariant measure of $(X_\sigma,S)$.
The matrix $M_{\sigma^{(\ell)}}$ has the same dominant eigenvalue as $M_{\sigma}$.
Furthermore, the vector $\mu^{(\ell)}$ is a positive eigenvector of $M_{\sigma^{(\ell)}}$ associated with this dominant eigenvalue.
\end{lem}

\begin{lem}
\label{lemma:maxmin}
Let $\sigma:A^* \to A^*$ be an aperiodic and primitive substitution.
There exists a positive integer $k \leq \card(A)^2$ such that for all $n \geq k$, $\langle \sigma^n \rangle \geq |\sigma^{n-k}|$.
Then for all $n \geq k$, one as
\begin{align*}
    \rho(M_\sigma)^{n-k}
    &\leq \langle \sigma^n \rangle 
    \leq |\sigma^n|
    \leq K_\sigma \rho(M_\sigma)^{n},
    \\
    |\sigma^n| 
    &\leq \rho(M_\sigma)^k K_\sigma \langle \sigma^n \rangle,
\end{align*}
where $\rho(M_\sigma)$ denotes the spectral radius of $M_\sigma$ and $K_\sigma$ is given by Equation~\eqref{eq:Ksigma}.
Moreover, if $M_\sigma$ has positive entries then one can take $k=1$.
\end{lem}

\begin{proof}
Since $\sigma$ is primitive, there exists $k \leq \card(A)^2$ such that $M_\sigma^k$ contains only positive entries (see~\cite[Corollary 8.5.8]{Horn&Johnson:1990}), which trivially implies $\langle \sigma^n \rangle \geq |\sigma^{n-k}|$.

By the previous lemma, the vector $\mu^{(0)}$ is a positive eigenvector of $M_\sigma$ associated with the eigenvalue $\rho(M_\sigma)$.
Furthermore, it satisfies $\|\mu^{(0)}\|_1 = \sum_{a \in A} \mu([a]) = 1$.
For all $n$, we thus have 
\[
    \sum_{a \in A} |\sigma^n(a)| \mu([a]) 
    =
    \|M_\sigma^n \mu^{(0)}\|_1 = \rho(M_\sigma)^n.
\]
Using Proposition~\ref{prop:lrmeasure}, we deduce that $|\sigma^n|\leq K_\sigma \rho(M_\sigma)^{n}$.
Since for all $n$, one also has $\rho(M_\sigma)^n \leq |\sigma^n|$, we deduce from the definition of $k$ that for all $n \geq k$, $\langle \sigma^n \rangle \geq \rho(M_\sigma)^{n-k}$.
\end{proof}

\section{Bounding the radius of a factor map between primitive substitution subshifts}\label{section:bounds}

In this section we show that for aperiodic primitive proper substitutions $\sigma$ and $\tau$, the radius of a factor map from $X_\sigma$ to $X_\tau$ can be chosen smaller than a bound only depending on $\sigma$ and $\tau$ and that one can algorithmically compute.

\subsection{Orbit-preserving maps, cocycles and dill maps}

Let $(X_1,T_1)$ and $(X_2,T_2)$ be two topological dynamical systems.
A map $f : X_1 \to X_2$ is said to be \emph{orbit-preserving} if $f(\mathcal{O}_{T_1}(x)) \subset \mathcal{O}_{T_2}(f(x))$ holds for all $x \in X_1$, where $\mathcal{O}_{T_1} (x)$ stands for $\{ T_1^n (x) \mid n\in \mathbb{Z} \}$.
Observe that, in this case, $f$ satisfies 
\begin{equation}
\label{eq:cocycle}
f(T_1(x)) = T_2^{c(x)}(f(x)),
\end{equation}
for all $x \in X_1$, for some map $c : X_1 \to \mathbb{Z}$ we call a {\em cocycle} for $f$.
If $f$ and $c$ are continuous and if for all $x$, $c(x)\geq 0$ and there exists some $y$ such that $c(y)>0$, then we say $f$ is a {\em dill map}.
Observe that if $f$ is a dill map, then the associated cocycle $c$ is bounded.
Furthermore, when $(X_2 , T_2)$ has no periodic point, $f$ is a factor map if and only if $c(x)=1$ for all $x$.
Other examples of dill maps are given by morphisms $\Theta : A^* \to B^*$. 
Indeed, take a minimal subshift $(X,S)$ generated by a sequence $x\in A^{\mathbb{Z}}$ and let $(Y,S)$ be generated by $y=\Theta (x)\in B^{\mathbb{Z}}$. 
Then $\Theta$ defines a dill map from $X$ to $Y$.

In the definition of dill maps, the existence of $y$ such that $c(y)>0$ is only required to avoid degenerated cases such as trivial maps satisfying $f(\mathcal{O}_{T_1}(x)) = \{f(x)\}$ for all $x$.
The next lemma shows that the existence of the point $y$ is just a consequence of the continuity of $f$.

\begin{lem}
\label{lem:simplify definition of dill map}
Let $(X_1,T_1)$ and $(X_2,T_2)$ two topological dynamical systems such that $(X_1,T_1)$ is minimal.
If $f : X_1 \to X_2$ is a continuous orbit-preserving map such that $f(X_1)$ is not a singleton and the cocycle $c:X_1 \to \mathbb{Z}$ associated with $f$ is continuous and non-negative, then $f$ is a dill map.
\end{lem}
\begin{proof}
Let assume that $c(x) = 0$ for all $x\in X_1$.
Then, $f(T_1(x)) = f(x)$ for all $x\in X_1$. 
Hence, $f(T_1^n (x))  =f(x)$ for all $n\in \mathbb{Z}$.
Since $\mathcal{O}_{T_1}( x) $ is dense in $X_1$ and $f$ is continuous, we get that $f(X_1)$ is a singleton, which contradicts the hypothesis.
\end{proof}

The notion of dill map has been introduced by Salo and T\"orm\"a~\cite{Salo&Torma:2015} in a one-sided context.

Suppose $(X,S)$ and $(Y,S)$ are subshifts and $f : (X,S) \to (Y,S)$ is a dill map. 
A continuous function $\phi : X \to A^*$ satisfying
\begin{equation}
f (x) = \cdots \phi(S^{-1} (x)) . \phi (x) \phi(S(x)) \phi(S^2(x)) \cdots
\end{equation}
for all $x \in X$ is called an {\em implementation of $f$}.
Observe that $\phi (x)$ could be the empty word.

\begin{lem}[{\cite[Lemma 8]{Salo&Torma:2015}}]
\label{lemme:implementationdill}
Let $(X,S)$ and $(Y,S)$ be subshifts and $f : (X,S) \to (Y,S)$ a dill map with cocycle $c$. 
Then, $f$ has an implementation $\phi$ given by 
$$
\phi(x) = f (x)_{[0,c(x)-1]} \hbox{ for all } x \in X .
$$
\end{lem}
Notice that $|\phi (x)| = c(x)$ and that $\phi (x)$ is the empty word whenever $c (x)=0$.
It is straightforward to check (as observed in~\cite{Salo&Torma:2015}) that if $Y$ does not contain $S$-periodic points then $c$ and $\phi$ are unique.
As we will deal with infinite minimal subshifts, there will be no $S$-periodic points in $Y$, and, thus, both maps will be unique.

Let $f : (X,S) \to (Y,S)$ be a dill map between subshifts without $S$-periodic points and with implementation map $\phi$.
The map $c = |\phi|$ is the unique cocycle of $f$. 
By compactness of $X$ and continuity of $c$, the set $c(X)$ is finite. 
Using the continuity of $f$, we can thus define 
\begin{align*}
r & = \min \{ R \in \mathbb{N} \mid  x_{[-R, R]} = y_{[-R, R]} \implies \phi(x) = \phi(y) \} 
\hbox{ and } \\
m & = \max \{ |\phi(x)| \mid x \in X \} .
\end{align*}
The couple $(r,m)$ is a called the {\em radius pair of $f$}.
It is clear that $r$ is a radius for $\phi$ and $c$.
If $f$ is a factor map, then its radius pair is $(r,1)$ and $\phi$ is a sliding block code defining $f$. 

Let $f$ and $g$ be two continuous maps from $(X_1,T_1)$ to $(X_2,T_2)$.
We say $f$ and $g$ are {\em orbitally related} (and we write $f \equiv g$) whenever there exists a continuous map $c : X_1\to \mathbb{Z}$ satisfying, for all $x\in X_1$,
\begin{align}
\label{rel:orbitalyrelated}
T_2^{c (x)}\circ f (x) = g (x) . 
\end{align}
It is clear that the orbit relation is an equivalence relation.




\begin{lem}
Let $f$ and $g$ be two orbitally related factor maps from $(X_1,T_1)$ to $(X_2,T_2)$.
Suppose $(X_1,T_1)$ is minimal and $(X_2,T_2)$ is aperiodic.
Then, there exists $e\in \mathbb{Z}$ such that $f = T_2^e \circ g$.
\end{lem}
\begin{proof}
Let $c : X_1\to \mathbb{Z}$ be a continuous map satisfying, for all $x\in X_1$,
\[
	T_2^{c (x)}\circ f (x) = g (x). 
\]
Therefore 
\[
 T_2^{c (T_1 (x))}\circ f (T_1 (x)) = g (T_1 (x)) = T_2 \circ g (x) = T_2^{c (x)}\circ f (T_1 (x)).
\]
As $(X_2 , T_2)$ is aperiodic, one gets $c (T_1 (x)) = c (x)$ for all $x$. 
From the minimality of $(X_1 , T_1)$ and by continuity of $c$, one concludes that $c$ is a constant function.
\end{proof}

\subsection{Specific dill maps associated with factor maps}

%
%
%

Let $\sigma : A^* \to A^*$ and $\tau : B^* \to B^*$ be two aperiodic primitive substitutions.
We let $K_\sigma$ and $K_\tau$ denote the constant given by \eqref{eq:Ksigma}.
In this section, starting with a factor map $f : (X_\sigma , S) \to (X_\tau , S)$, we build pairwise orbitally related dill maps $F_n : (X_\sigma , S) \to (X_\tau , S)$, $n \in \mathbb{N}$.
We show that for $n$ large enough, these dill maps have bounded radius pairs and so that there exist only finitely many such maps.
In Section~\ref{section:smaller radius}, we will use these dill maps to build a new factor map from $(X_\sigma , S)$ to $(X_\tau , S)$ with a radius only depending on $\sigma$ and $\tau$, not on the radius of $f$. 

Let us thus consider a factor map $f : (X_\sigma , S) \to (X_\tau , S)$ with radius $r$.
We recall a Cobham-type theorem proved in~\cite{Durand:1998c}.

\begin{theo}
\label{theo:cobham}
If $(X_\tau , S)$ is aperiodic then the dominant eigenvalues of $\sigma $ and $\tau$ share a non-trivial common power.
\end{theo}

Consequently, we can suppose that $\sigma$ and $\tau$ have the same dominant eigenvalue $\lambda$ (taking powers of the substitutions does not change the subshifts).
We also assume that $\sigma$ and $\tau$ are {\em positive}, i.e., their incidence matrices have only positive entries.
Observe that this implies that $\lambda \geq 2$.
We finally let $L_\tau$ denote the constant of recognizability of $\tau$.

For all $n\geq 0$ we define the map $F_n : X_\sigma  \to X_\tau$  by 
\begin{align}
\label{align:Fn}
\tau^n \circ F_n (x) = S^{r_n\circ f \circ \sigma^{n}(x)} \circ f \circ \sigma^{n} (x), x\in X_\sigma, 
\end{align}
where $r_n : X_\tau \to \mathbb{N}$ is the return map to the clopen set $\tau^n (X_\tau )$ (see Corollary~\ref{cor:mosse5.11}): 
$$
r_n(x) = \min \{ i \geq 0 \mid S^i (x) \in \tau^n (X_\tau ) \}.
$$

Thus, we have $\tau^n  \circ F_n \equiv  f \circ \sigma^{n} $. 
Observe that, using Moss\'e's theorem (Theorem~\ref{theo:mosse}), $r_n$ is continuous and $F_n$ is well-defined.

Now let us show that $F_n$ is a dill map and that, for $n$ large enough, its radius pair is bounded by a computable bound only depending on $\sigma$ and $\tau$.

\begin{prop}
\label{lemme:boundradius}
For all $n$, $F_n$ is a dill map.
Furthermore, if $n$ is such that $r/|\tau^n | \leq 1$, the radius pair $(s_1 , s_2)$ of $F_n$ satisfies
\[
	s_1 \leq \lceil K_\tau \lambda (\lfloor (K_\tau+K_\sigma) \lambda \rfloor  + L_\tau +2) \rceil 
	\quad \text{ and } \quad
	s_2 \leq \lfloor (K_\tau+K_\sigma) \lambda \rfloor.
\]
In particular, we have 
\[
\#\{F_n \mid r/|\tau^n | \leq 1\}
\leq 
\left(
    K_\tau \lambda^2
    (K_\sigma+K_\tau)
    \right)^{2 K_\sigma K_\tau \lambda^2 (K_\tau+K_\sigma + L_\tau +3 )}.
\]
\end{prop}

\begin{proof}
Let $n$ be a positive integer.
By Corollary~\ref{cor:mosse5.11}, $F_n$ is a continuous map.
We have:
\begin{align}
\label{eq:tau_nF_nSx}
\tau^n \circ F_n(S (x)) 
& =  S^{r_n\circ f \circ \sigma^{n} (S(x))} \circ f \circ \sigma^{n} (S(x)) 
\nonumber \\
& =  S^{r_n\circ f \circ S^{|\sigma^{n} (x_0)|}  \circ \sigma^{n} (x)} \circ f \circ S^{|\sigma^{n} (x_0)|} \circ \sigma^{n} (x) 
\nonumber \\
& =  S^{r_n\circ f \circ S^{|\sigma^{n} (x_0)|}  \circ \sigma^{n} (x)+|\sigma^{n} (x_0)|   } \circ f  \circ \sigma^{n} (x)
\nonumber \\
& =  S^{r_n\circ S^{|\sigma^{n} (x_0)|}  \circ f \circ  \sigma^{n} (x)+|\sigma^{n} (x_0)| -r_n\circ f \circ \sigma^{n} (x)  } \circ  \tau^n \circ F_n (x).
\end{align}
The map 
\[
x \in X_\sigma \mapsto r_n\circ S^{|\sigma^{n} (x_0)|}  \circ f \circ  \sigma^{n} (x)+|\sigma^{n} (x_0)| -r_n\circ f \circ \sigma^{n} (x)
\]
is continuous and $F_n(X_\sigma)$ is obviously not a singleton as $(X_\tau , S)$ is aperiodic.
Therefore, from Corollary~\ref{cor:mosse} and Lemma~\ref{lem:simplify definition of dill map}, to prove that $F_n$ is a dill map it suffices to show that 
\[
r_n \circ S^{|\sigma^{n} (x_0)|} \circ f  \circ \sigma^{n} (x)+|\sigma^{n} (x_0)| -r_n\circ f \circ \sigma^{n} (x)
\]
is non-negative for all $x$.
This is clear if $r_n\circ f \circ \sigma^{n} (x) \leq |\sigma^{n} (x_0)|$. 

Suppose $r_n\circ f \circ \sigma^{n} (x) > |\sigma^{n} (x_0)|$.
This means that $S^i \circ  f \circ \sigma^{n} (x)$ does not belong to $\tau^n (X_\tau)$ for all $i$ such that $0 \leq i \leq |\sigma^{n} (x_0)|$.
Hence $r_n\circ S^{|\sigma^{n} (x_0)|}\circ f \circ \sigma^{n} (x)+|\sigma^{n} (x_0)| =r_n\circ f \circ \sigma^{n} (x) $ and $F_n$ is a dill map.

Now suppose that $n \in \mathbb{N}$ is such that $r/|\tau^n | \leq 1$.
We let $c : X_\sigma \to \mathbb{N}$ and $\phi:X_\sigma \to B^*$ respectively denote the cocycle and the implementation of $F_n$.
Let us also denote $(s_1,s_2)$ the radius pair of $F_n$.
From Equation~\eqref{eq:tau_nF_nSx} and Corollary~\ref{cor:mosse}, we get 
\begin{align*}
c(x) & \leq \left(r_n \circ S^{|\sigma^{n} (x_0)|} \circ f  \circ \sigma^{n} (x)+|\sigma^{n} (x_0)| -r_n\circ f \circ \sigma^{n} (x)\right)/\langle\tau^n \rangle \\
&  \leq (|\tau^{n} |+|\sigma^{n}|)/  \langle\tau^n \rangle.
\end{align*}
Thus, using the fact that $\sigma$ and $\tau$ are positive and have the same dominant eigenvalue $\lambda$, we deduce from Lemma~\ref{lemma:maxmin} that 
\[
s_2 = \max\{c(x) \mid x \in X_\sigma\} \leq \lfloor (K_\sigma+K_\tau)\lambda \rfloor.
\]
Now let 
$s = \lceil K_\tau \lambda (\lfloor (K_\tau+K_\sigma) \lambda \rfloor  + L_\tau +2 ) \rceil$ 
and let us show that $\phi (x)$ only depends on $x_{[-s,s]}$.
This will imply that $s_1$ is at most equal to $s$.
Consider $x, y \in X_\sigma$ such that $x_{[-s,s]}  = y_{[-s,s]}$. 
Then $S^{r_n\circ f \circ \sigma^{n}} \circ f \circ \sigma^{n} (x)$ and $S^{r_n\circ f \circ \sigma^{n}} \circ f \circ \sigma^{n} (y)$ coincide on the indices $[-s \langle  \sigma^n \rangle +r - k , s \langle  \sigma^n \rangle -r - k]$ for some integer $k$ satisfying $0 \leq k < |\tau^n|$.
From Corollary~\ref{cor:mosse}, $F_n(x)$ and $F_n(y)$ coincide on the indices 
$$
\left[
\left\lceil
\frac{-s \langle  \sigma^n \rangle +r -k +L_{\tau^n}}{|\tau^n|} 
\right\rceil 
,
\left\lfloor
\frac{s \langle  \sigma^n \rangle -r -k - L_{\tau^n}}{|\tau^n|}
\right\rfloor
\right]
$$
and, since $0\leq k,r \leq |\tau^n|$, on the indices
\[
\left[
\left\lceil
\frac{-s \langle  \sigma^n \rangle +L_{\tau^n}}{|\tau^n|}+1  
\right\rceil 
,
\left\lfloor
\frac{s \langle  \sigma^n \rangle - L_{\tau^n}}{|\tau^n|}-2,
\right\rfloor
\right],
\]
where $L_{\tau^n}$ is a recognizability constant of $\tau^n$. 
Recalling that $\lambda \geq 2$ and that $|\tau^n| \leq K_\tau \lambda^n$ by Lemma~\ref{lemma:maxmin}, we get from Proposition~\ref{prop:constantrecsigmak} that 
$$
L_{\tau^n} \leq K_\tau L_\tau \lambda^n.
$$ 
Hence, $F_n(x)$ and $F_n(y)$ coincide on the indices 
$$
\left[
\left\lceil
\frac{-s \langle  \sigma^n \rangle }{K_\tau \lambda^n} + L_\tau +1
\right\rceil
,
\left\lfloor 
\frac{s \langle  \sigma^n \rangle}{K_\tau \lambda^n}- L_\tau -2
\right\rfloor
\right] .
$$
Then, it suffices to observe using Lemma~\ref{lemma:maxmin} that 
\begin{align*}
\frac{s \langle  \sigma^n \rangle}{K_\tau \lambda^n}  - L_\tau -2 
& \geq s_2,    \\
\frac{-s \langle  \sigma^n \rangle }{K_\tau \lambda^n} +L_\tau +1 
& \leq 0,
\end{align*}
which shows that $s_1 \leq s$.

To conclude the proof, the number of possible dill maps $F_n$ with $r/|\tau^n| \leq 1$ is at most equal to the number of maps from $\mathcal{L}_s(\sigma)$ to $\bigcup_{i=0}^{\lfloor (K_\sigma+K_\tau)\lambda \rfloor } \mathcal{L}_i(\tau)$.
We get, using Theorem~\ref{theo:encad},
\begin{align*}
	\#\{F_n \mid r/|\tau^n | \leq 1\} 
	&\leq 
	\left(
	\sum_{i=0}^{\lfloor (K_\sigma+K_\tau)\lambda \rfloor} p_\tau(i)
	\right)^{p_\sigma (s)} 
	\\
	&\leq 
    \left(
    K_\tau \lambda^2
    (K_\sigma+K_\tau)
    \right)^{2 K_\sigma K_\tau \lambda^2 (K_\tau+K_\sigma + L_\tau +3 )}
\end{align*}
\end{proof}

\begin{cor}
\label{coro:Fnfinite}
The set $\{ F_n \mid n\in \mathbb{N} \}$ is finite.
In particular, there exists an increasing sequence $(n_i)_{i \in \mathbb{N}}$ of integers such $F_{n_i} = F_{n_0}$ for all $i$.
\end{cor}

The following lemma will provide a key argument in the proof of Theorem \ref{theo:mainpropersub} which is a detailed version of Theorem \ref{theo:main}. 

\begin{lem}
\label{lem:equiv}
If $F_n \equiv F_m$ for some $n\not = m$, then $F_{n-1} \equiv F_{m-1}$.
\end{lem}
 \begin{proof}
Suppose $F_n \equiv F_m$.
 As $\tau^{n-1} \circ F_{n-1} \equiv f\circ \sigma^{n-1}$, one gets $\tau^{n-1} \circ F_{n-1}\circ \sigma  \equiv f\circ \sigma^{n}\equiv \tau^n \circ \ F_n$.
 Thus, from Corollary~\ref{cor:mosse}, $F_{n-1}\circ \sigma  \equiv \tau \circ  F_n$.
 We also have $F_{m-1}\circ \sigma  \equiv \tau \circ \ F_m$.
 Hence, $F_{n-1}\circ \sigma  \equiv F_{m-1}\circ \sigma $ from which we deduce, $F_{n-1}$ and $F_{m-1}$ being dill maps, that  $F_{n-1} \equiv F_{m-1}$. 
\end{proof}

\subsection{Factorizations with smaller radius using dill maps}
\label{section:smaller radius}

In this section, we show that if $\sigma$ is proper and if there exists a factor map $f:(X_\sigma,S) \to (X_\tau,S)$, then there exists another factor map $g:(X_\sigma,S) \to (X_\tau,S)$ whose radius is bounded by some computable constant depending on $\sigma$ and $\tau$.
Observe that if $\sigma:A^* \to A^*$ is proper and if $a,b \in A$ are such that $\sigma(A) \subset aA^* \cap A^*b$, then $x = \sigma^\omega(b \cdot a)$ is the unique admissible fixed point of $\sigma$ and we have $\{x\} = \bigcap_{n \in \mathbb{N}} \sigma^n(X_\sigma)$.

Let $(X,T)$ be a minimal dynamical system. 
A continuous function $c:X \to \mathbb{R}$ is a {\em coboundary} of $(X,T)$ if there is some continuous function $d:X \to \mathbb{R}$ such that $c = d \circ T - d$.

Let $c:X \to \mathbb{R}$ be a continuous function.
For $n \in \mathbb{N}$, we write
\begin{equation}
\label{eq:partialsums}
	c^{(n)} = 
	\begin{cases}
		0 								& 	\text{if } n = 0,	\\
		\sum_{k=0}^{n-1} c \circ T^k 	&	\text{if } n>0 , \\
		\sum_{k=n}^{-1} c \circ T^k 	&	\text{if } n<0 , 
	\end{cases}
\end{equation}

We remind the following useful identity: $c^{(n+m)} (x) = c^{(n)}(x) + c^{(m)} (T^n (x))$ and recall the following well-known result of Gottschalk and Hedlund~\cite{Gottschalk&Hedlund:1955}. 
A proof can be found in~\cite{Hasselblatt&Katok:1995} (Theorem 2.9.4 p. 102).

\begin{theo}
\label{lemma:gott-hed}
Let $(X,T)$ be a minimal dynamical system and $c:X \to \mathbb{R}$ be a continuous function.
The following are equivalent.
\begin{enumerate}
\item
$c$ is a coboundary;
\item
The sequence of functions $(c^{(n)})_{n \in \N}$ is uniformly bounded;
\item
The sequence $(c^{(n)}(x))_{n \in \N}$ is bounded for some $x \in X$.
\end{enumerate}
\end{theo}

Let us show this result in a particular case but obtaining more information.
Let $(X,S)$ be a subshift and $c : X \to \mathbb{Z}$ be a continuous map of radius $r$ defined by $\hat{c}$:  $c$ is constant, and equal to $\hat{c} (uv)$, on any cylinder $[u.v]$ where $|u|=r $ and $|v|=r+1$.
In the sequel we identify $\hat{c}$ with the vector $ (\hat{c} (uv))_{uv\in \mathcal{L}_{2r+1} (\sigma )}$. 

\begin{lem}
\label{lem:cobound}
Let $\sigma:A^* \to A^*$ be a primitive, aperiodic and proper substitution, and let $x \in A^\mathbb{Z}$ be its admissible fixed point.
Suppose that $c: X_\sigma \to \mathbb{Z}$ is a continuous map with radius $r$, defined by $\hat{c}$, such that $(c^{(n)})_n$ is bounded.
We have the following.
\begin{enumerate}
\item 
\label{cobord_item_2}
There exists $k>0$ such that $\sum_{i=e_1}^{e_2-1} c(S^i (x)) = 0$ for all $e_1,e_2 \in E (\sigma^k ,x  )$, where $E (\sigma^k ,x  ) = \{ i \in \mathbb{Z} \mid S^i (x) \in \sigma^k (X_\sigma)  \}$.
\item 
\label{cobord_item_3}
 $\hat{c } $ belongs to ${\rm Ker } \left( (M_{\sigma^{(r)} }^t)^{p_\sigma (2r+1)} \right)$, where $M^t$ denotes the transposed matrix of $M$ and $p_\sigma$ denotes the factor complexity function of $X_\sigma$.
\item
\label{cobord_item_4}
In \eqref{cobord_item_2}, $k$ can be chosen equal to $p_{\sigma }(2r+1)$.
\item 
\label{cobord_item_1}
$c = d \circ S  -d $ for some continuous map $d:X \to \mathbb{Z}$ having a radius bounded by 
\[
|\sigma^{p_{\sigma }(2r+1)}|+\max\{L_{\sigma^{p_{\sigma }(2r+1)}},r\}
\]
and such that 
\[
\max_{x \in X_\sigma}|d(x)| \leq |\sigma^{p_{\sigma }(2r+1)}| \max_{x \in X_\sigma} |c(x)|.
\]
\end{enumerate}
\end{lem}

\begin{proof}
Let $U_k = \sigma^k (X_\sigma)$. 
Since the admissible fixed point $x$ of $\sigma$ satisfies $\{ x \} = \cap_{k\in \mathbb{N}} U_k$, for any neighborhood $V$ of $x$ there exists $k$ such that $U_k$ is included in $V$. 
We set
$$
\Lambda_k = \{c^{(n)}(x) \mid n \in \mathbb{N}, S^n(x)\in U_k\} \hbox{  and  } \Lambda = \cap_{k\in \mathbb{N}} \Lambda_k  .
$$
Obviously $0$ belongs to  $\Lambda $.
Let us show that $\Lambda = \{ 0 \}$.
The map $c$ being continuous and thus bounded, it is sufficient to prove that $2a$ belongs to $\Lambda $ whenever $a$ belongs to $\Lambda$. 

Let $a\in \Lambda$ and  $k \in \mathbb{N}$. 
There exists $n\geq 0$ such that $S^n (x) $ belongs to $U_k$ and $c^{(n)} (x) = a$.
Because $S^n$ and $c^{(n)}$ are continuous, there exists a clopen set $V$ included in $U_k$ and containing $x$ such that for all $y \in V$,
$$
S^n (y) \in U_k \hbox{  and  } c^{(n)} (y) = a.
$$
Observe that there exists $l$ such that $U_l$ is included in $V$. 
Thus, one can suppose $V = U_l$.
As $a$ belongs to $\Lambda_l$, there exists $m\geq 0$ with $S^m (x)$ in $U_l$ and $c^{(m)} (x) = a$.  
We get that $S^{n+m} (x)$ belongs to $U_k$ and $c^{(n+m)} (x) = c^{(n)} (S^m (x)) +  c^{(m)} (x)  = 2a$.
Thus $2a $ belongs to $\Lambda_k$.
As this is true for all $k$, we get that $2a$ belongs to $\Lambda $.

\medskip
Observe that, since $c$ takes integer values and $\Lambda_{l+1} \subset \Lambda_l$ for all $l$, there exists $k$ such that $\Lambda_l = \{ 0\}$ for all $l\geq k$.
Thus \eqref{cobord_item_2} is proved.

\medskip

Consider $\sigma^{(r)} : A^{(r)*} \to A^{(r)*}$ the substitution on the blocks of radius $r$ defined in Section~\ref{subsec:sublengthn} and the notation used therein.
Let $e_1$ and $e_2$ be two consecutive integers in $E (\sigma^k , x)$ and $uv\in \mathcal{L} (\sigma )$ with $|u|=r$ and $|v| = r+1$.
The number of occurrences of $uv$ in $x_{[e_1-r , e_2-1+r]}$ is equal to the number of occurrences of $(uv)$ in 
$$
W= (x_{[e_1-r,e_1+r]}) (x_{[e_1-r+1,e_1+r+1]}) \cdots (x_{[e_2-1-r,e_2-1+r]}) .
$$
Then, as $c^{(e_1)} (x) = c^{(e_2)} (x) = 0$, one gets
\begin{align}
\label{eq:cobordzero1}
0=\sum_{i=e_1}^{e_2-1} c(S^i (x)) = 
\sum_{i=e_1}^{e_2-1} \hat{c} (x_{[i-r,i+r]} )
\end{align}
Observe that we can write
\begin{align}
\label{eq:cobordzero2}
\sum_{i=e_1}^{e_2-1} \hat{c}  (x_{[i-r,i+r]} )
= \langle \hat{  c } , \overrightarrow{W} \rangle,
\end{align}
where $\overrightarrow{W}=(|W|_{w})_{w\in \mathcal{L}_{2r+1} (\sigma )}$,  $\hat{c} = (\hat{c}  (w))_{w\in \mathcal{L}_{2r+1} (\sigma )}$ and $\langle \cdot , \cdot \rangle$ stands for the usual scalar product. 

From Proposition~\ref{prop:conj-sublengthn}, we have 
\begin{align*}
E (\sigma^k ,x  ) = E (\sigma^{(r)k} ,x^{(r)}  ) .
\end{align*}
Hence, $e_1$ and $e_2$ also belong to $E (\sigma^{(r)k} ,x^{(r)}  )$ and $W = \sigma^{(r)k} (\beta )$ for some $\beta \in A^{(r)}$.

Let ${\bf 1}_\beta$ be the column vector defined by ${\bf 1}_\beta (\beta )= 1$ and $0$ elsewhere. 
From the conjunction of \eqref{eq:cobordzero1} and \eqref{eq:cobordzero2} we obtain
\begin{align*}
 0= \langle \hat{ c } , M_{\sigma^{(r)}}^k {\bf 1}_\beta  \rangle .
\end{align*}
As this is true for all such $e_1$ and $e_2$, it is true for all $\beta \in A^{(r)}$ and one gets $ {\hat{ c }}^t M_{\sigma^{(r)} }^k  = 0$.
The dimension of $M_{\sigma_r }$ is $p_\sigma (2r+1)\times p_\sigma (2r+1)$, hence $\hat{ c}^t M_{\sigma^{(r)} }^{p_\sigma (2r+1)} = 0$ and \eqref{cobord_item_3} is proved.
Tracing back the previous arguments shows that for all $e_1, e_2\in E (\sigma^{p_\sigma (2r+1)} ,x  )$ one gets
\begin{align*}
0=\sum_{i=e_1}^{e_2-1} c(S^i (x))
\end{align*}
and that $k$ could be chosen to be $p_\sigma (2r+1)$. 
This proves \eqref{cobord_item_4} and we set $k = p_\sigma (2r+1)$.

\medskip

Let us find a continuous map $d: X \to \mathbb{Z}$ such that $c = d\circ S - d$.

For all $n \in \mathbb{N}$, we set $d(S^n(x)) = c^{(n)}(x)$.
On the positive orbit $\mathcal{O}_+(x) = \{S^n(x) \mid n \in \mathbb{N}\}$, we trivially have $c = d\circ S - d$.
Let us show that $d$ is uniformly continuous on $\mathcal{O}_+(x)$.
By density of $\mathcal{O}_+(x)$ in $X$, the function will then uniquely extend to a continuous function $d$ on $X$ still satisfying $c = d\circ S - d$.

For all $n \in \mathbb{N}$, there exists $\ell \in \{0,1,\dots,|\sigma^k|\}$ such that $S^{n+\ell}(x) \in U_k$.
Observe that all functions $c^{(\ell)}$ and $S^\ell$, $\ell \in \{0,1,\dots,|\sigma^k|\}$, are uniformly continuous on $X$.
Thus there exists $L > 0$ such that for all $m,n \in \mathbb{N}$ such that $S^m(x)$ and $S^n(x)$ coincide on the indices $[-L,L]$, there exists $\ell \in \{0,1,\dots,|\sigma^k|\}$ such that $S^{\ell+m}(x)$ and $S^{\ell+n}(x)$ belong to $U_k$ and $c^{(\ell)}(S^m(x)) = c^{(\ell)}(S^n(x))$.
Observe that, as $S^{\ell+m}(x)$ and $S^{\ell+n}(x)$ belong to $ U_k$, we have $c^{(m+\ell)}(x)=c^{(n+\ell)}(x)=0$.
We thus have
\begin{eqnarray*}
	|d(S^m(x))-d(S^n(x))| 
	&=& |c^{(m)}(x)-c^{(n)}(x)| 		\\
	&=& |c^{(m)}(x)+c^{(\ell)}(S^m(x))-c^{(n)}(x)-c^{(\ell)}(S^n(x))| \\
	&\leq & |c^{(m+\ell)}(x)| + |c^{(n+\ell)}(x)| = 0,	
\end{eqnarray*}
which shows that $d$ is uniformly continuous on $\mathcal{O}_+(x)$.

\medskip

Let us end showing that the radius of $d$ is at most $|\sigma^k|+\max\{L_{\sigma^k},r\}$. 
We recall (see Corollary~\ref{lem:rec}) that ${\mathcal{P}}=\{ S^m \sigma^k ([a] ) \mid a\in A,  0\leq m < |\sigma^k (a) | \}$ is a clopen partition of $X_{\sigma^k} = X_\sigma$.

Let $n \in \mathbb{N}$ with $n\geq \max_{a \in A} |\sigma^k (a) |$.
We have $S^n(x) \in S^{m(S^n(x))} \sigma^k ([a] )$ for some $a\in A$ and some $m(S^n(x)) \in \{0,\dots, |\sigma^k (a)|-1\}$. In particular, we have $n-m(S^n(x)) \geq 0$ and, since $S^{n-m(S^n(x))}(x) \in U_k$, $c^{(n-m(S^n(x)))}(x) = 0$.
We get
\[
	d(S^n(x)) = c^{(n)}(x) = c^{(m(S^n(x)))}(S^{n-m(S^n(x))}(x)).
\]
By continuity of $d$, for all $y \in X$, we have 
\begin{equation}
\label{eq:d(y)}
	d(y) = c^{(m(y))}(S^{-m(y)}y), 
\end{equation}
where $0 \leq m(y) < |\sigma^k(a)|$ for some $a \in A$ such that $y \in S^{m(y)} \sigma^k([a])$.
By Corollary~\ref{lem:rec}, the function $m: X \to \{0,\dots,|\sigma^k|\}$ is continuous and has radius at most $L_{\sigma^k}+|\sigma^k|$, hence $d$ has radius at most $|\sigma^k|+\max\{L_{\sigma^k},r\}$.
The inequality
\[
\max_{x \in X_\sigma}|d(x)| \leq |\sigma^{k}| \max_{x \in X_\sigma} |c(x)|
\]
directly follows from~\eqref{eq:d(y)}.
\end{proof}

Below we prove a more detailed version of Theorem~\ref{theo:main} when the subshift are generated by substitutions and where $\sigma$ is a proper substitution.

\begin{theo}
\label{theo:mainpropersub}
Let $\sigma : A^* \to A^*$ and $\tau : B^* \to B^*$ be two aperiodic and positive substitutions, $\sigma$ being proper.
Suppose $\sigma$ and $\tau$ have the same dominant eigenvalue $\lambda$.
The following are equivalent.
\begin{enumerate}
\item
\label{item:factormap}
There exists a factor map $f : (X_\sigma , S)\to (X_\tau , S)$.
\item
\label{item:dillmap}
There exists a dill map $F: (X_\sigma , S)\to (X_\tau , S)$ such that
\begin{itemize}
\item
its radius $(s_1, s_2)$ satisfies
$$
	s_1 
	\leq 
	\lceil K_\tau \lambda (\lfloor (K_\tau+K_\sigma) \lambda \rfloor ) + L_\tau +2 \rceil 
	\quad \text{ and } \quad
	s_2 
	\leq 
	\lfloor (K_\tau+K_\sigma) \lambda \rfloor.
$$
\item
for all $x\in X_\tau$, $F\circ S (x) = S^{c(x)} \circ F (x)$ where $c:X\to \mathbb{Z}$ is a continuous map satisfying that $(\sum_{k=0}^{n-1} (c \circ S^k -1))_n$ is bounded.
\end{itemize}
\item
\label{item:smallradius}
There exists a factor map $g : (X_\sigma , S)\to (X_\tau , S)$ having a radius bounded by, with $R = \lceil K_\tau \lambda (\lfloor (K_\tau+K_\sigma) \lambda \rfloor ) + L_\tau +2 \rceil$, 
\[
    \lfloor (K_\tau+K_\sigma) \lambda \rfloor
	(K_\sigma+1) 
	(2 R+1)
	|\sigma^{p_{\sigma }(2R+1)}|.
\]
\end{enumerate}
Moreover, once given $f$ one can find $g$ (as above) such that $f= S^n \circ g$ for some $n$.
\end{theo}

\begin{proof}
Of course \eqref{item:smallradius} implies \eqref{item:factormap}.
Let us first show that \eqref{item:factormap} implies \eqref{item:dillmap}.
Recall that for all $n$, $F_n: X_\sigma \to X_\tau$ is the dill map defined by
\[
	\tau^n \circ F_n = S^{r_n \circ f \circ \sigma^n} \circ f \circ \sigma^n.
\]
Let $r$ be a radius of $f$ and $L_\tau$ the constant of recognizability of $\tau$.
From Proposition~\ref{lemme:boundradius}, we have
$$
\#\{ F_n \mid r/|\tau^n| \leq 1 \}  
\leq 
\left(
    K_\tau \lambda^2
    (K_\sigma+K_\tau)
    \right)^{2 K_\sigma K_\tau \lambda^2 (K_\tau+K_\sigma + L_\tau +3 )}.
$$
Thus there exist $n$ and $m$, $n \geq m$, verifying $|\tau^{n-m}| \geq r$ and such that $F_n = F_m$.
From Lemma~\ref{lem:equiv} we obtain that $f=F_0 \equiv F_{n-m}$.

We fix $N = n-m$, $F = F_N$ and we let $(s_1,s_2)$ be the radius pair of $F$.
From Proposition~\ref{lemme:boundradius} one has 
$$
	s_1 
	\leq 
	\lceil K_\tau \lambda (\lfloor (K_\tau+K_\sigma) \lambda \rfloor ) + L_\tau +2 \rceil 
	\quad \text{ and } \quad
	s_2 
	\leq 
	\lfloor (K_\tau+K_\sigma) \lambda \rfloor.
$$
Let $c,d : X_\sigma \to \mathbb{Z}$ be the unique continuous maps verifying for all $x\in X_\sigma$, 
\[
	F (Sx) = S^{c (x) } F (x)
	\quad \text{and} \quad
	F (Sx) = S^{d (x) } f (x).
\]
One has $|c(x)| \leq s_2$ for all $x\in X_\sigma$. 
Furthermore, we have, for all $x$, 
\[
	F (S x) = S^{d (x) +1} f (S^{-1} x) = S^{d (x) +1-d (S^{-1} (x))} F ( x),
\]
hence $c(x) = d (x)-d (S^{-1} (x)) +1$.
This proves~\eqref{item:dillmap}.

\medskip

Let us show that~\eqref{item:dillmap} implies~\eqref{item:smallradius}.
Observe that the radius $R$ of $c$ is at most $\lceil K_\tau \lambda (\lfloor (K_\tau+K_\sigma) \lambda \rfloor ) + L_\tau +2 \rceil$.
From Lemma~\ref{lem:cobound}, there exists a continuous map $D: X \to \mathbb{Z}$ with radius 
\[
	H \leq |\sigma^{p_{\sigma }(2R+1)}|+\max\{L_{\sigma^{p_{\sigma }(2R+1)}},R\}
\]
such that $c -1 = D\circ S - D$ and for all $x \in X_\sigma$,
\[
	|D(x)| \leq \lfloor (K_\tau+K_\sigma) \lambda \rfloor |\sigma^{p_{\sigma }(2R+1)}|.
\]
 
Let $G: X\to Y$ be the continuous map defined by $G (x) = S^{-D(x)} \circ F(x)$.
We have:
\begin{align*}
G\circ S (x) & =  S^{-D(S (x))} \circ F(S(x)) = S^{-D(S (x)) +c(x)-1+1} \circ F(x) \\
& = S^{-D(x)+1} \circ F(x) = S\circ G(x) .
\end{align*}
Thus $G$ is a factor map from $(X_\sigma , S)$ to $(X_\tau , S)$.

Let us give a bound on the radius of $G$.
Let us set $M = \lfloor (K_\tau+K_\sigma) \lambda \rfloor |\sigma^{p_{\sigma }(2R+1)}|$.
The radius of $G$ is at most equal to $N$, where $N$ is such that for all $x \in X_\sigma$, $x_{[-N,N]}$ uniquely determines $F(x)_{[-M,M]}$, i.e.,
for all $x,y \in X_\sigma$, 
\[
x_{[-N,N]} = y_{[-N,N]} \Longrightarrow F(x)_{[-M,M]} = F(y)_{[-M,M]}.
\]
Let $\phi$ be the implementation of $F$ given by Lemma~\ref{lemme:implementationdill}.
For all $n$, we have
\[
\begin{array}{lcll}
F(x)_{[0,c^{(n)} (x) -1]} 
&=& \phi (x) \phi (Sx) \cdots \phi (S^{n-1} x), & \quad  \text{if} \; n> 0; \\
F(x)_{[c^{(n)} (x) , -1]} 
&=& \phi (S^{n} x) \phi (S^{n+1} x) \cdots \phi (S^{-1} x), & \quad  \text{if} \; n< 0.
\end{array}
\]
Furthermore, there is a local map $\hat{\phi}: A^{2s_1+1} \to B^*$ defining $\phi$ so that 
\[
\begin{array}{lcll}
F(x)_{[0,c^{(n)} (x) -1]} 
&=& \prod_{i=0}^{n-1} \hat{\phi} (x_{[-s_1+i,s_1+i]}), & \quad  \text{if }  n> 0; \\
F(x)_{[c^{(n)} (x) , -1]} 
&=& \prod_{i=0}^{n-1} \hat{\phi} (x_{[-s_1-n+i,s_1-n+i]}), & \quad  \text{if }  n< 0.
\end{array}
\]
As $F$ is a dill map, there is a word $u \in \mathcal{L}(\sigma)$ of length $2s_1+1$ such that $\hat{\phi}(u) \neq \epsilon$.
By Theorem~\ref{theo:encad}, $u$ occurs in any word $v \in \mathcal{L}(\sigma)$ of length $(K_\sigma+1) (2s_1+1)$.
Therefore, $u$ occurs at least $M$ times both in 
\[
	x_{[-s_1,s_1+M(K_\sigma+1) (2s_1+1)]}
	\quad \text{ and in } \quad
	x_{[-s_1-1-M(K_\sigma+1) (2s_1+1),s_1-1]}.
\]
The radius of $G$ is thus at most equal to
\begin{eqnarray*}
	N 
	&\leq & 
	M(K_\sigma+1) (2s_1+1) \\
	&\leq & 
	\lfloor (K_\tau+K_\sigma) \lambda \rfloor
	(K_\sigma+1) 
	(2 \lceil K_\tau \lambda (\lfloor (K_\tau+K_\sigma) \lambda \rfloor ) + L_\tau +2 \rceil+1)
	|\sigma^{p_{\sigma }(2R+1)}|,
\end{eqnarray*}
which ends the proof.
\end{proof}

As a corollary, and using Theorem \ref{theo:durand13bbis}, we obtain Theorem~\ref{theo:main} when $(Y,S)$ is aperiodic.

\section{Decidability of the factorization in the case of periodic subshifts}\label{section:periodic}

In this section we consider Theorem~\ref{theo:cordecidfactor}, Corollary~\ref{cor:cordecidfactor}, Theorem~\ref{theo:main} and Theorem~\ref{theo:main2} under the assumptions that $(X,S)$ and, or, $(Y,S)$ are periodic. 
Recall that a minimal subshift $(X,S)$ is periodic if there is some $n>0$ such that $S^n = \id$. 
Such an $n$ is called a {\em period} of $(X,S)$ and for any $x \in X$, we have $X = \{x,Sx,S^2x,\dots S^{n-1}x\}$. 

As we are dealing with decision problems, it is important to recall that it is decidable to know whether a minimal morphic subshift is periodic and that the minimal period can be computed.

\begin{theo}[\cite{Durand:2013}]
\label{theo:decidperiod}
Let $x$ be a uniformly recurrent morphic sequence. 
It is decidable whether $x$ is periodic.
Moreover some $u$ such that $x=uuu\cdots $ with $|u|$ minimal can be computed. \end{theo}

We also recall Fine and Wilf theorem.
\begin{theo}[\cite{Fine&Wilf:1965}]
Let $x,y$ be two sequences with respective periods $p$ and $q$. 
If $x$ and $y$ coincide on the indices $\{0,1,\dots,p+q-\gcd(p,q)-1\}$, then they are equal and have period $\gcd(p,q)$. 
\end{theo}
In particular, if $p$ and $q$ are two periods of a minimal periodic subshift $(X,S)$, then every sequence in $X$ has periods $p$ and $q$, so also has period $\gcd(p,q)$. 
Thus $\gcd(p,q)$ is also a period of $(X,S)$.

Let us give the following notation that will facilitate the description of the different situations. 
Let $(X,T)$ be a minimal topological dynamical system and consider the set 
\[
	P (X,T) = \{p \in \N \mid (X,T) \text{ admits a factor of cardinality } p\}.
\]

For the remaining of this section, we assume that $(Y,S)$ is periodic with minimal period $q$ and we write $Y = \{y,Sy,\dots S^{q-1}y\}$ for some fixed $y \in Y$.
We also assume without loss of generality that $X = X_\sigma$ for some primitive and proper substitution $\sigma:A^* \to A^*$.
We can trivially make this assumption if $(X,S)$ is periodic. 
Otherwise, we use Theorem~\ref{theo:durand13bbis}.

\subsubsection*{Proof of Theorem~\ref{theo:cordecidfactor} and Corollary~\ref{cor:cordecidfactor} for the periodic case}\label{subsubsec:proofperiodic}

If $(X,S)$ is periodic with minimal period $p$ then it suffices to test whether $q$ divides $p$ to decide whether $(Y,S)$ is  a factor of $(X,S)$.
Indeed, if $(Y,S)$ is  a factor of $(X,S)$, then both $p$ and $q$ are periods of $(Y,S)$ and thus so is $\gcd(p,q)$. 
As $q$ is the minimal period of $(Y,S)$, we get $q = \gcd(p,q)$, hence $p$ is a multiple of $q$.
Respectively, if $q$ divides $p$, then we may write $X = \{x,Sx,\dots S^{nq-1}x\}$ for some $x \in X$.
Then the map $f: X \to Y$ defined by $f(S^{i}x) = S^{i \bmod q}y$ is a factor map from $(X,S)$ to $(Y,S)$.

Now assume that $(X,S)$ is not periodic. 
Then the following result allows to conclude for the proofs of~\ref{theo:cordecidfactor} and Corollary~\ref{cor:cordecidfactor} in this periodic context.

\begin{theo}[\cite{Durand&Goyheneche:2019}]
\label{theo:valerie}
Let $(X,S)$ be an aperiodic uniformly recurrent morphic subshift and $p$ be a positive integer.
It is decidable to check whether $p$ belongs to $P (X,S)$ or not. 
\end{theo}

\subsubsection*{Proof of Theorem~\ref{theo:main} for the periodic case}

Like in the previous case, we split the proof depending on whether $(X,S)$ is periodic or not.
If $(X,S)$ is periodic with minimal period $p$ then, as explained above, $q$ should divide $p$ and, writing $X = \{x,Sx,\dots S^{nq-1}x\}$, the map $f: X \to Y$ defined by $f(S^{i}x) = S^{i \bmod q}y$ is a factor map from $(X,S)$ to $(Y,S)$.
Let us show that $p-1$ is a radius of $f$.
We claim that all words $x_{(n-p,n+p)}$, $n \in \{0,1,\dots,p-1\}$ are pairwise distinct.
Indeed, if $x_{(n-p,n+p)} = x_{(m-p,m+p)}$ for some $0 \leq m < n <p$, then for $0 \leq k < p$, we obtain
\[
    x_{k+n-m} = x_{n+(k-m)} = x_{m+(k-m)} = x_k.
\]
By $p$-periodicity, we deduce that $n-m$ is a period of $(X,S)$, which contradicts the minimality if $p$.
It is then clear that the sliding block code $\hat{f}$ defines $f$ where $\hat{f}$ is defined by
\[
    \hat{f}(x_{[n-p,n+p]}) = y_{n \bmod q}, \quad 0 \leq n < p.
\]
Now if $g:(X,S) \to (Y,S)$ is another factor map, there exists $i \in \{0,\dots,q-1\}$ such that $g(x)=S^iy = S^i(f(x))$.
Since both $f$ and $g$ are factor maps, this implies that $g = S^i \circ g$.

Now assume that $(X,S)$ is not periodic.
Thus the primitive and proper substitution $\sigma$ is aperiodic.  
Let us recall the following lemma (Proposition 3 and Lemma 10 in \cite{Durand&Goyheneche:2019} or Lemma 28 in \cite{Durand:2000}).

\begin{lem}\label{vpmatrice}
Let $\sigma : A^* \to A^*$ be an aperiodic and left proper substitution with incidence matrix $M$ and let $p$ be an integer. The following properties are equivalent:
\begin{enumerate}
\item
\label{vpmatrice:1}
$p$ belongs to $P(X_\sigma , S)$ ;
\item
\label{vpmatrice:2}
there exists $m\in\N$ such that $p$ divides $|\sigma^m(a)|$ for all $a\in A$ ;
\end{enumerate}
Moreover, when $p$ is a prime number, this is equivalent to:
\begin{enumerate}
\setcounter{enumi}{2}
\item
\label{vpmatrice:3}
${\bf 1} M^d = (|\sigma^d (a)|)_{a\in A}$ belongs to $p\Z^d$, where $d = |A|$,
\end{enumerate}
where ${\bf 1}$ is the row vector consisting of $1$'s.
\end{lem}

Let $q=q_1^{r_1} \cdots q_k^{r_k}$ where the $q_i$'s are distinct prime numbers.
For every $q_i$, we have $q_i \in P(Y,S) \subset (P(X,S))$. Thus, by~Lemma \ref{vpmatrice}, the vector $(|\sigma^{d} (a)|)_{a\in A}$ belongs to $q_i\mathbb{Z}^d$.
Considering the value $r= \max\{r_1,\dots,r_k\}$, we then have $(|\sigma^{rd} (a)|)_{a\in A} \in q_i^r\mathbb{Z}^d$.
Having $\gcd(q_i^r,q_j^r)=1$ for every distinct $i,j$, we deduce that 
\[  
    (|\sigma^{rd} (a)|)_{a\in A} 
    \in \left(\prod_{i=1}^k q_i^r\right) \mathbb{Z}^d 
    \subset q\mathbb{Z}^d
\]
By Corollary \ref{lem:rec}, the map $f:X \to Y$ defined by
\[
    f(x) = S^i y 
\]
whenever $x \in S^i \sigma^{rd}([a])$ for some $a \in A$ and some $0 \leq i < |\sigma^{rd}(a)|$ is well defined and is a factor map with radius $L_{\sigma^{rd}}+|\sigma^{rd}|$, where $L_{\sigma^{rd}}$ is the recognizability constant of $\sigma^{rd}$ (see Theorem \ref{theo:recognizability constant}).
Let finally $g:(X,S) \to (Y,S)$ be another factor map and let $x \in X$ be the unique fixed point of $\sigma$.
There exist unique $0\leq i,j < q$ such that $f(x)=S^i y$ and $g(x) = S^j y = S^{j-i} f(x)$.
We deduce that $g = S^{j-i} \circ f$ from the minimality of $(X,S)$.

\subsubsection*{Proof of Theorem~\ref{theo:main2} for the periodic case}

Let $\phi$ be a sliding block code.
Taking some block representation of $(X_\sigma,S)$, one can suppose $\phi$ is a coding, see Section~\ref{subsec:sublengthn}.
This can be done algorithmically.  

If $x$ is a fixed point of $\sigma$, then $\phi$ defines a factor map from $(X,S)$ to $(Y,S)$ if and only if $\phi(x)$ and $y$ have the same language. 
This problem is decidable using Theorem~\ref{theo:decidperiod}.

\section{Decidability of the factorization}\label{section:decidability}

In this section we prove that the factorization and isomorphism problems between minimal morphic subshifts are decidable, that is Theorem~\ref{theo:cordecidfactor} and Corollary~\ref{cor:cordecidfactor}.
Let us sketch briefly how we proceed. 

From Theorem~\ref{theo:durand13bbis}, and in particular Remark~\ref{rem:factor of radius K+r}, it is enough to consider primitive substitution subshifts.
Let $\sigma $ and  $\tau $ be two primitive substitutions.

In Section~\ref{section:bounds}, we have shown there is a computable constant $R$ such that if there is a factor map $f:(X_\sigma,S) \to (X_\tau,S)$, then there is another factor map $F :(X_\sigma,S) \to (X_\tau,S)$ with radius at most $R$.
In this section, we show that there is an algorithm that answers whether or not a given sliding block code $f: A^{2R+1} \to B$ defines a factor map from $(X_\sigma,S)$ to $(X_\tau,S)$.
Using Lemma~\ref{lemma:block representation isomorphic} and Proposition~\ref{prop:conj-sublengthn}, we can assume that $f$ has radius 0, i.e., $f:A \to B$ defines a coding. 
As a subshift is uniquely determined by its language, we want to check whether or not $f(\mathcal{L}(\sigma)) = \mathcal{L}(\tau)$.
If $x$ is an admissible fixed point of $\sigma$, as well as $y$ for $\tau$, all we have to check is that $\mathcal{L}(f(x)) = \mathcal{L}(y)$.

To this end, we will deeply use return words.
Roughly speaking, we will show that we have $\mathcal{L}(f(x)) = \mathcal{L}(y)$ if, and only if, there exists a prefix $u$ of $x$, whose length is bounded by a computable constant, such that the return words to $f(u)$ in $y$ and the return words to $f^{-1}(\{f(u)\})$ in $x$ (as defined in Section~\ref{subsection:return_word_set}) are identically concatenated in $y$ and in $x$ respectively (as precisely stated by Lemma~\ref{lemma:def of varphi_u} and Theorem~\ref{thm:sufficient_conditions}).
A key argument is to find some ``canonical morphisms'', i.e., morphisms that only depend on the combinatorics of $x$ and $y$ and not on the substitutions $\sigma$ and $\tau$.
One of the main steps for this is to extend the results of Section~\ref{subsection:return words} to return words to sets of words instead of words.

\subsection{Return words to a set of words}
\label{subsection:return_word_set}

Let $x \in A^\N$ be a sequence and let $U \subset \mathcal{L} (x)$ be a set of non-empty words of the same length.
We call {\em return word to $U$} any non-empty word $w_{[i,j-1]} = w_i w_{i+1} \cdots w_{j-1}$ for which there exist $u,v \in U$ such that $wv$ is a word in $\cL(x)$ that admits $u$ as a prefix and that contains no other occurrences of a word in $U$.
The set of return words to $U$ in $x$ is denoted by ${\mathcal{R}}_{x,U}$.


With such a definition, we can define a bijection $\Theta_{x,U}$ as in the case of return words to a single word.
We consider the set $R_{x,U} = \{1,2,\dots\}$ with $\#R_{x,U} = \#\mathcal{R}_{x,U}$ and we define 
\[
	\Theta_{x,U}: R_{x,U}^* \to \mathcal{R}_{x,U}^*
\]
as the unique morphism that maps $R_{x,U}$ bijectively onto $\mathcal{R}_{x,U}$ and such that for all $i \in R_{x,U}$, $\Theta_{x,U}(i)$ is the $i$th word of $\mathcal{R}_{x,U}$ occurring in $x$.


\begin{exem}
\label{ex:return set}
Considering the Thue-Morse sequence $y$ and the set of words $U = \{aa,bb\}$, we have $\cR_{y,U} = \{aa,bb,aaba,bbab\}$ and 
\[
	\Theta_{y,U}:
		\begin{cases}
			1 \mapsto bbab	\\
			2 \mapsto aa	\\
			3 \mapsto bb	\\
			4 \mapsto aaba
		\end{cases}.
\]
Similarly, for the Fibonacci sequence $x$ and the set $U = \{aa,ab\}$, we have $\cR_{x,U} = \{a,ab\}$ and 
\[
	\Theta_{x,U}:
		\begin{cases}
			1 \mapsto ab	\\
			2 \mapsto a		
		\end{cases}.
\]
\end{exem}

Our aim is to prove a result similar to Proposition~\ref{prop:derived sequence fixed point}, but for return words to sets of words.
In the case where $U$ is a singleton $\{u\}$, a key argument in the proof of Proposition~\ref{prop:derived sequence fixed point} is that $\mathcal{R}_{x,u}$ is a code (see Proposition~\ref{prop: return words code}).
This is in particular needed for the morphism $\lambda_{x,u,u'}$ of that proposition  to be properly defined.
Unfortunately, this property is not true anymore when $U$ is not a singleton.
Indeed, in~\cite{Leroy&Richomme:2013}, the authors exhibit an admissible fixed point $x$ of a primitive substitution for which, if we take $U_n$ to be the set of left special words of length $n$, we have $\{a,b,ab\} \subset \mathcal{R}_{x,U_n}$ for infinitely many values of $n$.
We thus introduce the following notion (see also~\cite{Durand:2013}).

\begin{defi}
Let $x \in A^\N$ be a sequence and let $U \subset \cL(x)$ be a set of non-empty words with the same length.
A {\em return pair} to $U$ is a pair $(w,u) \in \mathcal{R}_{x,U} \times U$ such that $wu$ belongs to $\mathcal{L} (x)$ and contains only two occurrences of words in $U$ ($u$ and the prefix of length $|u|$).
We let $\tilde{\mathcal{R}}_{x,U}$ denote the set of return pairs to $U$ in $x$.
\end{defi}

Using these return pairs, we again enumerate the elements $(w,u)$ in $\tilde{\mathcal{R}}_{x,U}$ in the order of the first occurrence of $wu$ in $x$.
Formally, we consider the set $\tilde{R}_{x,U} = \{ 1, 2, \dots \}$ with $\# \tilde{R}_{x,U} = \#\tilde{\mathcal{R}}_{x,U}$ and we define 
\[
	\tilde{\Theta}_{x,U} : \tilde{R}_{x,U}^*  \to \tilde{\mathcal{R}}_{x,U}^*
\]
as the unique morphism that maps bijectively $\tilde{R}_{x,U}$ onto $\tilde{\mathcal{R}}_{x,U}$ and for which for all $i \in \tilde{R}_{x,U}$, the element $\tilde{\Theta}_{x,U} (i) =(w,u)$ is such that $wu$ is the $i$th such word occurring in $x$.
Observe that we should only consider words 
$(w_1,u_1) \cdots (w_n,u_n) \in \tilde{\mathcal{R}}_{x,U}^*$ that correspond to real concatenations of return words.
More precisely, we say that $(w_1,u_1) \cdots (w_n,u_n) \in \tilde{\mathcal{R}}_{x,U}^*$ is {\em admissible} if 
\begin{enumerate}
\item
$w_1w_2 \cdots w_n u_n$ is a factor of $x$;
\item
for all $i \in \{1,2,\dots,n-1\}$, $u_i$ is a prefix of $w_{i+1} w_{i+2} \cdots w_n u_n$.
\end{enumerate}
By extension, a sequence $(w_1,u_1)(w_2,u_2) (w_3,u_3) \cdots \in \tilde{\mathcal{R}}_{x,U}^\N$ is admissible if so are its prefixes. 
By symmetry, we also say that a finite or infinite word $w$ on $\tilde{R}_{x,U}$ is admissible if so is $\tilde{\Theta}_{x,U}(w)$.

\begin{exem}
We continue our running examples with the Thue-Morse sequence and the Fibonacci sequence.
For the Thue-Morse sequence, we have 
$\tilde{\cR}_{y,U} = \{(aa, bb),(bb, aa),(aaba, bb), (bbab, aa)\}$ and
\[
	\tilde{\Theta}_{y,U}:
		\begin{cases}
			1 \mapsto (bbab, aa)	\\
			2 \mapsto (aa, bb)		\\
			3 \mapsto (bb, aa)		\\
			4 \mapsto (aaba, bb)
		\end{cases}.
\]
Observe that the word $(aa, bb)(bb, aa)(aaba, bb)$ is admissible (because the word $aabbaababb$ occurs in $y$), but $(aa, bb)(bb, aa)(aa, bb)$ is not.

For the Fibonacci sequence, we have 
$\tilde{\cR}_{x,U} = \{(ab, aa),(a, ab),(ab, ab)\}$ and
\[
	\tilde{\Theta}_{x,U}:
		\begin{cases}
			1 \mapsto (ab, aa)	\\
			2 \mapsto (a, ab)	\\
			3 \mapsto (ab, ab)
		\end{cases}.
\]
\end{exem}

This notion of return pairs satisfies a generalization of Proposition~\ref{prop: return words code}.
 
\begin{prop}
\label{prop:return pair code}
Let $x$ be a sequence in $A^{\mathbb{N}}$ and let $U \subset \cL(x)$ be a set of non-empty words with the same length.
\begin{enumerate}
\item \label{item return pair code 1}
If $(w_1,u_1) \cdots (w_k,u_k) \in \tilde{R}_{x,U}^*$ is admissible, and, if $(r,v) \in \tilde{R}_{x,U}$ is such that $rv$ occurs in the word $w_1w_2 \cdots w_ku_k$, then $(r,v) = (w_i,u_i)$ for some $i \in \{1,\dots,k\}$.
\item \label{item return pair code 2}
If two admissible words 
$(w_1,u_1) \cdots (w_k,u_k)$ and $(r_1,v_1)\cdots (r_\ell,v_\ell)$ belonging to $\tilde{R}_{x,U}^*$ 
are such that 
$w_1w_2 \cdots w_ku_k = r_1 r_2 \cdots r_\ell v_\ell$, 
then $k=\ell$ and for all $i \in \{1,\dots,k\}$, $(w_i,u_i)=(r_i,v_i)$.
\end{enumerate}

\end{prop}
\begin{proof}
Let us prove~\eqref{item return pair code 1}.
Let $i \leq k$ be the largest integer such that $rv$ occurs in $w_i \cdots w_k u_k$.
If $(r,v) \neq (w_i,u_i)$, there exists a non-empty word $p$ such that $prv$ is a prefix of $w_i \cdots w_k u_k$.
Furthermore, by choice of $i$, we have $|p|<|w_i|$.
Let $v'$ denote the prefix of $rv$ which is in $U$.
If $|pr| \leq |w_i|$, then $w_iu_i$ contains at least three occurrences of words of $U$: $u_{i-1}$, $v'$ and $u_i$, where $u_0$ is the prefix of $w_1u_1$ which is in $U$. 
This contradicts the fact that $(w_i,u_i)$ is a return pair.
If $|pr| > |w_i|$, then we get an analogous contradiction with $rv$ containing occurrences of $v'$, $u_i$ and $v$.
In all cases, we get a contradiction, so we have $(r,v) = (w_i,u_i)$. 

Let us now prove~\eqref{item return pair code 2}.
The proof goes by induction on $k$ and uses the same reasoning as for~\eqref{item return pair code 1}.
Assume that $k = 1$ and let $u'$ be the prefix of $w_1 u_1$ which is in $U$. 
Since $w_1 u_1 = r_1 r_2 \cdots r_\ell v_\ell$, all words $u,r_1, \dots, r_\ell$ occur in $w_1 u_1$.
By definition of a return pair, we must have $\ell = 1$.
As $|u_1| = |v_1|$, this shows that $(w_1,u_1) = (r_1,v_1)$.

Assume that the result is true for $k \geq 1$ and let us prove it for $k+1$.
With the same reasoning as for the case $k=1$, we must have $\ell > 1$.
Since $|u_{k+1}| = |v_\ell|$, we again have $u_{k+1} = v_\ell$.
If $w_{k+1} \neq r_\ell$, then since $w_1 \cdots w_{k+1} = r_1 \cdots r_\ell$, we have $|w_{k+1}| \neq |r_\ell |$.
Let us consider the case $|w_{k+1}| < |r_\ell|$, the other one being symmetric.
Then $r_\ell v_\ell$ contains at least three occurrences of words of $U$: $v_{\ell-1}$ has a prefix, $v_\ell$ as a suffix and $u_k$.
This contradicts the definition of a return pair, so we have $(w_{k+1},u_{k+1})=(r_\ell,v_\ell)$, hence $w_1 \cdots w_k u_k = r_1 \cdots r_{\ell-1} v_{\ell-1}$, which ends the proof.
\end{proof}

Let $\delta_{x,U,1} : \tilde{\mathcal{R}}_{x,U}^* \to \mathcal{R}_{x,U}^*$ and $\delta_{x,U,2} : \tilde{\mathcal{R}}_{x,U}^* \to U^*$ be respectively defined
\[
\delta_{x,U,1} (w,u) = w \qquad \text{and} \qquad \delta_{x,U,2} (w,u) = u.
\]
Let also $d_{x,U,1}:\tilde{R}_{x,U}^* \to \mathcal{R}_{x,U}^*$ and $d_{x,U,2}:\tilde{R}_{x,U}^* \to U^*$ be respectively defined by
\[
d_{x,U,1}(i) = \delta_{x,U,1}(\tilde{\Theta}_{x,U}(i)) 
\qquad \text{and} \qquad
d_{x,U,2}(i) = \delta_{x,U,2}(\tilde{\Theta}_{x,U}(i)).
\]
We get the following generalization of Proposition~\ref{prop:def suite derivee singleton}.

\begin{prop}
\label{prop:def-DU}
Let $x$ be a sequence in $A^{\mathbb{N}}$ and let $U \subset \cL(x)$ be a set of non-empty words with the same length.
Suppose that $x$ has infinitely many occurrences of elements of $U$ and let $i \in \N$ be the first occurrence of a word of $U$ in $x$.
Then, 
\begin{enumerate}
\item
There exists a unique admissible sequence $(w_n,u_n)_n $ belonging to $ \tilde{\mathcal{R}}_{x,U}^\mathbb{N}$ such that for all $n$, the word $w_0 w_1\cdots w_n u_n$ is a prefix of $S^i(x)$;
\item
There exists a unique admissible sequence $z \in \tilde{R}_{x,U}^\mathbb{N}$ satisfying  $d_{x,U,1} (z) = S^i(x)$.
\end{enumerate}
\end{prop}
\begin{proof}
The existence of an admissible sequence $(w_n,u_n)_n \in \tilde{\mathcal{R}}_{x,U}^\mathbb{N}$ satisfying item 1 directly follows from the fact that $x$ has infinitely many occurrences of elements of $U$.
Let us show that it is unique.
Let $(w_n',u_n')_n \in \tilde{\mathcal{R}}_{x,U}^\mathbb{N}$ be another such admissible sequence and assume that $n_0$ is the smallest integer for which $(w_{n_0},u_{n_0}) \neq (w_{n_0}',u_{n_0}')$.
Since both $w_0 w_1\cdots w_{n_0} u_{n_0}$ and $w_0' w_1' \cdots w_{n_0}' u_{n_0}'$ are prefixes of $S^i(x)$, with $|u_{n_0}|=|u_{n_0}'|$, we must have $|w_{n_0}| \neq |w_{n_0}'|$.
Assume without loss of generality that $|w_{n_0}| < |w_{n_0}'|$.
Thus $w_{n_0}u_{n_0}$ is a proper prefix of $w_{n_0}'u_{n_0}'$, which implies that $w_{n_0}'u_{n_0}'$ contains at least three occurrences of words in $U$.
This contradicts the fact that $(w_{n_0}',u_{n_0}')$ is a return pair to $U$.

Item 2 directly follows from Item 1.
\end{proof}

\begin{defi}
The sequence $z \in \tilde{R}_{x,U}^\mathbb{N}$ of the previous proposition is called a {\em derived sequence} of $x$ (w.r.t. $U$).
We denote it by $\mathcal{D}_U(x)$. 
\end{defi}

\begin{Rem}
Observe that when $U$ is a singleton $\{u\}$, we have $\tilde{\mathcal{R}}_{x,\{u\}} = \mathcal{R}_{x,u} \times \{u\}$ and $\cD_{\{u\}}(x) = \cD_u(x)$. 
\end{Rem}

The next result directly follows from Proposition~\ref{prop:return pair code}.

\begin{cor}
\label{cor:derivee ensemble de mot reste UR}
Let $x \in A^\N$ be uniformly recurrent and let $U \subset \cL(x)$ be a set of non-empty words with the same length.
Then the derived sequence $\mathcal{D}_U(x)$ is uniformly recurrent.
\end{cor}

We can now state a result analogous to Proposition~\ref{prop:derived sequence fixed point}.
We first independently define the morphism $\lambda_{x,U,V}$ in a more general setting.

\begin{lem}
\label{lemme:def de lambda}
Let $x$ be a uniformly recurrent sequence in $A^\mathbb{N}$.
Let $U$ and $V$ be subsets of $\mathcal{L}(x)$ such that all words of $U$ have the same length $\ell_U \geq 1$ and all words of $V$ have the same length $\ell_V > \ell_U$.
Assume that any word of $V$ has a prefix in $U$.
There exists a unique morphism $\lambda_{x,U,V}: \tilde{R}_{x,V}^* \to \tilde{R}_{x,U}^*$ such that
\begin{equation}
\label{eq:lambda_x,U,V}
	\begin{cases}
		d_{x,U,1} \circ \lambda_{x,U,V} = d_{x,V,1}; 	\\
		d_{x,U,2} (\lambda_{x,U,V}(i)_{|\lambda_{x,U,V}(i)|-1}) 
		= (d_{x,V,2}(i))_{[0,\ell_U -1]}, & \text{for all } i \in \tilde{R}_{x,U}.	
	\end{cases}
\end{equation}
In particular, we have $\lambda_{x,U,V}(\mathcal{D}_V(x)) = S^k \mathcal{D}_U(x)$, where $k \geq 0$ is the smallest integer such that the sequence $d_{x,U,1}(S^k \mathcal{D}_U(x))$ has a prefix in $V$.
\end{lem}

\begin{proof}
Let us first show that the morphism $\lambda_{x,U,V}$ is well defined.
Let $i \in \tilde{R}_{x,V}$ and let $w \in \mathcal{R}_{x,V}$ and $v',v \in V$ be the words such that $\tilde{\Theta}_{x,V}(i) = (w,v)$ and $v'$ is a prefix of $wv$.
As any word in $V$ has a prefix in $U$, there exists an admissible word  $(w_1,u_1) \cdots (w_k,u_k) \in \tilde{\mathcal{R}}_{x,U}^*$ such that $wv_{[0,\ell_U -1]} = w_1 w_2 \cdots w_k u_k$.
By Proposition~\ref{prop:return pair code}, this admissible word is unique.
We thus define $\lambda_{x,U,V}(i)$ as the unique word $i_1 \cdots i_k \in \tilde{R}_{x,U}^*$ such that $\tilde{\Theta}_{x,U}(i_j) = (w_j,u_j)$ for all $j$ such that $1 \leq j \leq k$.
Observe that this definition is equivalent to~\eqref{eq:lambda_x,U,V}.

Now let us prove that $\lambda_{x,U,V}(\mathcal{D}_V(x)) = S^k \mathcal{D}_U(x)$.
By Proposition~\ref{prop:return pair code}, $\mathcal{D}_U(x)$ is the unique sequence $z \in \tilde{R}_{x,U}$ satisfying $d_{x,U,1}(z) = S^i(x)$, where $i$ is the first occurrence in $x$ of a word of $U$.
Similarly, the sequence $\mathcal{D}_V(x)$ is such that $d_{x,V,1}(\mathcal{D}_V(x)) = S^j(x)$, where $j$ is the first occurrence in $x$ of a word of $V$.
The result thus directly follows from~\eqref{eq:lambda_x,U,V}.
\end{proof}

\begin{prop}
\label{prop:lambda_set_prim}
Let $x$ be a linearly recurrent sequence for the constant $K$.
Let $U$ and $V$ be subsets of $\mathcal{L}(x)$ such that all words of $U$ have the same length $\ell_U \geq 1$ and all words of $V$ have the same length $\ell_V \geq 1$.
Assume the following:
\begin{enumerate}
\item
$\ell_V \geq K(K+1) \ell_U$;
\item
any word of $V$ has a prefix in $U$;
\item
any return word to $V$ has length at least $\ell_V/K$;
\item
$\Omega(\mathcal{D}_{U}(x)) = \Omega(\mathcal{D}_{V}(x))$.
\end{enumerate}
Then the morphism $\lambda_{x,U,V}$ is a primitive substitution and $X_{\lambda_{x,U,V}} = \Omega (\mathcal{D}_{U}(x))$.
\end{prop}
\begin{proof}
To prove that $\lambda_{x,U,V}$ is primitive, we show that for all $(i,j) \in \tilde{R}_{x,V} \times \tilde{R}_{x,U}$, $j$ occurs in $\lambda_{x,U,V}(i)$.
Let $\tilde{\Theta}_{x,U}(j) = (w,u) \in \tilde{\mathcal{R}}_{x,U}$ be a return pair to $U$. 
We have $|w| \leq \max\{|r| \mid r \in \mathcal{R}_{x,u}\}$.
Hence, by Theorem~\ref{theo:encad}, we have $|wu| \leq (K+1)\ell_U$ and $wu$ occurs in any factor of length $(K+1)^2 \ell_U$ of $x$.
Since any return pair $(r,v) \in \tilde{\mathcal{R}}_{x,V}$ is such that $|rv| \geq \ell_V \frac{K+1}{K} \geq (K+1)^2 \ell_U$, we deduce from Theorem~\ref{theo:encad} that $j$ occurs in $\lambda_{x,U,V}(i)$ for all $i \in \tilde{R}_{x,V}$.

With Lemma~\ref{lemme:def de lambda} one has $\lambda_{x,U,V}(\mathcal{D}_V(x)) = S^k (\mathcal{D}_U(x))$ for some $k$.
As $\Omega(\mathcal{D}_{U}(x)) = \Omega(\mathcal{D}_{V}(x))$, we deduce that 
$X_{\lambda_{x,U,V}}$ is included in $\Omega (\mathcal{D}_{U}(x))$, which implies the equality by minimality of $\Omega(\mathcal{D}_{U}(x))$ (see Corollary~\ref{cor:derivee ensemble de mot reste UR}).
\end{proof}

\subsection{Some sufficient conditions for $f:A \to B$ to be a factor map}
Let $\sigma $ and $\tau $ be two primitive substitutions. 
In this section, we give some sufficient conditions for $f:A \to B$ to be a factor map from $(X_\sigma,S)$ to $(X_\tau,S)$.
These conditions are defined on pairs of words $(u,v) \in \mathcal{L}(\tau) \times \mathcal{L}(\tau)$. 
In the next sections, we will show that they are algorithmically checkable and that if $f:A \to B$ really defines a factor map, then there should exist such a pair $(u,v)$ where $u,v$ have length bounded by a computable constant.

\begin{lem}
\label{lemma:def of varphi_u}
Let $x \in A^{\mathbb{N}}$ and $y \in B^{\mathbb{N}}$ be uniformly reccurrent sequences and let $f: A^* \to B^*$ be a coding.
Assume that $u \in f(\cL(x))$ is a non-empty prefix of $y$ such that for $U = f^{-1}(\{u\}) \cap \cL(x)$, one has $f(\cR_{x,U}) \subset \cR_{y,u}$. 
Then there is a unique coding $\varphi_u:\tilde{R}_{x,U}^* \to R_{y,u}^*$ such that
\[
		f \circ d_{x,U,1} = \Theta_{y,u} \circ \varphi_u.
\]
In particular, if $f(\cL(x)) = \cL(y)$, then $\varphi_u$ defines a factor map from $\Omega(\mathcal{D}_U(x))$ to $\Omega(\mathcal{D}_u(y))$.
\end{lem}

\begin{proof}
The existence of $\varphi_u$ directly follows from $f(\cR_{x,U}) \subset \cR_{y,u}$ and the fact that $\Theta_{y,u}$ is injective (see Proposition~\ref{prop: return words code}).
If $f(\cL(x)) = \cL(y)$, then by uniform recurrence of $\mathcal{D}_U(x)$ and $\mathcal{D}_u(y)$ (see Proposition~\ref{prop:def suite derivee singleton} and Corollary~\ref{cor:derivee ensemble de mot reste UR}), it suffices to show that $\varphi_u(\cL(\mathcal{D}_U(x)))$ is included in $\cL(\mathcal{D}_u(y))$, which is direct from the definition of $\varphi_u$. 
\end{proof}

\begin{theo}
\label{thm:sufficient_conditions}
Let $\sigma:A^* \to A^*$ and $\tau:B^* \to B^*$ be primitive substitutions and $x \in A^\mathbb{N}$ and $y \in B^\mathbb{N}$ be fixed points of $\sigma$ and $\tau$ respectively.
Let $K_\sigma$ and $K_\tau$ denote some constants of linear recurrence of $x$ and $y$ respectively, and set $K = \max\{K_\sigma,K_\tau\}$.
Let $f:A^* \to B^*$ be a coding.
Assume that there exist non-empty words $u,v \in f (\cL(\sigma))$ that are prefixes of $y$ and satisfy, for $U = f^{-1} (\{ u\}) \cap \cL(\sigma)$ and $V = f^{-1}(\{v\}) \cap \cL(\sigma)$,
\begin{enumerate}
\item
$|v| \geq K(K+1)|u|$;
\item
\label{item condition 0}
$f(\cR_{x,U}) \subset \cR_{y,u}$
and
$f(\cR_{x,V}) \subset \cR_{y,v}$;
\item
$\varphi_u = \varphi_v$;
\item
\label{item condition 2}
$\mathcal{D}_u(y) = \mathcal{D}_v(y)$
and
$\Omega(\mathcal{D}_{U}(x)) = \Omega(\mathcal{D}_{V}(x))$.
\end{enumerate}
Then $f$ defines a factor map from $(X_\sigma,S)$ to $(X_\tau,S)$.
\end{theo}

\begin{proof}
By minimality, it suffices to show that $f(\mathcal{L}(\sigma)) \subset \mathcal{L}(\tau)$.
Since for any return word $w$ to $V$ in $x$, $f(w)$ is a return word to $v$ in $y$ and $|w|=|f(w)|$, we have from Theorem~\ref{theo:encad} that $|w| \geq |v|/K$.
We are thus in the conditions of 
Proposition~\ref{prop:derived sequence fixed point} and Proposition~\ref{prop:lambda_set_prim} and we can define the primitive morphisms $\lambda_{y,u,v}$ and 
$\lambda_{x,U,V}$ such that $X_{\lambda_{y,u,v}} = \Omega(\cD_u(y))$ and $X_{\lambda_{x,U,V}} = \Omega(\cD_U(x))$.
Using Proposition~\ref{prop:derived sequence fixed point}, Equation~\eqref{eq:lambda_x,U,V} and Lemma~\ref{lemma:def of varphi_u}, we get
\begin{eqnarray*}
	\Theta_{y,u} \circ \lambda_{y,u,v} \circ \varphi_v
	&=&
	\Theta_{y,v} \circ \varphi_v	\\
	&=&
	f \circ d_{x,V,1}	\\
	&=&
	f \circ d_{x,U,1} \circ \lambda_{x,U,V} \\
	&=&
	\Theta_{y,u} \circ \varphi_u \circ \lambda_{x,U,V}.
\end{eqnarray*}
Using Proposition~\ref{prop: return words code} and setting $\varphi = \varphi_u = \varphi_v$, we then obtain
\[
	\lambda_{y,u,v} \circ \varphi = \varphi \circ \lambda_{x,U,V}.
\]

By definition of return words, we have 
\begin{eqnarray*}
	\mathcal{L}(\sigma) 	
	=& \mathcal{L}(d_{x,U,1}(\mathcal{D}_U(x))) 	
	&= 	\mathcal{L}(d_{x,U,1}(\mathcal{L}(\lambda_{x,U,V})))		\\
	\mathcal{L}(\tau)	
	=& \mathcal{L}(\Theta_{y,u}(\mathcal{D}_u(y))) 
	&= 	\mathcal{L}(\Theta_{y,u}(\mathcal{L}(\lambda_{y,u,v})))		
\end{eqnarray*}
Let $w$ be a word in $\mathcal{L}(\sigma)$.
By primitiveness, there is a positive integer $n$ such that $w$ occurs in $d_{x,U,1}(\lambda_{x,U,V}^n(1))$.
Thus $f(w)$ occurs in the word
\begin{eqnarray*}
	f (d_{x,U,1}(\lambda_{x,U,V}^n(1)))
	&=& \Theta_{y,u}(\varphi(\lambda_{x,U,V}^n(1)))	\\
	&=& \Theta_{y,u}(\lambda_{y,u,v}^n(\varphi(1))),
\end{eqnarray*}
which belongs to the set $\mathcal{L}(\Theta_{y,u}(\mathcal{L}(\lambda_{y,u,v}))) = \mathcal{L}(\tau)$ and this concludes the proof.
\end{proof}

\subsection{The sufficient conditions of Theorem~\ref{thm:sufficient_conditions} are necessary and algorithmically checkable}
\label{subsection:return_subst_singleton}

The aim of this section is to show that for $f:A \to B$ to be a factor map from $(X_\sigma,S)$ to $(X_\tau , S)$, there should exist some words $u,v$ of bounded length (with a computable bound) that satisfy the hypotheses of Theorem~\ref{thm:sufficient_conditions}.
We also show that these hypotheses are algorithmically checkable.
This can be achieved through return substitutions.

\subsubsection{Return substitutions for a single word}

We first recall the classical notion of return substitution that has been introduced in~\cite{Durand:1998}.
Observe that the notions that we develop here are closely related to the discussion made in Section~\ref{subsection:return words}. 
We decide to introduce them in this section because we want to extend them to the more general notion of return words to a set of words. 

\begin{prop}[\cite{Durand:1998}]
\label{prop: return morphism exists}
Let $\sigma$ be an aperiodic primitive substitution, $x \in A^\N$ be a fixed point of $\sigma$ and $u$ be a non-empty prefix of $x$.
The sequence $\mathcal{D}_u(x)$ is the admissible fixed point starting with $1$ of the primitive substitution $\sigma_{x,u}: R_{x,u}^* \to R_{x,u}^*$ defined by
\begin{equation}
\label{eq: def return morphism}
	\Theta_{x,u} \circ \sigma_{x,u} = \sigma \circ \Theta_{x,u}.
\end{equation}
Furthermore, if $K$ is a constant of linear recurrence of $x$, then $\mathcal{D}_u(x)$ is linearly recurrent for the constant $K^3$.
\end{prop}

\begin{defi}
The morphism $\sigma_{x,u}$ of the previous proposition is called a {\em return substitution} (w.r.t. $u$).
\end{defi}

From Theorem~\ref{theo:encad} and the definition of $\sigma_{x,u}$ we deduce the following corollary.

\begin{cor}
\label{cor:nbr_de_return_subst}
Let $\sigma$ be an aperiodic primitive substitution and $x \in A^\N$ be a fixed point of $\sigma$.
We have
\[
	\card(\{\sigma_{x,x_{[0,n]}} \mid n \in \N \}) \leq  Q_\sigma = {\left( 1 + (K_\sigma+1)^3\right) }^{|\sigma| K_\sigma^2(1+(K_\sigma+1)^3)}.
\]
\end{cor}

The following three lemmas are clear from~\cite{Durand:1998}, but not exactly stated as follows.
We give the proofs as they are a key point of one decision procedure.

\begin{lem}
\label{lemma:returnsubalgo}
Let $\sigma$ be an aperiodic primitive substitution, $x\in A^\N$ be a fixed point of $\sigma$ and $u$ be a non-empty prefix of $x$.
The morphisms $\sigma_{x,u}$ and $\Theta_{x,u}$ are algorithmically computable. 
\end{lem}
\begin{proof}[Sketch of the algorithm.]
Since $x$ is a fixed point of $\sigma$, the word $u$ is a prefix of $\sigma(u)$ and there is an integer $n$ such that $u$ occurs at least twice in $\sigma^n(u)$.
We define $\Theta_{x,u}(1)$ as the shortest non-empty word $w$ such that $wu$ is a prefix of $\sigma^n(u)$.
We then compute $\sigma (\Theta_{x,u}(1))$. 
By construction, the word $\sigma (\Theta_{x,u}(1)) u$ is a prefix of $x$ and has $u$ as prefix.
By Proposition~\ref{prop: return words code}, there exist some unique words $w_1,w_2, \dots,w_k \in \mathcal{R}_{x,u}$ such that $\sigma (\Theta_{x,u}(1)) = \Theta_{x,u}(1)w_1 \cdots w_k$.
Furthermore the word $\Theta_{x,u}(1)w_1 \cdots w_k u$ is a prefix of $x$.
We add some letters $2,3,\dots,\ell \leq k$ to $R_{x,u}$ and define $\Theta_{x,u}$ on $\{1,\dots,\ell\}$ by $\{\Theta_{x,u}(i) \mid 1 \leq i \leq \ell \} = \{\Theta_{x,u}(1),w_1,\dots,w_k\}$ and in the order of occurrences of the words $w_i$'s in $\sigma (\Theta_{x,u}(1))$.
This also defines the image $\sigma_{x,u}(1)$.
We iterate the process on the added letters.
The algorithm stops when no new return word appears when we compute the images $\sigma (\Theta_{x,u}(i))$ for $i \in R_{x,u}$.
\end{proof}

\begin{exem}
\label{ex:TM algo return substitution}
Let $\nu$ be the Thue-Morse substitution and $y$ be the Thue-Morse sequence.
Let us follow the algorithm described above to compute the return substitution $\nu_{x,a}$.

\begin{enumerate}
\item
We have $\nu^2(a)=abba$, so we set $\Theta_{y,a}(1)=abb$.

\item
We compute $\nu \circ \Theta_{y,a}(1) = \nu(abb) = abbaba$.
The return words $ab$ and $a$ do not belong yet to $\cR_{y,a}$ so we add the letters $2$ and $3$ to $R_{y,a}$ and set $\Theta_{y,a}(2)=ab$, $\Theta_{y,a}(3)=a$ and $\nu_{y,a}(1)=123$.

\item
We compute $\nu \circ \Theta_{y,a}(2) = \nu(ab) = abba$.
The return words $abb$ and $a$ already belong to $\cR_{y,a}$ so we do not add any letter to $R_{y,a}$. 
We set $\nu_{y,a}(2)=13$.

\item
We compute $\nu \circ \Theta_{y,a}(3) = \nu(a) = ab$.
The return word $ab$ already belongs to $\cR_{y,a}$ so we do not add any letter to $R_{y,a}$. 
We set $\nu_{y,a}(3)=2$.

\item
No more letters have to be checked so the algorithm stops and returns 
\[
	\Theta_{y,a}: 
	\begin{cases}
		1 \mapsto abb	\\
		2 \mapsto ab	\\
		3 \mapsto a
	\end{cases}
	\qquad \text{and} \qquad
	\nu_{y,a}: 
	\begin{cases}
		1 \mapsto 123	\\
		2 \mapsto 13	\\
		3 \mapsto 2
	\end{cases}.
\]
\end{enumerate}
\end{exem}

\begin{lem}
\label{lemma:links between derived}
Let $x$ be an aperiodic uniformly recurrent sequence. 
If $u$ is a non-empty prefix of $x$, $v$ is a non-empty prefix of $\mathcal{D}_u(x)$ and $w = \Theta_{x,u}(v)u$, then $w$ is a non-empty prefix of $x$, $\mathcal{D}_w(x) = \mathcal{D}_v(\mathcal{D}_u(x))$ and
\begin{equation}
\label{eq:link between Theta}
	\Theta_{x,w} = \Theta_{x,u} \circ \Theta_{\mathcal{D}_u(x),v}.
\end{equation}
Furthermore, if $x$ is a fixed point of a primitive substitution $\sigma$ and if $\tau$ is the return substitution $\sigma_{x,u}$, then
\[
	\sigma_{x,w} = \tau_{\mathcal{D}_u(x),v}.
\]
\end{lem}
\begin{proof}
The first part of the statement is exactly~\cite[Proposition 2.6 (5)]{Durand:1998} and Equation~\eqref{eq:link between Theta} directly follows from the definition of these maps.
To alleviate notation, we set $y = \mathcal{D}_u(x)$.
By Proposition~\ref{prop: return morphism exists}, we have the equalities
\begin{align*}
	\Theta_{x,w} \circ \sigma_{x,w} &= \sigma \circ \Theta_{x,w}	\\
	\Theta_{x,u} \circ \sigma_{x,u} &= \sigma \circ \Theta_{x,u}	\\
	\Theta_{y,v} \circ \tau_{y,v} &= \tau \circ \Theta_{y,v}	
\end{align*}
from which we deduce
\begin{align*}
	\Theta_{x,u} \circ \Theta_{y,v} \circ \sigma_{x,w}
	&= \Theta_{x,w} \circ \sigma_{x,w}	\\
	&= \sigma \circ \Theta_{x,w}			\\
	&= \sigma \circ \Theta_{x,u} \circ \Theta_{y,v}	\\
	&= \Theta_{x,u} \circ \sigma_{x,u} \circ \Theta_{y,v}	\\
	&= \Theta_{x,u} \circ \tau \circ \Theta_{y,v}	\\
	&= \Theta_{x,u} \circ \Theta_{y,v} \circ \tau_{y,v}.		
\end{align*}
Using Proposition~\ref{prop: return words code}, we finally get 
\[
	\sigma_{x,w} = \tau_{y,v}.
\]
\end{proof}

\begin{exem}
We continue the previous example and ilustrate Lemma~\ref{lemma:u_n enumerable} with $y$ the Thue-Morse sequence.
We have
\[
	z 
	= \mathcal{D}_{a}(y)
	= \nu_{y,a}^\omega(1) 
	= 1231321232131231321312321231321232131232 \cdots
\]
If we set $\tau = \nu_{y,a}$ and compute $\Theta_{z,1}$ and $\tau_{z,1}$ following Lemma~\ref{lemma:returnsubalgo}, we get
\[
	\Theta_{z,1}
	\begin{cases}
		1 \mapsto 123	\\
		2 \mapsto 132	\\
		3 \mapsto 1232	\\
		4 \mapsto 13
	\end{cases}
	\qquad \text{and} \qquad
	\tau_{z,1} : 
	\begin{cases}
		1 \mapsto 12	\\
		2 \mapsto 34	\\
		3 \mapsto 124		\\
		4 \mapsto 3
	\end{cases}.
\] 
Setting $u_1 = \Theta_{y,a}(1)a = abba$, we find
\[
	\Theta_{y,u_1}: 
	\begin{cases}
		1 \mapsto abb ab a	\\
		2 \mapsto abb a ab	\\
		3 \mapsto abb ab a ab	\\
		4 \mapsto abb a
	\end{cases}
	\qquad \text{and} \qquad
	\nu_{y,u_1} : 
	\begin{cases}
		1 \mapsto 12	\\
		2 \mapsto 34	\\
		3 \mapsto 124		\\
		4 \mapsto 3
	\end{cases}
\]
and we indeed check that $\nu_{y,u_1} = \tau_{z,1}$ and $\Theta_{y,u_1} = \Theta_{y,a} \circ \Theta_{z,1}$.
\end{exem}

\begin{lem}
\label{lemma:u_n enumerable}
Let $\sigma$ be an aperiodic primitive substitution, $x\in A^\N$ be a fixed point of $\sigma$.
There is a recursively enumerable sequence $(u_n)_{n \geq 0}$ of non-empty prefixes of $x$ such that 
\begin{enumerate}
\item
$|u_0|=1$ and for every $n$, $\frac{K_\sigma+1}{K_\sigma} |u_n| \leq |u_{n+1}| \leq (K_\sigma +1)|u_n|$;
\item
there exists $m,p$ such that $p >0$, $0 \leq m \leq Q_\sigma - p$ and for every $i \in \{0,1,\dots,p-1\}$ and every $n \in \mathbb{N}$, $\sigma_{x,u_{m+i+pn}} = \sigma_{x,u_{m+i}}$.
\end{enumerate}
\end{lem}

\begin{proof}
Let us  define the following sequence of derived sequences $(x^{(n)})_{n \geq 0}$ by
\begin{align*}
	x^{(0)} 	&= 	\mathcal{D}_{x_0}(x)	\\
	x^{(n+1)} 	&=	\mathcal{D}_{1}(x^{(n)}), \quad n \geq 0.
\end{align*}
By Proposition~\ref{prop: return morphism exists}, each $x^{(n)}$ is the fixed point starting with $1$ of the return substitution $\sigma_n$ defined by
\begin{align*}
	\Theta_{x,x_0} \circ \sigma_0 	&= 	\sigma \circ \Theta_{x,x_0}	 \\
	\Theta_{x^{(n)},1} \circ \sigma_{n+1} 	&= 	\sigma_n \circ \Theta_{x^{(n)},1}, \quad n \geq 0.
\end{align*}
We recursively define the sequence $(u_n)_{n \geq 0}$ by $u_0 = x_0$ and
\[
	u_{n+1} = \Theta_{x,u_n}(1)u_n.
\]
By Theorem~\ref{theo:encad}, one has $\frac{K_\sigma+1}{K_\sigma} |u_n| \leq |u_{n+1}| \leq (K_\sigma +1)|u_n|$ for each $n$ and, 
using Lemma~\ref{lemma:returnsubalgo}, the sequence $(u_n)_{n \geq 0}$ is recursively enumerable.
Furthermore, each $u_n$ is a prefix of $x$ and, using Lemma~\ref{lemma:links between derived}, it is easily seen by induction that for all $n \geq 0$,
\[
	x^{(n)} = \mathcal{D}_{u_n}(x)
	\quad \text{and} \quad
	\sigma_n = \sigma_{x,u_n}.
\]
By Corollary~\ref{cor:nbr_de_return_subst}, there exists $m,p$ such that $p >0$, $0 \leq m \leq Q_\sigma - p$ and $\sigma_m = \sigma_{m+p}$.
This implies that $x^{(m)} = x^{(m+p)}$ and so by induction on $n$, for all $n \geq m$, 
$x^{(n)} = x^{(n+p)}$, $\Theta_{x^{(n)},1} = \Theta_{x^{(n+p)},1}$ and $\sigma_n = \sigma_{n+p}$.
\end{proof}

\begin{exem}
We continue the previous example and ilustrate Lemma~\ref{lemma:u_n enumerable} with $y$ the Thue-Morse sequence.
Considering the notation of the proof, we have $u_0 = a$, $u_1 = abba$ and
\begin{align*}
	y^{(0)} 
	&= \mathcal{D}_{u_0}(y)
	= \nu_{y,u_0}^\omega(1); 
		\\
	y^{(1)}
	&= \mathcal{D}_{1}(y^{(0)})
	= \mathcal{D}_{u_1}(y)= \nu_{y,u_1}^\omega(1)	\\
	&= 1234124312343124123412431241234312341243 \cdots
\end{align*}
Since $\nu_{y,u_0} \neq \nu_{y,u_1}$, we continue and set $u_2 = \Theta_{y,u_1}(1)u_1 = abba ba abba$ and we find (by computing $\Theta_{y^{(1)},1}$ and the return substitution associated with $y^{(2)} = \mathcal{D}_1(y^{(1)})$):
\begin{align*}
	\Theta_{y,u_2}: 
	&\begin{cases}
		1 \mapsto abb ab a	  abb a ab   abb ab a ab   abb a \\
		2 \mapsto abb ab a	  abb a ab   abb a    abb ab a ab	\\
		3 \mapsto abb ab a	  abb a ab   abb ab a ab   abb a    abb ab a ab	\\
		4 \mapsto abb ab a	  abb a ab   abb a
	\end{cases}
	\\
	\nu_{y,u_2} : 
	&\begin{cases}
		1 \mapsto 12	\\
		2 \mapsto 34	\\
		3 \mapsto 124		\\
		4 \mapsto 3
	\end{cases}.
\end{align*}
As $\nu_{y,u_1} = \nu_{y,u_2}$, we have $y^{(2)} = y^{(1)}$ and we conclude that when setting $u_{n+1} = \Theta_{y,u_n}(1)u_n$ for every $n \geq 1$, we have $\nu_{y,u_n} = \nu_{y,u_1}$ for all $n$. 
\end{exem}

\subsubsection{Return substitutions for a set of words $U$}
\label{subsection:return_subst_set}

When $x$ is an admissible fixed point of a primitive substitution $\sigma$, Proposition~\ref{prop: return morphism exists} states that the derived sequence w.r.t. any prefix $u$ of $x$ is again an admissible  fixed point of a primitive substitution $\sigma_{x,u}$.
This result can be generalized to the case where $U$ is not a singleton but we need to be more careful.
Indeed, the existence of $\sigma_{x,u}$ uses the fact that $\sigma(\cR_{x,u})u \subset \cR_{x,u}^*u \cap \cL(\sigma)$, which may be not be the case when $U$ is not a singleton.

\begin{exem}
Consider the Thue-Morse substitution $\nu$ and the set $U = \{aa,bb\}$.
We obtain $\cR_{y,U} = \{aa,bb,aaba,bbab\}$ and, for instance, $\nu(aa) = abab$ is not prefix-comparable to any word in $\cR_{y,U}^* U$.
\end{exem}

\begin{prop}
\label{prop: return morphism to prefix code}
Let $\sigma:A^* \to A^*$ be an aperiodic primitive substitution, $x$ be an element of $X_\sigma$, $f:A^* \to B^*$ be a coding such that $f(x)$ is not periodic, $u$ a non-empty word in  $f(\cL(\sigma))$ and $U = f^{-1}(u) \cap \cL(\sigma)$.
There is a computable primitive substitution $\sigma_{x,U}: \tilde{R}_{x,U}^* \to \tilde{R}_{x,U}^*$ such that $X_{\sigma_{x,U}} = \Omega(\mathcal{D}_U(x))$.
Furthermore, $\card(\tilde{R}_{x,U}) \leq K_\sigma(K_\sigma+1)^2$ and  there is a computable constant $\tilde{Q}_\sigma$ such that
\[
	\card(\{\sigma_{x,U} \mid \varepsilon \neq u \in f(\cL(\sigma)), U = f^{-1}(u) \cap \cL(\sigma) \}) 
	\leq \tilde{Q}_\sigma
\]
\end{prop}

\begin{proof}
Let us firstly define the morphism $\sigma_{x,U}$.
First observe that for any return word $w$ to $U$ in $x$, $f(w)$ is a return word to $u$ in $f(x)$ and has length $|w|$.
The sequence $f(x)$ is obviously linearly recurrent and $K_\sigma$ is a constant of linear recurrence for it. 
As $f(x)$ is aperiodic by assumption, Theorem~\ref{theo:encad} implies that any return word $w$ to $U$ in $x$ satisfies
\begin{equation}
\label{eq:longueur de |w| dans f(x)}
	\frac{|u|}{K_\sigma} \leq |w| \leq K_\sigma |u|.
\end{equation}
It also implies that for every $(w,u) \in \tilde{\cR}_{x,U}$, the word $wu$ occurs in every word of length $(K_\sigma+1)|wu|$.
This shows that 
\begin{equation}
\label{eq:cardinal return pairs}
	\card(\tilde{\cR}_{x,U}) 
	\leq 
	\frac
	{(K_\sigma+1)^2|u|}
	{\min_{w \in \cR_{x,U}}|w|}
	\leq
	K_\sigma (K_\sigma+1)^2 = C.
\end{equation}

Let $k_U$ be the smallest integer such that for every $(w,u) \in \tilde{\mathcal{R}}_{x,U}$, both $|\sigma^{k_U}(w)|_U \geq 2$ and $|\sigma^{k_U}(u)|_U \geq 1$, where $|z|_U$ stands for the number of occurrences of words of $U$ in $z$.
The constant $k_U$ is obviously computable.

For any $w \in \mathcal{R}_{x,U} \cup U$, we set 
\[
	\sigma^{k_U} (w) = p(w) m(w) s(w) 
\]
where $m(w)$ belongs to $U$ and $p(w)m(w)$ has a unique occurrence of a word of $U$.
Now take $(w,u) \in \tilde{\mathcal{R}}_{x,U}$. 
Then 
\[
	\sigma^{k_U} (wu) = p(w) m(w) s(w) p(u) m(u) s(u) 
\]
and, by Proposition~\ref{prop:return pair code}, there exists a unique admissible word $(w_1,u_1) \cdots (w_\ell, u_\ell)$ in $\tilde{\mathcal{R}}_{x,U}^*$ such that
\[
	w_1 \cdots w_\ell u_\ell = m(w) s(w) p(u) m(u)
\]
We set 
\[
\tilde{\sigma}_{x,U} (w,u) = (w_1,u_1) \dots  (w_\ell, u_\ell) .
\]
This defines an endomorphism of $\tilde{\mathcal{R}}_{x,U}^*$ and $\tilde{\sigma}_{x,U}$ is clearly computable.
We then define $\sigma_{x,U}$ as the unique endomorphism of $\tilde{R}_{x,U}^*$ satisfying 

\begin{equation*}
\label{def:sigmaU}\
\tilde{\sigma}_{x,U} \circ \tilde{\Theta}_{x,U} 
= 
\tilde{\Theta}_{x,U} \circ \sigma_{x,U}.
\end{equation*}
It is also computable.
Using equation~\eqref{eq:longueur de |w| dans f(x)} and the choice of $k_U$, we then have
\begin{align}
2 \leq \langle \sigma_{x,U} \rangle \leq |\sigma_{x,U}| 
	&\leq 
	\dfrac
	{\max_{(w,u) \in \tilde{\mathcal{R}}_{x,U}}|\sigma^{k_U}(wu)|}
	{\min_{(w,u) \in \tilde{\mathcal{R}}_{x,U}}|w|}
\nonumber
	\\
	&\leq  
	|\sigma^{k_U}| 
	\dfrac
	{\max_{(w,u) \in \tilde{\mathcal{R}}_{x,U}}|wu|}
	{\min_{w \in \mathcal{R}_{x,U}}|w|}
\nonumber
	\\
	&\leq 
	|\sigma^{k_U}| 
	K_\sigma(K_\sigma+1).	
\end{align}

Observe that if $k$ is the smallest integer such that $\langle \sigma^{k} \rangle \geq 2 K_\sigma (K_\sigma +1)$, then from  Equation~\eqref{eq:longueur de |w| dans f(x)} and Theorem~\ref{theo:encad}, for every $(w,u) \in \tilde{\mathcal{R}}_{x,U}$, every word of $U$ occurs twice both in $\sigma^k(w)$ and in $\sigma^k(u)$.
This implies that $k_U \leq k$. 
In particular, $k$ does not depend on $U$ and, $K_\sigma$ being computable, so is $k$.
The constant $D=|\sigma^k| K_\sigma(K_\sigma+1)$ is thus also computable and does not depend on $U$.
Together with Equation~\eqref{eq:cardinal return pairs}, this shows that
\[
	\card(\{\sigma_{x,U} \mid \varepsilon \neq u \in f(\cL(\sigma)), U = f^{-1}(u) \cap \cL(\sigma) \})
	\leq C^{C(D+1)} = \tilde{Q}_\sigma
\]
and the constant $\tilde{Q}_\sigma$ is computable.

Now let us show that $\sigma_{x,U}$ is primitive and satisfies $X_{\sigma_{x,U}} = \Omega(\cD_{U}(x))$.
By Corollary~\ref{cor:derivee ensemble de mot reste UR}, the subshift $\Omega(\cD_{U}(x))$ is minimal.
Since $\langle \sigma_{x,U}^n \rangle$ goes to infinity when $n$ increases, it suffices to show that each word $\sigma_{x,U}^n(a)$ belongs to $\cL(\cD_{U}(x))$.
Indeed, the minimality of $\Omega(\cD_{U}(x))$ will imply the primitiveness of $\sigma_{x,U}$, hence the minimality of $X_{\sigma_{x,U}}$, henceforth the equality $X_{\sigma_{x,U}} = \Omega(\cD_{U}(x))$.

We proceed by induction on $n$.
By construction of $\sigma_{x,U}$, the result is true for $n = 1$.
Let us suppose it is true for $n-1$.
Let $a \in \tilde{R}_{x,U}$ and let us write $\sigma_{x,U}^{n-1}(a) = a_1 a_2 \cdots a_\ell \in \cL(\mathcal{D}_U(x))$, with $a_1,\dots,a_\ell \in \tilde{R}_{x,U}$.

For $i \in \{1,\dots,\ell\}$, let us write $\tilde{\Theta}_{x,U}(a_i) = (w_i,u_i)$.
As $a_1 a_2 \cdots a_\ell $ belong  to $\cL(\mathcal{D}_U(x))$, the word $w_1 w_2 \cdots w_\ell u_\ell $ belongs to $\cL(x)$ and for all $i$ such that $1 \leq i \leq \ell$, $w_1 w_2 \cdots w_i u_i$ is a prefix of $w_1 w_2 \cdots w_\ell u_\ell$.
We have to show that $\sigma_{x,U}(a_1 a_2 \cdots a_\ell) $ belongs to  $\cL(\cD_U(x))$.
By definition of $\sigma_{x,U}$, we have for all $i \in \{1,\dots, \ell\}$, $\sigma_{x,U}(a_i) = b_{i,1} \cdots b_{i,L_i}$, where $b_{i,1}, \dots, b_{i,L_i}$ are the unique letters in $\tilde{R}_{x,U}$ that satisfy both conditions
\begin{enumerate}
\item
$\sigma^{k_U}(w_i u_i) = p(w_i) d_{x,U,1}(b_{i,1} \cdots b_{i,L_i}) m(u_i) s(u_i)$;
\item
for all $j$ such that $j \in \{1 \dots,L_i\}$, the word $d_{x,U,1}(b_{i,1} \cdots b_{i,j})d_{x,U,2}(b_{i,j})$ is a prefix of $$d_{x,U,1}(b_{i,1} \cdots b_{i,L_i})m(u_i).$$
\end{enumerate}
Observe that for all $i \in \{1,\dots,\ell-1\}$, the word $w_i u_i$ is a prefix of $w_i w_{i+1}u_{i+1}$ and we have
\[
	\sigma^{k_U} (w_iu_i)	
	= 
	p(w_i) m(w_i) s(w_i) p(u_i) m(u_i) s(u_i)
\]
and
\begin{align*}
&	\sigma^{k_U} (w_iw_{i+1}u_{i+1})\\	 	
	= &
	p(w_i) m(w_i) s(w_i) 
	p(w_{i+1}) m(w_{i+1}) s(w_{i+1}) 
	p(u_{i+1}) m(u_{i+1}) s(u_{i+1}).
\end{align*}
Consequently, we have $p(u_i) = p(w_{i+1})$ and $m(u_i) = m(w_{i+1})$.

In definitive, we have that $\sigma^{k_U}(w_1 w_2 \cdots w_\ell u_\ell) $ belongs to $ \cL(x)$ and the letters $b_{1,1},\dots,b_{1,L_1}$, $\dots$, $b_{\ell,1},\dots,b_{\ell,L_\ell}$ are the unique letters in $\tilde{R}_{x,U}$ that satisfy
\begin{enumerate}
\item
$\sigma^{k_U}(w_1 w_2 \cdots w_\ell u_\ell) = p(w_1)d_{x,U,1}(b_{1,1}\cdots b_{1,L_1}\cdots b_{\ell,1} \cdots b_{\ell,L_\ell}) m(u_\ell) s(u_\ell)$;
\item
for all $i \in \{1,\dots,\ell\}$, for all $j \in \{1,L_i\}$, the word 
\[
	d_{x,U,1}(b_{1,1}\cdots b_{1,\ell_1}\cdots b_{i,1} \cdots b_{i,j})d_{x,U,2}(b_{i,j})
\]
is a prefix of 
\[
	d_{x,U,1}(b_{1,1}\cdots b_{1,\ell_1}\cdots b_{\ell,1} \cdots b_{\ell,L_\ell}) m(u_\ell).
\]
\end{enumerate}
Hence, the word $\sigma_{x,U}(a_1 \cdots a_\ell)$ belongs to $\cL(\cD_U(x))$, which ends the proof.
\end{proof}

\begin{defi}
The substitution $\sigma_{x,U}$ of the previous proposition is called a {\em return substitution} (w.r.t. $U$).
\end{defi}

\begin{exem}
Let us consider our running example with the Fibonacci sequence and follow the algorithm given by Proposition~\ref{prop: return morphism to prefix code}.
Recall that we have $U = \{aa,ab\}$, $\tilde{\cR}_{x,U} = \{(ab, aa),(a, ab),(ab, ab)\}$ and
\[
	\tilde{\Theta}_{y,U}:
		\begin{cases}
			1 \mapsto (ab,aa)	\\
			2 \mapsto (a,ab)		\\
			3 \mapsto (ab, ab)
		\end{cases}.
\]
We have $|\varphi^2(a)|_U = 1$ and $|\varphi^3(a)|_U = 2$ so that, with the notation of the proof, $k_U = 3$.
Computing the images of $U \cup \cR_{x,U}$ under $\varphi^3$, we get
\[
	\varphi^3:
		\begin{cases}
			a  & \mapsto	abaab		\\
			ab & \mapsto  	abaababa	\\	
			aa & \mapsto 	abaababaab	
		\end{cases},
\]
hence 
\[
	\tilde{\varphi}_{x,U}:
		\begin{cases}
			(a,ab)  & \mapsto (ab,aa)(a,ab)	\\
			(ab,ab) & \mapsto 
				(ab,aa)(a,ab)(ab,ab)(ab,aa)(a,ab)	\\
			(ab,aa) & \mapsto
				(ab,aa)(a,ab)(ab,ab)(ab,aa)(a,ab)
		\end{cases}
\]
and finally
\[
	\varphi_{x,U}:
		\begin{cases}
			1 \mapsto 12312	\\
			2 \mapsto 12 	\\
			3 \mapsto 12312	
		\end{cases}.
\]
\end{exem}

\begin{exem}
Let us consider our running example with the Thue-Morse sequence and follow the algorithm given by Proposition~\ref{prop: return morphism to prefix code}.
Recall that we have $U = \{aa,bb\}$, $$\tilde{\cR}_{y,U} = \{(aa, bb),(bb, aa),(aaba, bb), (bbab, aa)\}$$ and
\[
	\tilde{\Theta}_{y,U}:
		\begin{cases}
			1 \mapsto (bbab, aa)	\\
			2 \mapsto (aa, bb)		\\
			3 \mapsto (bb, aa)		\\
			4 \mapsto (aaba, bb)
		\end{cases}.
\]
We have $|\nu(aa)|_U = 0$ and $|\nu^2(aa)|_U =|\nu^2(bb)|_U = 3$ so that, with the notation of the proof, $k_U = 2$

Computing the images of $U \cup \cR_{y,U}$ under $\nu^2$, we get
\[
	\nu^2:
		\begin{cases}
			aa 		& \mapsto	abbaabba	\\
			aaba 	& \mapsto	abbaabbabaababba	\\
			bb 		& \mapsto	baabbaab	\\
			bbab 	& \mapsto	baabbaababbabaab	
		\end{cases},
\]
hence 
\[
	\tilde{\nu}_{y,U}:
		\begin{cases}
			(aa,bb)   & \mapsto (bb,aa)(aa,bb) (bbab,aa)	\\
			(aaba,bb) & \mapsto (bb,aa)(aa,bb) (bbab,aa) (aaba,bb)(bbab,aa)	\\
			(bb,aa)   & \mapsto (aa,bb)(bb,aa) (aaba,bb)	\\
			(bbab,aa) & \mapsto (aa,bb)(bb,aa) (aaba,bb) (bbab,aa)(aaba,bb)	
		\end{cases}
\]
and
\[
	\nu_{y,U}:
		\begin{cases}
			1 \mapsto 23414	 \\
			2 \mapsto 321 	\\
			3 \mapsto 234	\\
			4 \mapsto 32141
		\end{cases}.
\]
\end{exem}

\subsubsection{It suffices to compute a finite number of return substitutions}

We are now ready to state a necessary and sufficient condition for a coding to be a factor map between two minimal substitution subshifts.

\begin{theo}
\label{theo:decidcoding}
Let $\sigma:A^* \to A^*$ and $\tau:B^* \to B^*$ be aperiodic primitive substitutions and let $x \in A^\mathbb{N}$ and $y \in B^\mathbb{N}$ be fixed points of $\sigma$ and $\tau$ respectively.
Let $f:A^* \to B^*$ be a coding and $K = \max\{K_\sigma,K_\tau\}$.
There is a computable constant $C$ such that 
$f$ defines a factor map from $(X_\sigma,S)$ to $(X_\tau,S)$ if and only if there exist non-empty words $u,v \in f(\cL(X_\sigma))$ that are prefixes of $y$ and satisfy the following conditions, where $U = f^{-1}(\{ u\}) \cap \cL(\sigma)$ and $V = f^{-1}(\{ v\} )\cap \cL(\tau)$:
\begin{enumerate}
\item
$K(K+1)|u| \leq |v| \leq C$;
\item
$f(\cR_{x,U}) \subset \cR_{y,u}$
and
$f(\cR_{x,V}) \subset \cR_{y,v}$;
\item
$\varphi_u = \varphi_v$;
\item
$\tau_{y,u} = \tau_{y,v}$ and $\sigma_{x,U} = \sigma_{x,V}$.
\end{enumerate}
\end{theo}

\begin{proof}
Let us first define the computable constant $C$.
Let $Q_\tau$ and $\tilde{Q}_\sigma$  be the computable constants of Corollary~\ref{cor:nbr_de_return_subst} and Proposition~\ref{prop: return morphism to prefix code} respectively.
Let also $D = (K+1)^{3(K+1)^3}$ and let $k$ denote the smallest positive integer such that $\left(\frac{K+1}{K}\right)^k \geq K (K+1)$.
It is computable.
We set $C = (K(K+1))^{Q_\sigma + k D Q_\sigma \tilde{Q}_\sigma}$.

Assume that $u,v\in f(\cL(X_\sigma))$ satisfy the hypotheses.
By Propositions~\ref{prop: return morphism exists} and~\ref{prop: return morphism to prefix code}, we get
\begin{align*}
\mathcal{D}_u(y) = \tau_{y,u}^\omega(1) 
&= \tau_{y,v}^\omega(1) = \mathcal{D}_v(y);	\\
\Omega(\mathcal{D}_U(x)) = X_{\sigma_{x,U}} 
&= X_{\sigma_{x,V}} = \Omega(\mathcal{D}_V(x)). 
\end{align*}
Theorem~\ref{thm:sufficient_conditions} then shows that $f$ defines a factor map from $X_\sigma$ to $X_\tau$.

Now assume that $f$ defines a factor map from $X_\sigma$ to $X_\tau$.
Then by minimality of $X_\tau$ and $X_\sigma$, every prefix $u$ of $y$ belongs to $f(\cL(\sigma))$, satisfies $f(\mathcal{R}_{x,U}) \subset \mathcal{R}_{y,u}$ with $U = f^{-1}(\{u\}) \cap \cL(\sigma)$ and defines a factor map $\varphi_u$ from $\Omega(\mathcal{D}_U(x))$ to $\Omega(\mathcal{D}_u(y))$ (see Lemma~\ref{lemma:def of varphi_u}). 

Let $(u_n)_{n \geq 0}$ be the recursively enumerable sequence of prefixes of $y$ given by Lemma~\ref{lemma:u_n enumerable} and let $m,p$ be the constants given by that lemma.
We consider the finite subsequence $(v_n)_{0 \leq n \leq D \tilde{Q}_\sigma}$, where $v_n = u_{m+npk}$ for every $n$.
By choice of $k$, for every $n>0$, we have $K(K+1)|v_{n-1}| \leq |v_n| \leq C$.
For each $n$, we set $V_n = f^{-1}(\{v_n\}) \cap \mathcal{L}(\sigma)$.
Using Corollary~\ref{cor:nbr_de_return_subst} and Proposition~\ref{prop: return morphism to prefix code},
there exist $i,j$ such that $0 \leq i < j \leq D \tilde{Q}_\sigma$ and  $(\sigma_{x,V_i},\varphi_{v_i}) = (\sigma_{x,V_j},\varphi_{v_j})$, which ends the proof.
\end{proof}

Since all data in Theorem~\ref{theo:decidcoding} are computable from $\sigma$ and $\tau$, we obtain Theorem~\ref{theo:main2} as a corollary.

\subsection{Proof of Theorem~\ref{theo:cordecidfactor} and Corollary~\ref{cor:cordecidfactor}}

We deduce Theorem~\ref{theo:cordecidfactor} from Theorem~\ref{theo:durand13bbis}, Lemma~\ref{lemma:block representation isomorphic}, Proposition~\ref{prop:conj-sublengthn} and Theorem~\ref{theo:decidcoding}.

To prove Corollary~\ref{cor:cordecidfactor}, observe that one can decide whether there exist two factor maps $f : (X_\sigma , S ) \to (X_\tau , S)$ and $f : (X_\tau , S ) \to (X_\sigma , S)$.
And thus that $f\circ g : (X_\sigma , S ) \to (X_\sigma , S)$ is an endomorphism. 
But it is known from~\cite{Durand:2000} (see also~\cite{Coven:1971} for binary alphabets) that endomorphisms of minimal substitution subshifts are isomorphisms.
This is what is called the {\em coalescence property}.
This shows Corollary~\ref{cor:cordecidfactor}.

\section{The constant length case}\label{section:constantlength}

In this section we look in more details to the factor maps between constant length substitution subshifts and in particular to the maps $F_n$ defined by
\eqref{align:Fn}.
The main theorem in~\cite{Coven&Dekking&Keane:2017} shows that when there exists a factor map $f : (X_\sigma ,S) \to (X_\tau , S)$ between aperiodic minimal one-to-one constant length  substitution subshifts then there exists another factor map $g : (X_\sigma ,S) \to (X_\tau , S)$ given by a sliding block code of radius 1.
In~\cite{Coven&Quas&Yassawi:2016} the result is improved for automorphism $f : (X_\sigma ,S) \to (X_\sigma , S)$ with some extra assumptions. 
In this case $g$ can be chosen with memory $0$ and anticipation $1$.

In this section we improve these results by removing the assumption of injectivity and the additional assumptions in~\cite{Coven&Quas&Yassawi:2016} and by giving a link between $f$ and $g$. 
This is Theorem~\ref{theo:radius1}.

\subsection{Factor maps between constant length primitive substitution subshifts}

We say that a substitution $\sigma : A^* \to A^*$ is {\em one-to-one on letters} when $\sigma (a) \not = \sigma (b)$ whenever $a\not = b$.
When $\sigma $ has constant length, it is clear that if $\sigma : A^* \to A^*$ is one-to-one on letters then $\sigma : A^* \to A^*$ is one-to-one and $\sigma^n$ is one-to-one  for all $n\geq 1$.

\begin{theo}
\label{theo:radius1}
Let $\sigma $ and $\tau $ be two primitive substitutions of constant length.
Suppose there exists a factor map $f : (X_\sigma , S) \to (X_\tau , S)$ of radius $r$ with $(X_\tau , S)$ aperiodic.
Let $$N_n = \inf \{ k \geq 0 \mid S^k f (x) \hbox{ belongs to }  \tau^n (X_\tau )  ,  x \in \sigma^n (X_\sigma ) \},$$ and, set $L = 0 $ if $\tau$ is one-to-one on letters and $L=L_\tau$ otherwise. 

Then, there exists a factor map $F : (X_\sigma , S) \to (X_\tau , S)$ such that:
\begin{enumerate}
\item
If $r \geq  N_n $ for all large enough $n$, then the anticipation of $F$ is bounded by $L+1$ and its memory by $L +1$;
\item
If $r<N_n \leq |\tau^n |-r$ for all large enough $n$, then the anticipation of $F$ is bounded by $L$  and its memory by $L +1$;
\item
If $N_n \geq |\tau^n| -r$ for all large enough $n$, then the anticipation of $F$ is bounded by $L$  and its memory by $L +2$;
\item
$f=F\circ S^n$ for some $n \in \mathbb{Z}$.
\end{enumerate}
\end{theo}

Our proof follows the strategy used in~\cite{Salo&Torma:2015} to show that the automorphism group of subshifts generated by primitive one-to-one constant length substitutions is virtually $\mathbb{Z}$. 

We then use Theorem~\ref{theo:radius1} to show the decidability of the factorization between two minimal constant length substitution subshifts.

Observe that when dealing with one-to-one property it is important to distinguish $\sigma : A^* \to A^*$  and $\sigma : A^\mathbb{K} \to A^\mathbb{K}$ as there exists substitutions $\sigma $ such that $\sigma : A^* \to A^*$ is not one-to-one but $\sigma : X_\sigma \to X_\sigma$ is. 
For example not all primitive aperiodic substitutions $\sigma : A^* \to A^*$ are one-to-one but all of them are such that 
$\sigma : X_\sigma \to X_\sigma$ is one-to-one (Corollary~\ref{cor:mosse5.11}).

As taking powers of a substitution does not change the associated subshift, the next result implies that we can restrict ourselves to the case of substitutions of the same constant length.

\begin{prop}[{\cite[Théorème 15]{Fagnot:1997a}}]
\label{prop:fagnot}
Let $\sigma $ and $\tau $ be two primitive substitutions of constant length.
If $(X_\tau , S)$ is aperiodic and if there exists a factor map $f: (X_\sigma , S) \to (X_\tau , S)$, then there is exist $k,l\geq 1$ such that $|\sigma^k|=|\tau^l|$.
\end{prop}

\begin{lem}
\label{lemma:returntimes}
Let $\sigma $ and $\tau $ be two primitive substitutions of constant length $p$.
Suppose there exists a factor map $\phi : (X_\sigma , S )\to (X_\tau , S)$ with $(X_\tau , S)$ aperiodic.
Then, there exists a unique $i$ belonging to $[0, \dots , p-1 ]$ such that if $x$ belongs to $\sigma (X_\sigma )$ then 
$\phi (x)$ belongs to $S^{-i} (\tau (X_\tau))$.
\end{lem}

\begin{proof}
Let $x \in \sigma (X_\sigma )$.
From Corollary~\ref{lem:rec}, there exists a unique $i(x)$ belonging to $[0, \dots , p-1 ]$ such that $\phi (x)$ belongs to $S^{-i(x)} (\tau (X_\tau))$.
Let us show $i(x)$ is constant on $\sigma ( X_\sigma )$.
Observe that $S^{j}(x) $ belongs to $\sigma(X_\sigma)$ if and only if $j=k|\sigma|$ for some $k$ and, in this case, we have $\phi(S^{k|\sigma|}(x)) = S^{k|\sigma|} \phi(x) \in S^{-i(x)}(\tau(X_\tau))$.
Let $y\in \sigma (X_\sigma)$.
Using the minimality of $(X_\sigma , S)$, there exists a sequence $(S^{k_n |\sigma |} x)_n$ tending to $y$.
As $\phi (S^{k_n |\sigma |} x)$ belongs to $S^{-i(x)}\tau (X_\tau )$ for all $n$, we get that $\phi (y)$ also belongs to $S^{-i(x)}\tau (X_\tau)$.
\end{proof}

\begin{prop}
\label{prop:reducfact}
Let $\sigma $ and $\tau $ be two primitive substitutions of the same constant length and let $l\geq 0$.
Suppose there exists a factor map $f : (X_\sigma , S) \to (X_\tau , S)$ with radius $r$ and that $(X_\tau , S)$ is aperiodic.
Then, 

\begin{enumerate}
\item \label{enum:newsbc}
There exist $n$ and a factor map $F : (X_\sigma , S) \to (X_\tau , S)$ uniquely defined by 
\begin{align}
\label{def:sbcF}
\tau^l \circ F  = S^n\circ f\circ \sigma^l  , \ \  0 \leq  n<  |\tau|^l .
\end{align}
\item
\label{enum:radius}
The memory $t$ and the anticipation $s$  of $F$ satisfy 
\begin{align*}
t & \leq  \max ( 0 , L + \lceil (r-n)/|\tau|^l\rceil ), \\
s & \leq L + \lceil (r+n)/|\tau|^l\rceil,
\end{align*}
where $L = 0$ if $\tau $ is one-to-one on letters and $L = L_\tau$ otherwise. 
\end{enumerate}
\end{prop}

\begin{proof}
\eqref{enum:newsbc}
By Lemma~\ref{lemma:returntimes}, there exists a unique $n$ such that $0  \leq  n <  |\tau|^l $ and for all $x \in X_\sigma$, $S^n \circ f \circ \sigma^l (x)$ belongs to $\tau^l (X_\tau)$.
Thus, using Corollary~\ref{cor:mosse5.11}, Equation~\eqref{def:sbcF} uniquely defines a continuous map $F$.
Using again Corollary~\ref{cor:mosse5.11} and noticing that 
$$
\tau^l \circ F\circ S (x) = S^{n + |\sigma|^l}\circ f \circ \sigma^l (x)= S^{|\sigma|^l} \circ \tau^l \circ F (x) = \tau^l \circ S \circ F( x), 
$$
one obtains that $F$ is a factor map between $(X_\sigma , S)$ and $(X_\tau , S)$.

\medskip

\eqref{enum:radius}
We know from the Curtis-Hedlund-Lyndon theorem that $F$ is a factor map if and only if there exist $s,t\geq 0$ such that for all $x$, $F(x)_0$ only depends on $x_{[-t,s]}$.

Let $s,t\geq 0$ and $x\in X_\sigma$.
Then the word $\sigma^l (x)_{[-t|\sigma|^l , s|\sigma|^l]}$ only depends on $x_{[-t,s]}$.
Thus, it is also the case of the words $f\circ \sigma^l (x)_{[-t|\sigma|^l +r , s|\sigma|^l-r]}$ and $\tau^l \circ F (x)_{ [-t|\sigma|^l +r-n , s|\sigma|^l-r-n]}$.
Let $L_l = 0$ if $\tau $ is one-to-one on letters and $L_l = L_{\tau^l}$ otherwise. 
From Corollary~\ref{cor:mosse},
\[
F (x)_{ \left[ \left\lceil \frac{-t|\sigma|^l +r-n +L_l}{|\tau|^l}\right\rceil  , \left\lfloor\frac{ s|\sigma|^l-r-n - L_{l}}{|\tau|^l}\right\rfloor \right] } , 
\]
and thus, from Proposition~\ref{prop:constantrecsigmak},
\[
F (x)_{ \left[ \left\lceil \frac{-t|\sigma|^l +r-n +|\tau |^lL_1}{|\tau|^l}\right\rceil  , \left\lfloor\frac{ s|\sigma|^l-r-n - |\tau |^l L_{1}}{|\tau|^l}\right\rfloor \right] }  
\]
only depends on $x_{[-t,s]}$, that is, 
\[
F (x)_{ \left[ \lceil -t +L_1 +(r-n)/|\tau|^l\rceil  , \lfloor s -L_1 -(r+n)/ |\tau |^l \rfloor \right] }  
\]
only depends on $x[-t,s]$.
Hence it is sufficient that $s$ and $t$ satisfy  $ s- L_1 -(r+n)/|\tau|^l\geq 0$ and $-t +L_1 +(r-n)/|\tau|^l\leq 0$,  which ends the proof.
\end{proof}

Consider $\Delta (\sigma , \tau )$ the set of factor maps from  $(X_\sigma ,S)$ to $(X_\tau ,S)$.
From Proposition~\ref{prop:reducfact}  there is a map $\Phi_{\sigma , \tau }$ from $\Delta (\sigma , \tau )$ to itself defined by
\begin{equation}
\tau \circ \Phi_{\sigma , \tau } (f) = S^{N_{\sigma , \tau }(f)} \circ f \circ \sigma , \qquad \forall f \in \Delta (\sigma , \tau ) ,
\end{equation}
for some uniquely defined integer $N_{\sigma , \tau } (f)$ verifying $0 \leq   N_{\sigma , \tau } (f)<  |\tau|$. 
Observe $N_{\sigma , \tau } (f)$ is the unique integer $n$ such that $0 \leq n < |\tau|$ and $S^n \circ f \circ \sigma(X_\sigma) \subset \tau(X_\tau)$.

\begin{lem}
\label{lem:reducfact1}
Let $\sigma $ and $\tau $ be two primitive substitutions of the same constant length.
Suppose $(X_\tau , S)$ is non-periodic.
Let $g, h \in \Delta (\sigma , \tau )$. We have the following properties.
\begin{enumerate}
\item
If $h= S^n g$ with $N_{\sigma , \tau }(g) - |\tau |  < n \leq  N_{\sigma , \tau }(g) $ 
then $N_{\sigma , \tau }(S^n g)=N_{\sigma , \tau }(g) - n$ 
and $\Phi_{\sigma , \tau } (S^n g ) = \Phi_{\sigma , \tau } (g)$.
\item
If $\Phi_{\sigma , \tau } (g) = S^k \circ \Phi_{\sigma , \tau } (h)$ for some $k$ then $g= S^l \circ h$ for some $l$. 
\item
\label{item:compositionphi}
For all $m,n > 0$ one has 
\[
N_{\sigma^{m+n},\tau^{m+n}}(g) = |\tau^m| N_{\sigma^n , \tau^n }(  \Phi_{\sigma^m , \tau^m } (g) ) + N_{\sigma^m , \tau^m }(g)
\]
and 
$\Phi_{\sigma^n , \tau^n } \circ \Phi_{\sigma^m , \tau^m } (g) =  \Phi_{\sigma^{n+m} , \tau^{n +m}} (g)$.
\item
\label{enum:phin}
For all $n > 0$, $\Phi_{\sigma , \tau }^{n}  = \Phi_{\sigma^n , \tau^n }$. 
\end{enumerate}
\end{lem}

\begin{proof}
For the first assertion, we observe that for all $n$ such that $N_{\sigma , \tau }(g)-|\tau| < n \leq N_{\sigma , \tau }(g)$, we have $0 \leq N_{\sigma , \tau }(g) -n < |\tau|$ and
\[
	S^{N_{\sigma , \tau }(g)} \circ g \circ \sigma = S^{N_{\sigma , \tau }(g)-n} \circ S^n \circ g \circ \sigma.
\]
We thus have $N_{\sigma , \tau }(S^m g) = N_{\sigma , \tau }(g) -n$, which implies, using Corollary~\ref{cor:mosse5.11}, that $\Phi_{\sigma , \tau } (S^n g ) = \Phi_{\sigma , \tau } (g)$.

Let us now prove the second assertion and consider $g$ and $h$ in  $\Delta (\sigma , \tau )$ such that $\Phi_{\sigma , \tau } (g) = S^k \Phi_{\sigma , \tau } (h)$ for some $k$. 
Then, one gets
\begin{eqnarray*}
S^{N_{\sigma , \tau} (g)} \circ g \circ \sigma 
&=& \tau \circ \Phi_{\sigma , \tau } (g) \\
&=& \tau \circ S^k \circ \Phi_{\sigma , \tau } (h) \\ 
&=& S^{k|\tau |} \circ \tau  \circ \Phi_{\sigma , \tau } (h) \\
&=& S^{k|\tau |} \circ S^{N_{\sigma , \tau} (h)}  \circ h \circ \sigma.
\end{eqnarray*}

By minimality of $(X_\sigma , S)$ we deduce that $g = S^{k|\tau |+N_{\sigma , \tau} (h)-N_{\sigma , \tau} (g)} \circ h$.

\medskip

We now prove the third assertion. 
Let $m, n$ be two positive integers.
One has the following equalities:
\begin{eqnarray*}
\tau^m \circ \Phi_{\sigma^m , \tau^m } (g) 
&=& 
S^{N_{\sigma^m , \tau^m }(g)} \circ g \circ \sigma^m , \\
\tau^n \circ \Phi_{\sigma^n , \tau^n } \circ \Phi_{\sigma^m , \tau^m } (g) 
&=&
S^{N_{\sigma^n , \tau^n }(  \Phi_{\sigma^m , \tau^m } (g) )} \circ \Phi_{\sigma^m , \tau^m } (g)  \circ \sigma^n.
\end{eqnarray*}
Thus, we get
\begin{eqnarray*}
\tau^{n+m} \circ  \Phi_{\sigma^n , \tau^n } \circ \Phi_{\sigma^m , \tau^m } (g)
&=&
\tau^m \circ S^{N_{\sigma^n , \tau^n }(  \Phi_{\sigma^m , \tau^m } (g) )} \circ \Phi_{\sigma^m , \tau^m } (g)  \circ \sigma^n \\
&=&
S^{|\tau^m| N_{\sigma^n , \tau^n }(\Phi_{\sigma^m , \tau^m} (g))} \circ  \tau^m \circ \Phi_{\sigma^m , \tau^m } (g)  \circ \sigma^n \\
&=&
S^{|\tau^m |N_{\sigma^n , \tau^n }(  \Phi_{\sigma^m , \tau^m } (g) ) + N_{\sigma^m , \tau^m }(g)}\circ g  \circ \sigma^{n+m}.
\end{eqnarray*}
As 
\begin{eqnarray*}
0 \leq 
|\tau^m |N_{\sigma^n , \tau^n }(  \Phi_{\sigma^m , \tau^m } (g) ) + N_{\sigma^m , \tau^m }(g) 
& \leq & |\tau^m| (|\sigma^n|-1) + |\sigma^m|-1 \\
&\leq & |\sigma^{m+n}|-1, 
\end{eqnarray*}
we obtain, by definition, that 
\[
N_{\sigma^{m+n},\tau^{m+n}}(g) = |\tau^m |N_{\sigma^n , \tau^n }(  \Phi_{\sigma^m , \tau^m } (g) ) + N_{\sigma^m , \tau^m }(g),
\]
hence that $\Phi_{\sigma^n , \tau^n } \circ \Phi_{\sigma^m , \tau^m } (g) =  \Phi_{\sigma^{n+m} , \tau^{n +m}} (g)$.

For the \eqref{enum:phin}, the equality $\Phi_{\sigma , \tau }^{n}  = \Phi_{\sigma^n , \tau^n }$ is a direct consequence of \eqref{item:compositionphi}.
\end{proof}

\begin{proof}[Proof of Theorem~\ref{theo:radius1}]
Let $r$ be the radius of $f$.
From Proposition~\ref{prop:fagnot}, one can assume, taking powers of the substitutions if needed, that $|\sigma | = |\tau |$. 
We recall that from Lemma~\ref{lem:reducfact1}, $\Phi_{\sigma^n , \tau^n } (f) = \Phi^n_{\sigma , \tau } (f) $ is a factor map from $(X_\sigma , S)$ to $(X_\tau , S)$. 
We will find $F$ among these maps. 

Observe that the quantity $N_n$ defined in the statement of Theorem~\ref{theo:radius1} is exactly $N_{\sigma^{n} , \tau^{n}} (f)$.

Let us show that for some $n>0$ the memory and the anticipation of $\Phi^n_{\sigma , \tau } (f)$ are as required.
Indeed, let $n\in \mathbb{N}$.
Let $t$ and $s$ be, respectively, the memory and the anticipation of $\Phi^n_{\sigma , \tau } (f)$.
From Proposition~\ref{prop:reducfact} one can take  
\begin{align*}
t & = \max ( 0 ,   L+\lceil (r-N_{\sigma^{n} , \tau^{n}} (f))/|\tau^{n}|\rceil )\\
s & = L+\lceil (r+N_{\sigma^{n} , \tau^{n}} (f))/|\tau^{n}|\rceil.
\end{align*}
This shows the first three statements in the conclusion.

As there are finitely many such factor maps of radius at most $1$, there exist $p,q$, with $p<q$, such that $\Phi_{\sigma , \tau }^p (f) = \Phi_{\sigma , \tau }^q (f)$.
From Lemma~\ref{lem:reducfact1}, one gets $\Phi^{q-p}(f) = S^m \circ f$ for some $m$.
\end{proof}

An example in~\cite{Coven&Keane&Lemasurier:2008} studying the subshifts isomorphic to the Thue-Morse subshift shows that these bounds are sharp.
In particular, this implies that, considering the notations of Theorem~\ref{theo:radius1}, if $\tau$ is a one-to-one primitive constant length substitution, then $(X_\tau,S)$ can be a factor of $(X_\sigma,S)$ only if $\tau$ is defined on an alphabet of size at most $\#\cL_3(\sigma)$.
This result is not true anymore for non-constant length substitutions.
Indeed, using a slightly different definition of the substitution $\xi$ of Theorem~\ref{theo:durand13bbis}, we could build a primitive and non-constant length substitution $\xi'$ that is injective on the letters, defined on a 24-letter alphabet and such that $(X_{\xi'},S)$ is isomorphic to the Thue-Morse subshift (see also Section~\ref{subsection:example isomorphic}).

\subsection{Decidability of the factorization}

Assume that we are given two primitive substitutions of constant length $\sigma $ and $\tau $. By Theorem~\ref{theo:radius1}, to decide whether there exists a factor map between $(X_\sigma , S)$ and $(X_\tau , S)$, it suffices to test all sliding block codes of radius $L+1$. 
Thus, for each such sliding block code $f$, we have to decide whether $\mathcal{L} (f (X_\sigma ) ) $ is equal to $\mathcal{L} (X_\tau  )$. 
But it can be observed that both languages are languages of automatic sequences.
A result attributed to V.~Bruy\`ere in~\cite{Fagnot:1997a} allows to conclude as it states that the inclusion of automatic languages is decidable.
We give some details below.

\subsubsection{Background on automatic sequences}
For more details concerning this section, we refer to~\cite{Bruyere&Hansel&Michaux&Villemaire:1994} and~\cite{Rigo:2014b}.

A {\em $k$-automatic sequence} is a substitutive sequence w.r.t. some constant substitution $\sigma$ such that $k=|\sigma |$.

Let $k\geq 2$ be an integer and define the function $V_k : \mathbb{N} \to \mathbb{N}$ by $V_k (0)=1$ and 
$$
V_k (n) = k^p, \quad \text{where } p= \max \{ r \mid k^r \hbox{ divides } n \}, \ n\not = 0 .
$$

Let us consider the first order logical structure $\langle \mathbb{N} , + , V_k, = \rangle $.
A subset $E \subset \mathbb{N}$ is {\em $k$-definable} if there exists a formula in  $\langle \mathbb{N} , + , V_k, = \rangle $ describing $E$.
For example, the formula ``$V_2 (n) = n$'' describes the set $\{ 2^n \mid n\in \mathbb{N} \}$ as the power of $2$ are exactly the fixed points of $V_2$. 
Consider the formula $\phi (n)$ defined by $(\exists m)( n = m+m )$, then $\{ n \mid \phi (n) \}  = 2\mathbb{N}$. 
As $\phi$ is a formula expressed in $\langle \mathbb{N} , + , = \rangle$, this means that $2\mathbb{N}$ is $k$-definable for all $k\geq 2$.

A sequence $x\in A^\mathbb{N}$ is $k$-{\em definable} if for all $a\in A$ the set $\{ n \mid x_n = a \}$ is $k$-definable. 

\begin{theo}
[\cite{Buchi:1960,Cobham:1972}]
Let $k\geq 2$ and $A$ be a finite alphabet.
A sequence $x\in A^\mathbb{N}$ is $k$-automatic if and only if it is $k$-definable.
\end{theo}

\begin{theo}
\cite{Buchi:1960}
\label{theo:Buchi1960} 
The theory $\langle \mathbb{N} , + , V_k, = \rangle$ is decidable, that is,
for each closed formula expressed in the first order logical structure $\langle \mathbb{N} , + , V_k, = \rangle$, there is an algorithm deciding whether it is true or not.
\end{theo} 

We refer to~\cite{CharlierRampersadShallit2012,Charlier:2018} for expository texts on this decision procedure.
In particular, it has been implemented in the software {\em Walnut}~\cite{Mousavi2016} and can be downloaded from \url{www.cs.uwaterloo.ca/shallit/papers.html}.

As a consequence of Theorem~\ref{theo:Buchi1960} the following theorem is proved in~\cite{Fagnot:1997a} and attributed to Bruy\`ere.

\begin{theo}
\label{theo:decidlangauto}
Let $x$ and $y$ be two $k$-automatic sequences.
It is decidable whether $\mathcal{L} (x)$ is included in $\mathcal{L} (y)$. 
\end{theo}

\subsubsection{Decidability  of the  factorization for constant length primitive substitution subshifts}

Let us prove that the factorization is decidable for aperiodic subshifts generated by primitive constant length substitutions.
Let $\sigma $ and $\tau $ be two primitive substitutions of constant length. 
Let us describe the algorithm testing whether there exists a factor map $F$ from $(X_\sigma , S)$ to $(X_\tau , S)$.

\medskip

Algorithm: Test all sliding block codes of radius  $L+1$ using the algorithm given by Theorem~\ref{theo:decidlangauto} where $L$ is defined in Theorem~\ref{theo:radius1}.

\medskip

Let us give some explanations.

First we check whether $(X_\tau , S)$ is periodic or not as this is decidable~\cite{Honkala:1986,Honkala:2008,Allouche&Rampersad&Shallit:2009}.
If $(X_\tau , S )$ is periodic then one can decide whether there exists a factor map (Section~\ref{section:periodic}).
Suppose $(X_\tau , S)$ is non-periodic. 
One has first to test whether $\sigma $ and $\tau$ have some (non-trivial) power of their length that are equal.
This is clearly decidable. 
If they do not, then there is no factor map (Proposition~\ref{prop:fagnot}).
If they do, then Theorem~\ref{prop:reducfact} ensures that it suffices to consider sliding block codes of radius $L+1$.
There are finitely many such maps.
Thus it suffices to test each of them. 
Let $f$ be such a sliding block code of radius $L+1$. 
To test whether it defines a factor map from $(X_\sigma , S)$ to $(X_\tau , S)$ is equivalent to prove that $\mathcal{L} (f(x)) = \mathcal{L} (y)$ where $x$ and $y$ are respectively fixed points of $\sigma $ and $\tau$.
Due to Theorem~\ref{theo:decidlangauto} it suffices to prove that $f(x)$ is an automatic sequence. 
This is clear using the substitutions on the blocks of radius $L+1$ (see Section~\ref{subsec:sublengthn}).
Thus we proved the following theorem.

\begin{theo}
\label{theo:factorconstant}
Let $(X,S)$ and $(Y,S)$ be two minimal constant length substitution subshifts.
It is decidable to know whether there exists a factor map between these two systems.
\end{theo}

\subsection{How to get rid of the injective substitution assumption}\label{section:injectivisation}

The proof of Theorem~\ref{theo:factorconstant} involves sliding block codes of radius $L+1$ where $L=0$ if the substitutions are one-to-one.  
In this section, we show one can always suppose the substitutions are one-to-one, up to isomorphism.
This will reduce the number of sliding block codes to test and the size of the involved alphabets.
Hence, this improves the algorithm to decide the factorization.

Let $\sigma : A^* \to A^*$ be a primitive substitution such that $(X_\sigma , S)$ is not periodic.
Let $B\subset A$ and $\phi : A^* \to B^*$ be a coding such that: 

\begin{align}
\label{align:injectivization}
\hbox{ if } \phi (a) = \phi (b) \hbox{ then } \sigma (a) = \sigma (b).
\end{align}

Then, there exists a unique endomorphism $\tau : B^* \to B^*$ defined by
$$
\tau \circ \phi = \phi \circ \sigma .
$$

It is clear that $\tau$ is primitive.
We say $\tau$ is an {\em injectivization} of $\sigma$.

\begin{prop}
\cite{Blanchard&Durand&Maass:2004}
Let $\tau $ be an injectivization of the non-periodic primitive endomorphism $\sigma $.
Then, $(X_\sigma , S)$ is conjugate to $(X_\tau , S)$.
Moreover $\tau $ can be chosen injective and such a $\tau $ can be algorithmically found.
\end{prop}

\begin{proof}
The first claim is proved in~\cite{Blanchard&Durand&Maass:2004}. 
The two other claims are straightforward but we provide a proof below.

Suppose there exist $a$ and $b$ such that $\sigma (a) = \sigma (b)$. 
Let $B = A \setminus \{ b\}$.
Consider the map $\phi: A^* \to B^*$ defined by $\phi (c) = c$ if $c\not = b$ and $\phi_1 (b)=a$.
It satisfies \eqref{align:injectivization} and thus defines an injectivization $\zeta$ of $\sigma$.
Observe that the cardinality of the alphabet decreases.
If $\zeta $ is not injective, we proceed in the same way to obtain a new substitution $\tau$ that will be an injectivization of $\zeta$ and of $\sigma$ and that will be defined on a smaller alphabet. 
As $A$ is finite, iterating the process, we will obtain an injectivization $\tau $ of $\sigma$ that is injective.  
This achieves the proof.
\end{proof}

\subsection{Listing of the factors}\label{section:factorlist}
Let $\sigma $ be a primitive constant length substitution.
We know from~\cite{Durand:2000} that the number of  aperiodic subshift factors of $(X_\sigma , S)$ is finite and all such factors are conjugate to a substitution subshift (Theorem~\ref{theo:durand13bbis}).
We are not able to establish the list of its subshift factors but what precedes allows to find the exhaustive list of its constant length substitution subshift factors.
Indeed, periodic factors can be easily found with the main result in~\cite{Dekking:1978}, see Section~\ref{section:periodic}.
For the aperiodic constant length substitution subshift factors, Section~\ref{section:injectivisation} allows us to suppose the substitutions are one-to-one on letters.
Then,
Theorem~\ref{theo:radius1} indicates it suffices to test finitely many sliding block codes. 
Theorem~\ref{theo:decidlangauto} can decide whether or not they are conjugate one to the other. 
This gives the exhaustive list of aperiodic constant length substitution subshift factors. 

In~\cite{Coven&Dekking&Keane:2017} another algorithm is given but with more assumptions on the subshifts.

There are examples of primitive constant length substitution subshifts having factors that are aperiodic non-constant length substitution subshifts. Consider the subshift $(X,S)$ generated by the primitive substitution $A\mapsto ABCAA$, $B\mapsto AAAAA$ and $C\mapsto AABCA$.
It has as a factor the non-constant length substitution subshift $(Y,S)$ generated by $A\mapsto ABCAA$, $B\mapsto AAAAAA$ and $C\mapsto ABCA$. Indeed, one has $X=Y$.

But it is now known \cite[Theorem 22]{Mullner&Yassawi:2020} that factors of primitive constant length substitution subshifts are isomorphic to constant length substitution subshifts. 
This answers a question we asked in a previous version of this work. 

As a consequence, the list obtained above corresponds to the list of all subshift factors.

\section{Factor maps between linearly recurrent subshifts}\label{section:LR}
In this section we prove Theorem~\ref{theo:factorLR}.

Given a dynamical system $(X,T)$, {\rm Aut}$(X,T)$ will stand for the group of automorphisms of $(X,T)$, that is the group of continuous and bijective maps defined on $X$ and commuting with $T$.
In~\cite{Donoso&Durand&Maass&Petite:2016} (see also~\cite{Cyr&Kra:2016}) has been proved the following result.

\begin{theo} \label{Thm1}
Let $(X,S)$ be an aperiodic minimal subshift. 
If 
\[ 
\liminf_{n\in \N} \frac{p_X(n)}{n} < + \infty ,
\] 
then {\rm Aut}$(X,S)$ is virtually $\mathbb{Z}$, {\em i.e.}, the quotient of ${\rm Aut}(X,S)$ by the group generated by the shift map is a finite group.
\end{theo}

The previous result means that there exist finitely many elements $f_1, \dots , f_n$ of  ${\rm Aut}(X,S)$ such that for any $f \in  {\rm Aut}(X,S)$ there exist $m \in \mathbb{Z}$ and $i \in \{1,\dots,n\}$ satisfying $f = f_i \circ S^m$.
In this section we extend this result to factor maps but in the restrictive context of the linearly recurrent subshifts (which are known to have linear complexity).
Observe that for factor maps we loose the group structure which is a key argument in both papers~\cite{Donoso&Durand&Maass&Petite:2016} and~\cite{Cyr&Kra:2016}.



Theorem~\ref{theo:factorLR} includes answers to questions stated in~\cite{Salo&Torma:2015} but in the more general context of factor maps.
The proof is inspired by the one of Theorem 1 in~\cite{Durand:2000} and uses the following result.

\begin{prop}
\label{prop:maxmin}
Let $(X , S)$ be an aperiodic linearly recurrent subshift for the constant $K$ and $(Y , S)$ be one of its aperiodic factors given by a factor map of radius $r$. 
We have the following.
\begin{enumerate}
\item
\label{item:complexity}
For all $n$, $p_{Y}(n) \leq K(n+2r)$.
\item
\label{item:occtau}
For all $n$, every word of $\mathcal{L}_n (Y)$ occurs in every word of $\mathcal{L}_{(K  +1)n+2r K} (Y)$.
\item
\label{item:rltau}
For all $v\in \mathcal{L} (Y )$ and all $w\in \mathcal{R}_{Y ,v }$, 
\[
	\frac{|v|-2r K}{K} \leq |w| \leq K(|v| +2r).
\]
\item
\label{item:card}
For all $v\in \mathcal{L} (Y )$ with $|v| \geq 2rK+1$ one has 
$$
\# (\mathcal{R}_{Y,v})\leq \frac{(K+1)K^2 (|v| + 2r)+2rK^2}{|v| -2rK} .
$$
\end{enumerate}
\end{prop}

\begin{proof}
Let $\hat{f}$ be a block map with radius $r$ defining a factor map from $(X,S)$ to $(Y,S)$.

\eqref{item:complexity}
We have $\mathcal{L}_n(Y) = \hat{f}(\mathcal{L}_{n+2r}(X))$ so the result follows from Theorem~\ref{theo:encad}.

\eqref{item:occtau}
Let $u \in \mathcal{L}_n(Y)$ and $v \in \mathcal{L}_{(K  +1)n+ 2r K}(Y)$. 
There exist $u' \in \mathcal{L}_{n+2r}(X)$ and $v' \in \mathcal{L}_{(K+1)(n+2r)}(X)$ such that $u = \hat{f}(u')$ and $v = \hat{f}(v')$.
By Theorem~\ref{theo:encad}, $u'$ occurs in $v'$.
Thus $u$ occurs in $v$.

\eqref{item:rltau}
Let $v \in \mathcal{L}(Y)$.
For any return word $w \in \mathcal{R}_{Y,v}$, there exist $v_1,v_2 \in \hat{f}^{-1}(\{ v\} )$ and $w' \in \mathcal{L}(X)$ such that $v_1$ is a prefix of $w'v_2$, $w'v_2$ contains only two occurrences of words in $\hat{f}^{-1}(\{ v\} )$ and $\hat{f}(w'v_2) = wv$.
In particular $w'$ is a subword of a return word in $\mathcal{R}_{X,v_1}$ and, using Theorem~\ref{theo:encad}, we have $|w| = |w'| \leq K(|v|+2r)$.

Let us now show by contradiction that $\frac{|v|-2r K}{K} \leq |w|$.
If $\frac{|v|-2r K}{K} > |w|$, then, by Remark~\ref{remark:power return word}, the word $u = w^{1+ |v|/|w|}$ belongs to $\mathcal{L} (Y )$ and this word has length $|u| = |w|+|v| > (K+1) |w|+2r K$.
By Item~\eqref{item:occtau}, $u$ contains an occurrence of all words in $\mathcal{L}_{|w|}(Y)$.
As the number of different words of length $|w|$ occurring in $u$ is at most $|w|$, this implies by Morse and Hedlund's theorem~\cite{Morse&Hedlund:1940} that $Y$ is periodic, which contradicts the hypothesis.

\eqref{item:card}
Let $v\in \mathcal{L} (Y )$ with $|v| \geq 2rK+1$.
From \eqref{item:rltau}, the length of return words to $v$ are less than $K(|v| +2r)$.
Let $w$ be a word of length $(K  +1)K(|v| +2r)+2r K$.
From \eqref{item:occtau} all return words to $v$ appear in $w$.
Hence, from \eqref{item:rltau} the number of return words to $v$ is bounded by 
\[
\frac{(K  +1)K(|v| +2r)+2r K}{\frac{|v|-2r K}{K}} .
\]
\end{proof}

We are now ready to Prove Theorem~\ref{theo:factorLR}.

\begin{proof}[Proof of Theorem~\ref{theo:factorLR}]
Let $A$ and $B$ be the alphabets of $X$ and $Y$ respectively.
Consider the constant $N = 1+(2(K+1))^{36 (K+1)^5}$ and suppose that for all $i \in \{1,\dots,N\}$, $f_i : (X , S) \to (Y , S)$ is a factor map.
It suffices to show that there exist $i \neq j$ and $k \in \mathbb{Z}$ such that $f_i = S^k \circ f_j$.

By the Curtis-Hedlund-Lyndon Theorem there exists a positive integer $r$ such that for all $i \in \{1,\cdots ,N \}$ there is a sliding block code $\phi_i : A^{2r+1} \rightarrow B$ satisfying 
\begin{equation}
\label{eq:sbc}
	(f_i (x))_j = \phi_i (x_{[ j-r,j+r]}) \quad \forall x \in X , \  \forall j \in \mathbb{Z} .
\end{equation}

Let us fix some $x \in X$ and some $y \in Y$.
By Proposition~\ref{prop:maxmin}, all return words to $y_{[0,2r)}$ have length at most $4rK$.
Thus for all $i \in \{1,\dots, N\}$, there exists some integer $k_i$ satisfying $0 \leq k_i \leq 4rK$ and such that $(S^{k_i}f_i(x))_{[0,2r)} = y_{[0,2r)}$.

We then define the factor map $\tilde{f}_i: (X,S) \to (Y,S)$ by $\tilde{f_i}(x') = S^{k_i} f_i(x')$ for all $x' \in X$.
Thus $\tilde{f}_i$ has a radius of $4rK+r$: there exists a sliding block code $\tilde{\phi}_i: A^{2r(4K+1)+1} \to B$ such that
\[
	(\tilde{f}_i (x'))_j = \tilde{\phi}_i (x'_{[ j-r(4K+1),j+r(4K+1)]}) \quad \forall x' \in X , \  \forall j \in \mathbb{Z}.
\]
Furthermore, by choice of $k_i$ we have for all $i$ 
\begin{equation}
\label{eq:utov}
	\tilde{\phi}_i(x_{[ -r(4K+1),r(4K+1)+2r)}) = y_{[0,2r)}.
\end{equation}
We will now show that there exist $i$ and $j$, $i \neq j$, such that $\tilde{f}_i = \tilde{f}_j$, hence $f_i = S^{k_j-k_i} f_j$.

We fix $r' = r(4K+1)$ 
and we consider the block representation $X^{(r')}$ of $X$ which is isomorphic to $X$ through the isomorphism $h$ associated with the sliding block code $\phi: A^{2r'+1} \to A^{(r')}, a_1 \cdots a_{2r'+1} \mapsto (a_1 \cdots a_{2r'+1})$ (see Lemma~\ref{lemma:block representation isomorphic}).

Observe that, from Theorem~\ref{theo:encad}, for all $u \in \mathcal{L}(X^{(r')})$ we have:

\begin{enumerate}
\item
\label{item1}
for all words $v \in \mathcal{R}_{X^{(r')},u}$, $\frac{1}{K}(|u|+2r') \leq |v| \leq K (|u|+2r')$, and,
\item 
$\# (\mathcal{R}_{X^{(r')},u})\leq (K+1)^3$.
\end{enumerate}

For all $i \in \{1,\cdots , N \}$ let $\psi_i : A^{(r')*} \rightarrow B^*$ be the morphism defined by, for $(u) \in A^{(r')}$, $\psi_i ((u)) = \tilde{\phi}_i(u)$.
It defines a factor map $g_i : (X^{(r')},S) \rightarrow (Y,S)$.
Since for all $i$, $\tilde{f}_i = g_i \circ h$, it suffices to prove that there exist $i$ and $j$, $i \neq j$, such that $g_i = g_j$. 
By minimality, it is enough to show that $g_i (x^{(r')})= g_j(x^{(r')})$ for some $i \neq j$.
Indeed, for all $z^{(r')} \in X^{(r')}$, there is an increasing sequence of integers $(n_\ell)_{\ell \in \N}$ such that $z^{(r')} = \lim_{\ell \to +\infty} S^{n_\ell} x^{(r')}$. 
If $g_i (x^{(r')})= g_j(x^{(r')})$, we get by continuity that 
\begin{align*}
g_i(z^{(r')}) &= \lim_{\ell \to +\infty} g_i (S^{n_\ell}x^{(r')}) = \lim_{\ell \to +\infty} S^{n_\ell} g_i(x^{(r')})		\\
&= \lim_{\ell \to +\infty} S^{n_\ell} g_j(x^{(r')}) = g_j(z^{(r')}).
\end{align*}

Take $u = (x^{(r')})_{[0,r')}$ and $v = y_{[0,r')}$.
From~\eqref{eq:utov} we have $\psi_i(u)=v$ for all $i$.
Hence for each return word $w \in \mathcal{R}_{x^{(r')},u}$, the word $\psi_i (w)$ is a concatenation of return words in $\mathcal{R}_{y,v}$. 
This induces a non-erasing morphism $\lambda_i$ from $R_{x^{(r')}, u}^*$ to $R_{y,v}^*$ defined by 
\[
	\psi_i \circ \Theta_{x^{(r')},u} = \Theta_{y,v} \circ \lambda_i.
\]
This morphism naturally induces a map from $R_{x^{(r')}, u}^{\mathbb{Z}}$ to $R_{y,v}^{\mathbb{Z}}$ that we also denote by $\lambda_i$ and that satisfies 
\[
	g_i \circ \Theta_{x^{(r')}, u} = \Theta_{y,v} \circ \lambda_i.
\]
As $\psi_i$ is a coding and $|v|=r'$, for all $b \in R_{x^{(r')}, u}$ we have from \eqref{item1} above and \eqref{item:rltau} of Proposition~\ref{prop:maxmin}
\[
	|\lambda_i (b)| \leq \frac{K(|u|+2r')}{\frac{|v|-2r K}{K} } = \frac{3K^2r'}{r'-2r K } =  \frac{3K^2(4K+1)r}{(2K+1)r } < 6K^2.
\]
Since,  from \eqref{item:card} of Proposition~\ref{prop:maxmin}, we also have 
\begin{align*}
\# (\mathcal{R}_{Y,v}) 
& \leq 
\frac{(K+1)K^2 (r(4K+1) + 2r)+2rK^2}{2rK+1} \\
& \leq
(K+1)K^2 \frac{4K+3}{2K}+K
<
2 (K+1)^3 ,
\end{align*}
we get $\# \{ \lambda_i \mid 1\leq i \leq N \} < N$ which implies that there exist $i,j$, $i \neq j$, such that $\lambda_i = \lambda_j$. 
We set $\lambda = \lambda_i$.

Since $u$ is a prefix of $x^{(r')}_{[0,+\infty)}$, from Proposition~\ref{prop:def suite derivee singleton} there exists a unique sequence $z \in R_{x^{(r')} ,u}^{\mathbb{Z}}$ such that $\Theta_{x^{(r')}, u} (z) = x^{(r')}$.
We get
\[
g_i (x^{(r')})  
=  g_i \circ \Theta_{x^{(r')} ,u} (z) 			
=  \Theta_{y,v} \circ \lambda (z) 				
=  g_j \circ \Theta_{x^{(r')} ,u} (z)
=  g_j(x^{(r')}),
\]
which ends the proof.
\end{proof}



\section{Other decidability problems for morphic subshifts}\label{section:openquestions}

In this section we mention some open problems of decidability concerning morphic subshifts. 
The inputs will always be the morphism $\phi$ and the endomorphism $\sigma$ defining  the morphic sequence or pairs of such couples $(\phi , \sigma )$.

\subsection{Factor maps for subshifts generated by automatic sequences}
Fixed points of constant length substitutions are a special kind of automatic sequences. 
Even if Theorem~\ref{theo:main} applies to minimal automatic subshifts, to have a better radius for factor maps
it would be interesting to know whether Theorem~\ref{theo:radius1} remains valid for minimal automatic subshifts. 
This would make more efficient the algorithm to find its automatic subshift factors. 

\subsection{The non-minimal case}
We recall that primitive substitutions, as well as uniformly recurrent sequences, generate minimal subshifts.
Is it possible to extend the main results stated in the introduction (Theorem~\ref{theo:cordecidfactor}, Corollary~\ref{cor:cordecidfactor}, Theorem~\ref{theo:main} and Theorem~\ref{theo:main2}) to (non-uniformly recurrent) morphic sequences?

Let us focus on Theorem~\ref{theo:main2}.
Extension to the non-minimal case, in this case, is equivalent to the decidability problem of the equality problem of morphic languages.

\medskip

{\bf  Equality problem of morphic languages.}
Given two morphic sequences $x$ and $y$, is it decidable whether $\mathcal{L} (x) = \mathcal{L} (y)$?

\medskip

Indeed, suppose $(X,S)$ and $(Y,S)$ are two morphic subshifts generated by, respectively, the morphic sequences $x$ and $y$, and let $\phi : A^\Z \to B^\Z$ be a factor map.
It is easy to show that $\phi $ is a factor map from $(X,S)$ to $(Y,S)$ if and only if $\mathcal{L} (\phi (x))$ is equal to $\mathcal{L} (y)$.

If we do not ask the factor map $\phi : (X,S) \to (Y,S)$ to be onto, then it is equivalent to the inclusion problem of morphic languages. 

\medskip

{\bf  Inclusion problem of morphic languages.}
Given two morphic sequences $x$ and $y$, is it decidable whether $\mathcal{L} (x) $ is included in  $\mathcal{L} (y)$?

\medskip

Of course the Inclusion problem implies the Equality problem, both problems being open.

\subsection{Morphic languages}

Let us consider the Inclusion and Equality problems.

In~\cite{Fagnot:1997b}, positive answers are given to the Inclusion problem for some families of purely morphic sequences, one for those defined on a 2-letter alphabet and another one for those generated by primitive substitutions under some mild assumptions.
Observe that Theorem~\ref{theo:main2} solves this problem for uniformly recurrent morphic sequences as inclusion of languages of uniformly recurrent sequences implies the equality.

\begin{theo}
Inclusion problem for morphic languages is decidable for uniformly recurrent morphic sequences.
\end{theo} 
\begin{proof}
It is left as an exercise.
\end{proof}

Both problems are open if we do not assume uniform recurrence. 
Nevertheless, Fagnot~\cite{Fagnot:1997a} solved both decidability problem for automatic sequences, without the assumption of uniform recurrence. 
She considered two cases: the multiplicative dependence and independence of the dominant eigenvalues $p$ and $q$ (which are integers in these settings) of the underlying substitutions, that is whether $\log p / \log q $ belongs or not to $\mathbb{Q}$.
She managed to treat the multiplicatively independent case using the extension of Cobham's theorem for automatic languages she obtained in the same paper, then testing the ultimate periodicity of the languages to see whether they are equal. 
For the dependent case she used the logical framework of Presburger arithmetic associated with automatic sequences.

\begin{theo}[\cite{Fagnot:1997a}]
Inclusion problem for morphic languages is decidable for automatic sequences.
\end{theo}

\subsection{Equality problem of morphic sequences}
Given two morphic sequences $x$ and $y$, is it decidable whether $x = y$?

For purely morphic sequences, the Equality problem of morphic sequences is called the {\em D$0$L $\omega$-equivalence problem} and was solved in 1984 by K.~Culik~II and T.~Harju~\cite{Culik&Harju:1984}.
In~\cite{Durand:2012}, it was shown to be true for primitive morphic sequences and that more can be said for the D$0$L $\omega$-equivalence problem in the primitive case: the two (underlying) primitive substitutions have some non-trivial powers that should coincide on some cylinders if $x=y$. 

The general case of the Equality problem remains open.

\subsection{The listing of the factors}
As we mentioned in Section~\ref{section:factorlist} for constant length substitution subshifts it would be of interest to find an algorithm that, given a minimal  morphic subshift, provides the list of all its (finitely many) aperiodic subshift factors. 
More precisely, to consider the following problem.

\medskip

{\bf  Computability of the set of aperiodic factor subshifts of a minimal morphic subshift.} 
Given a minimal morphic subshift $(X,S)$, can we compute a finite set of primitive substitutions $\{ \sigma_1 , \dots , \sigma_N\}$ such that any aperiodic factor subshift of $(X,S)$ is isomorphic to some $(X_{\sigma_n} , S)$?

\medskip

We have seen in Section \ref{section:factorlist} that the answer is positive when $(X,S)$ is a primitive constant length substitution subshift.

In the general case, from~\cite{Durand:2000} and Theorem~\ref{theo:durand13bbis} we know that such a set of substitutions exists but we do not have an algorithm to find them.
We leave this as an open question.

\subsection{Orbit equivalence}
A factor map $\phi : (X,S) \to (Y,T)$ between minimal dynamical  has the property that orbits are sent to orbits:
\begin{align}
\label{align:orbittoorbit}
\phi (\{ S^n (x) \mid n\in \mathbb{Z} \}) = \{ S^n (\phi (x)) \mid n\in \mathbb{Z} \})
\end{align} 
for all $x\in X$.
When considering isomorphisms, this defines an equivalence relation. 
One can relax the conditions and consider $\phi$ is just an homeomorphism satisfying~\eqref{align:orbittoorbit}.
In this case we say $(X,S)$ and $(Y,T)$ are {\it orbit equivalent}. 
It is classical to observe that in this situation there exists $N : X \to \Z$ such that 

\begin{align}
\phi (S (x)) = S^{N(x)} \phi (x) .
\end{align}

When $N$ has at most one point of discontinuity, we say  $(X,S)$ and $(Y,T)$ are {\it strong orbit equivalent}.

A remarkable result in~\cite{Giordano&Putnam&Skau:1995}
characterizes strong orbit equivalence of minimal Cantor systems by means of dimension groups.
In~\cite{Bratteli&Jorgensen&Kim&Roush:2001} it is shown that the isomorphism of stationary simple dimension groups is decidable. 
This shows,  as dimension groups of minimal substitution subshifts are simple and stationary~\cite{Forrest:1997,Durand&Host&Skau:1999} that the strong orbit equivalence is decidable for minimal substitution subshifts. 
However, this result does not provide a new proof of Corollary~\ref{cor:cordecidfactor} since for a given minimal substitution subshift $(X,S)$, there might be other minimal substitution subshifts that are not isomorphic to $(X,S)$, but that are in the same strong orbit equivalence class of $(X,S)$~\cite{Werner:2009}.
Thus, decidability of the strong orbit equivalence for minimal substitution subshifts is not sufficient to decide the isomorphism between minimal substitution subshifts.

\medskip

For the orbit equivalence we address the following question: Is orbit equivalence decidable for minimal substitution subshifts?

\subsection{Beyond substitution and linearly recurrent subshifts}

Considering Theorem~\ref{Thm1} and Theorem~\ref{theo:factorLR} we wonder whether the conclusion of the second is true under the assumptions of the first. 
We leave this as an open question. 
We know that assuming zero entropy is too weak to have such a conclusion~\cite{Donoso&Durand&Maass&Petite:2016}. 
We have no counter example for finite topological rank subshifts (see~\cite{Bressaud&Durand&Maass:2010} for the definition).

\subsection{For tilings}

All these questions could also be asked for self-similar tilings or multidimensional substitutions.
The more tractable problems would certainly concern ``square multidimensional substitutions'' (see, for example,~\cite{Cerny&Gruska:1986,Salon:1986,Salon:1987,Salon:1989} or, in an interesting logical context, the nice survey~\cite{Bruyere&Hansel&Michaux&Villemaire:1994}). 
Moreover, the striking papers~\cite{Mozes:1989} and~\cite{GoodmanStrauss:1998} should also be mentioned as they show that in higher dimension the situation is not as ``simple'' as it is in dimension 1.

\newpage 




\section*{Acknowledgments} 
The authors would like to thank Samuel Petite for many interesting discussions.
This research is supported by the ANR project\footnote{Ref. ANR-13-BS02-0003} ``Dyna3S'' as well as by a Hubert Curien Partnership\footnote{Ref. SCO/AD/FR/SOR/2015/2019880} (PHC) ``Tournesol''.

\bibliographystyle{amsplain}


\begin{dajauthors}
\begin{authorinfo}[fabd]
  Fabien Durand\\
  Laboratoire Ami\'enois de Math\'ematiques Fondamentales et Appliqu\'ees, CNRS-UMR 7352\\ 
  Universit\'{e} de Picardie Jules Verne\\
  Amiens, France\\
  \url{https://www.lamfa.u-picardie.fr/fdurand/}\\
  fabien\imagedot{}durand\imageat{}u-picardie\imagedot{}fr\\

\end{authorinfo}
\begin{authorinfo}[jull]
  Julien Leroy\\
  D\'epartement de Math\'ematique\\ 
  Universit\'{e} de Li\`ege\\
  Li\`ege, Belgique\\
  \url{http://www.discmath.ulg.ac.be/leroy/}\\
    j\imagedot{}leroy\imageat{}uliege\imagedot{}be\\

\end{authorinfo}
\end{dajauthors}

\end{document}